\pdfoutput=1
\documentclass[11pt]{article}

\usepackage[a4paper,margin=28mm]{geometry}

\usepackage[T1]{fontenc}
\usepackage[utf8]{inputenc}
\usepackage{lmodern}
\usepackage{amsmath,amssymb,amsthm,mathtools}

\AtBeginDocument{%
  \setlength{\abovedisplayskip}{6pt plus 2pt minus 3pt}%
  \setlength{\belowdisplayskip}{6pt plus 2pt minus 3pt}%
  \setlength{\abovedisplayshortskip}{3pt plus 2pt minus 2pt}%
  \setlength{\belowdisplayshortskip}{5pt plus 2pt minus 3pt}%
  \setlength{\jot}{3pt}%
}
\numberwithin{equation}{section}
\usepackage{bm}
\usepackage{mathrsfs}
\usepackage{comment}

\usepackage{graphicx}
\usepackage{tikz}
\usetikzlibrary{arrows.meta,calc}
\usepackage{booktabs}
\usepackage{array}

\usepackage[colorlinks=true,linkcolor=blue,citecolor=blue,urlcolor=blue]{hyperref}
\hypersetup{
  pdftitle={Anisotropic Lagrange interpolation on exact spherical triangles},
  pdfauthor={Hiroki Ishizaka},
  pdfsubject={Anisotropic finite element interpolation on exact spherical triangles},
  pdfkeywords={anisotropic interpolation, exact spherical triangle, semi-regular mesh, maximum-angle condition, radial projection}
}
\usepackage[nameinlink,capitalize]{cleveref}

\theoremstyle{plain}
\newtheorem{theorem}{Theorem}[section]
\newtheorem{lemma}[theorem]{Lemma}
\newtheorem{proposition}[theorem]{Proposition}
\newtheorem{corollary}[theorem]{Corollary}

\theoremstyle{definition}
\newtheorem{definition}[theorem]{Definition}

\theoremstyle{remark}
\newtheorem{remark}[theorem]{Remark}

\newcommand{\R}{\mathbb{R}}

\newcommand{\diff}{\mathop{}\!\mathrm{d}}

\newcommand{\Sph}{\mathbb{S}}

\newcommand{\aff}{\operatorname{aff}}
\newcommand{\dist}{\operatorname{dist}}
\newcommand{\conv}{\operatorname{conv}}

\theoremstyle{definition}
\newtheorem{convention}[theorem]{Convention}

\title{Anisotropic Lagrange interpolation on exact spherical triangles}

\author{
  Hiroki Ishizaka\\
  Team FEM, Matsuyama, Japan\\
  E-mail: \texttt{h.ishizaka005@gmail.com}
}

\date{}

\begin{document}

\maketitle

\begin{abstract}
We study linear Lagrange interpolation on exact geodesic spherical triangles obtained by radial projection of chordal triangles. Factorising the element map into an affine map and a radial correction separates chordal anisotropy from radial distortion. We derive exact identities for the surface measure, Dirichlet energy, and directional derivatives, and use them to transfer planar anisotropic interpolation estimates to the sphere. The resulting local error estimates retain the two chordal directional length scales and keep the geometric factors explicit, with constants independent of the element size, shape, and sphere radius. In the gradient estimate, the maximum-angle factor multiplies only the longer directional scale. Four explicit triangle families quantify the loss caused by replacing these scales with the diameter and establish the sharp maximum-angle dependence. They also show that, within the same two-term directional bound with a constant independent of the element, the maximum-angle factor cannot be moved to the shorter scale or have its exponent reduced below one. The fourth family proves the optimality of the square-root radial-transfer factor while retaining the same transported directional scale. The local estimates extend to compatible spherical triangulations. As an application, they yield an anisotropic energy-error estimate for the conforming approximation of the mean-zero Laplace--Beltrami problem with exact integration.
\end{abstract}



\section{Introduction}
\label{sec:introduction}
Interpolation error analysis is an essential theme in finite element analysis. It is well known that the interpolation error of piecewise polynomials depends not only on the size of the elements but also on their geometry. For linear interpolation on triangles, the maximum-angle condition goes back to Synge \cite{Synge1957}, and it was developed in the context of finite element analysis by Babu\v{s}ka and Aziz \cite{BabuskaAziz1976} and Jamet \cite{Jamet1976}. This condition only requires that no interior angle approaches $\pi$, and it is weaker than the minimum-angle condition or the shape-regularity condition. Following K\v{r}\'{\i}\v{z}ek \cite{Krizek1991}, we call such families of meshes semi-regular. A semi-regular family of meshes may contain elements with arbitrarily large aspect ratios. For such meshes, estimates that take the directional length scales of each element into account give more precise information than estimates written only in terms of the element diameter.

Anisotropic interpolation error estimates for affine Lagrange elements have been well studied; see, for example, \cite{Apel1999,ApelDobrowolski1992,IshizakaKobayashiTsuchiya2021,IshizakaKobayashiTsuchiya2023}. For curved elements, however, the situation is more complicated, because the parametrisation of the element changes the surface measure, the $H^1$ seminorm, and higher-order derivatives. In the classical theory of curved elements, these effects are controlled by regularity assumptions on the element maps \cite{Ciarlet1978,CiarletRaviart1972,Bernardi1989}. In surface finite element analysis, functions on the discrete and exact surfaces are usually related by lifts or closest-point maps \cite{Dziuk1988,Demlow2009,DziukElliott2013}. Finite element methods on spherical geodesic grids have a long history \cite{Giraldo1997}. In particular, Di et al. \cite{DiLiTangZhang2006} constructed continuous linear Lagrange elements by radially lifting the basis functions of the chordal triangles. Apel and Pester \cite{ApelPester2005} derived weighted error estimates for a Cl\'ement-type interpolant on spherical domains, using isotropic triangulations in spherical coordinates. Anisotropic estimates under maximum-angle-type conditions have also been studied for piecewise planar surface meshes induced by level-set functions \cite{OlshanskiiReuskenXu2012,Demlow2023Mixed}.

For the Poisson problem on two-dimensional curved domains, Ishizaka \cite{Ishizaka2026ExactCurved} developed an exact-curved Lagrange framework in which the element map is factorised into an affine map and a curved correction. This factorisation separates affine anisotropy from curvature effects and provides a framework for transferring anisotropic interpolation estimates from affine triangles to exact curved elements.

In this paper, we specialise this factorisation to the radially mapped linear Lagrange element on exact geodesic spherical triangles. The geometric comparison results used below, including the relation between the chordal and intrinsic maximum-angle parameters and the support-plane identity proved here, are also established in the companion preprint \emph{Exact characterisation of maximum-angle conditions for spherical finite element meshes} \cite{Ishizaka2026SphericalMaxAngle}. The present paper develops the interpolation analysis through exact radial transfer identities, asymmetric directional scales, and explicit sharpness families.

Let us briefly describe the setting of this paper. Let $T\subset\R^3$ be a non-degenerate chordal triangle whose vertices lie on the sphere $\Sph_R^2$, and assume $d_T:=\dist(0,\aff(T))>0$. We define the radial projection $\Psi_T$ by $\Psi_T(y):=Ry/|y|$ and set $K:=\Psi_T(T)$. Let $I_K$ be the nodal interpolation operator on $K$. We use the canonical labelling introduced in Section~\ref{sec:geometry}. In this labelling, $h_{T,1}$ and $h_{T,2}$ denote the lengths of the two edges of $T$ that share the vertex opposite the longest edge, ordered so that $h_{T,2}\leq h_{T,1}$. We set $\mu_T:=2/\sin\beta_T$, where $\beta_T$ is the largest interior angle of $T$, and $\chi_T:=R/d_T$. See Section~\ref{sec:geometry} for details.

We first derive exact transfer identities for the surface measure and the $H^1$ seminorm under the radial projection. The transfer factors for the $L^2$ norm are $\chi_T^{-1/2}$ and $\chi_T$, and those for the $H^1$ seminorm are $\chi_T^{-1/2}$ and $\chi_T^{1/2}$. We also derive identities for the first- and second-order directional derivatives. These identities transport each chordal direction to a pointwise tangent direction on $K$ and separate the transformed covariant Hessian from an additional correction involving first derivatives.

Combining these identities with affine anisotropic interpolation estimates gives local $L^2$- and $H^1$-error bounds on the exact spherical triangle. In particular, for $v\in H^2(K)$ and $w:=v\circ\Psi_T$,
\begin{align*}
|v-I_Kv|_{H^1(K)}
\leq C\chi_T^{1/2}
\left(
\mu_T h_{T,1}|\partial_{\ell_{T,1}}w|_{H^1(T)}
+2h_{T,2}|\partial_{\ell_{T,2}}w|_{H^1(T)}
\right),
\end{align*}
where $\ell_{T,1}$ and $\ell_{T,2}$ are the two affine directions and $C$ is independent of the element size, shape, and sphere radius. Keeping the derivative components separate during affine scaling confines the maximum-angle factor $\mu_T$ to the first directional coefficient. Section~\ref{sec:local-interpolation} also rewrites the directional data in terms of covariant derivatives on $K$.

To examine the role of the geometric factors in the estimates, we consider four explicit families of spherical triangles. The first family shows that the anisotropic order is lost if the directional lengths $h_{T,1}$ and $h_{T,2}$ are replaced by the diameter of the element. The second family gives the sharp order of $\mu_T$ and rules out replacing it by $\mu_T^{\theta}$ with $\theta<1$ in the same two-term directional bound. The third family shows that, within the same two-term directional bound with a constant independent of the element, the maximum-angle factor cannot be moved from the longer directional scale to the shorter one. The fourth family proves that the upper $H^1$ radial-transfer exponent $1/2$ is optimal for the same transported directional scale.

Because the radial projections induce matching parametrisations of the edges, the local construction can be extended to compatible spherical triangulations. Under the intrinsic spherical maximum-angle condition and relative refinement $h/R\to0$, where $h$ is the maximum chordal diameter, we obtain the standard global interpolation rates in the $L^2$ norm and the $H^1$ seminorm for $v\in H^2(\Sph_R^2)$ from the local estimates. Furthermore, using the same directional estimates, we present an anisotropic error estimate in the energy norm for the conforming finite element approximation of the Laplace--Beltrami problem with the mean-zero condition, assuming exact integration.

The remainder of this paper is organised as follows. In Section~\ref{sec:geometry}, we introduce the geometric setting and derive the exact radial transfer identities. In Section~\ref{sec:local-interpolation}, we prove the local interpolation error estimates. In Section~\ref{sec:sharpness}, we analyse the four families of spherical triangles and discuss the sharpness of the estimates. In Section~\ref{sec:global}, we present the global interpolation error estimates and the error estimate for the finite element approximation of the Laplace--Beltrami problem. Finally, we conclude the paper with some remarks in the last section.

\section{Geometric setting and element maps}
\label{sec:geometry}
We construct the exact spherical element from the reference triangle and separate affine anisotropy from radial distortion. We first factor the affine core, then establish the affine and radial transfer identities, and finally relate the chordal and spherical maximum-angle parameters to radial locality. These ingredients are used separately in Section~\ref{sec:local-interpolation}.

\subsection{Reference triangle and construction of the spherical element map}
\label{subsec:element-map}
Recall that the sphere of radius $R>0$ is
\begin{align*}
  \Sph_R^2=\{x\in\R^3:|x|=R\}.
\end{align*}
We consider the reference triangle
\begin{align*}
  \widehat T&:=\conv\{\hat p_1,\hat p_2,\hat p_3\},
  \quad
  \hat p_1:=(0,0)^{\top},
  \quad
  \hat p_2:=(1,0)^{\top},
  \quad
  \hat p_3:=(0,1)^{\top}.
\end{align*}
The distinct, non-collinear vertices $p_1,p_2,p_3\in\Sph_R^2$ determine the chordal triangle
\begin{align*}
  T=\conv\{p_1,p_2,p_3\}\subset\R^3.
\end{align*}
Here, $T$ is the chordal triangle and $\aff(T)$ denotes its affine hull. We assume $0\notin\aff(T)$, equivalently $\det[p_1\ p_2\ p_3]\neq0$; geometrically, the three vertices do not lie on a common great circle. Let $\nu_T$ be the unit normal oriented so that
\begin{align}
  \aff(T)=\{y\in\R^3:\nu_T\cdot y=d_T\},
  \quad
  d_T=\dist(0,\aff(T))>0.
  \label{eq:support-plane-normalisation}
\end{align}
Because $|\nu_T|=1$, the point
\begin{align*}
  c_T:=d_T\nu_T
\end{align*}
satisfies $\nu_T\cdot c_T=d_T$ and is the orthogonal projection of the origin onto $\aff(T)$.  The vector space parallel to $\aff(T)$ is
\begin{align*}
  V_T:=\nu_T^\perp
  =\{\xi\in\R^3:\nu_T\cdot\xi=0\}.
\end{align*}
Indeed, for any $y\in\R^3$,
\begin{align*}
  y\in\aff(T)
  \quad\Longleftrightarrow\quad
  \nu_T\cdot(y-c_T)=0
  \quad\Longleftrightarrow\quad
  y-c_T\in V_T.
\end{align*}
Therefore,
\begin{align*}
  \aff(T)=c_T+V_T,
  \quad
  T\subset c_T+V_T\subset\R^3,
  \quad
  \dim V_T=2.
\end{align*}
In particular, any $y\in T$ has the unique decomposition
\begin{align*}
  y=c_T+z,
  \quad z:=y-c_T\in V_T,
  \quad \nu_T\cdot y=d_T>0.
\end{align*}

\begin{convention}[Canonical labelling of the chordal affine core]
  \label{conv:chordal-labelling}
  Choose a longest edge and label its endpoints $p_2,p_3$ and the opposite vertex is $p_1$. With the notation of the Introduction,
  \begin{align*}
    \beta_T=\angle p_2p_1p_3,
    \quad
    h_T=|p_2-p_3|,
    \quad
    h_{T,1}=|p_1-p_2|,
    \quad
    h_{T,2}=|p_1-p_3|.
  \end{align*}
  Interchange $p_2,p_3$ if necessary so that $h_{T,2}\leq h_{T,1}$. If the longest edge is not unique, choose any one of them; if $h_{T,1}=h_{T,2}$, either ordering of $p_2,p_3$ may be used.
\end{convention}
Then, $\beta_T$ is a largest angle, $h_{T,2}\leq h_{T,1}\leq h_T<2h_{T,1}$, and this convention imposes no geometric restriction. It only fixes the vertex labelling used by the reference parametrisation; the sets $T$ and $K$ and the quantities $\nu_T$ and $d_T$ are unchanged. All subsequent estimates hold for every labelling satisfying this convention.

Recall that the affine element map from the reference triangle onto the chordal triangle is
\begin{align}
  \Phi_T:\widehat T&\longrightarrow T,
  \quad
  \Phi_T(\hat x)=p_1+B_T\hat x,
  \quad
  B_T:=\begin{pmatrix}p_2-p_1&p_3-p_1\end{pmatrix}.
  \label{eq:reference-affine-map}
\end{align}
The matrix $B_T$ has rank two, hence $\Phi_T$ maps $\widehat T$ bijectively onto $T$. Because $\nu_T\cdot y=d_T>0$ on $T$, the radial projection is well defined:
\begin{align}
  \Psi_T:T&\longrightarrow\Sph_R^2,
  \quad
  \Psi_T(y)=R\frac{y}{|y|}.
  \label{eq:radial-projection}
\end{align}
The exact spherical triangle and the complete element map are
\begin{align}
  K&=\Psi_T(T),
  \quad
  F_K=\Psi_T\circ\Phi_T:\widehat T\longrightarrow K.
  \label{eq:complete-element-map}
\end{align}
Explicitly,
\begin{align}
  F_K(\hat x)
  =R\frac{(1-\hat x_1-\hat x_2)p_1+\hat x_1p_2+\hat x_2p_3}
  {|(1-\hat x_1-\hat x_2)p_1+\hat x_1p_2+\hat x_2p_3|},
  \quad
  F_K(\hat p_i)=p_i.
  \label{eq:complete-element-map-explicit}
\end{align}

The following lemma also shows how the chordal edges are mapped onto great-circle arcs, a fact used later to assemble conforming spherical triangulations.

\begin{lemma}[Basic properties of the radial correction]
  \label{lem:radial-basic}
  The map $\Psi_T:T\to K$ is a homeomorphism with inverse
  \begin{align}
    G_T(x)
    :=\Psi_T^{-1}(x)
    =\frac{d_T}{\nu_T\cdot x}x,
    \quad x\in K.
    \label{eq:radial-inverse-map}
  \end{align}
  Furthermore, for each edge
  \begin{align*}
    e_{ij}:=\conv\{p_i,p_j\},
    \quad 1\leq i<j\leq3,
  \end{align*}
  the restriction $\Psi_T|_{e_{ij}}$ maps $e_{ij}$ bijectively onto the closed shorter great-circle arc joining $p_i$ and $p_j$. Consequently, $F_K$ maps each edge of $\widehat T$ bijectively onto the corresponding closed shorter great-circle arc of $K$.
\end{lemma}

\begin{proof}
  Because
  \begin{align*}
    \nu_T\cdot y=d_T>0,
    \quad y\in T,
  \end{align*}
  one has $y\neq0$, and hence $\Psi_T$ is well defined and continuous on $T$. Suppose that $\Psi_T(y_1)=\Psi_T(y_2)$. Then,
  \begin{align*}
    \frac{y_1}{|y_1|}
    =
    \frac{y_2}{|y_2|},
  \end{align*}
  and therefore
  \begin{align*}
    y_1=c y_2,
    \quad
    c :=\frac{|y_1|}{|y_2|}>0.
  \end{align*}
  Taking the scalar product with $\nu_T$ gives $d_T=c d_T$. Because $d_T>0$, it follows that $c=1$, and hence $y_1=y_2$.  Thus, $\Psi_T$ is injective.  Its surjectivity follows from the definition $K=\Psi_T(T)$. Let $x=\Psi_T(y)=Ry/|y|$. Then,
  \begin{align*}
    \nu_T\cdot x
    =
    \frac{Rd_T}{|y|}
    >0,
  \end{align*}
  and consequently
  \begin{align*}
    \frac{d_T}{\nu_T\cdot x}x
    =
    \frac{d_T}{Rd_T/|y|}\frac{R}{|y|}y
    =y.
  \end{align*}
  This proves \eqref{eq:radial-inverse-map}. Because $\nu_T\cdot x>0$ on $K$, the inverse formula defines a continuous map on $K$. Therefore, $\Psi_T:T\to K$ is a homeomorphism.

  Fix $1\leq i<j\leq3$. The endpoints $p_i$ and $p_j$ cannot be antipodal. Indeed, otherwise $0$ would be their midpoint and hence $0\in e_{ij}\subset T$, contradicting $\nu_T\cdot y=d_T>0$ on $T$. We define the corresponding great circle and its positive semicircle as
  \begin{align*}
    \mathcal C_{ij}
    &:=\Sph_R^2\cap\operatorname{span}\{p_i,p_j\},
    \quad
    \mathcal C_{ij}^{+}
    :=\{x\in\mathcal C_{ij}:\nu_T\cdot x>0\}.
  \end{align*}
  Because $\nu_T\cdot p_i=\nu_T\cdot p_j=d_T>0$, the restriction of $x\mapsto\nu_T\cdot x$ to $\operatorname{span}\{p_i,p_j\}$ is non-zero. Therefore, $\mathcal C_{ij}^{+}$ is one of the two open semicircles determined by the equation $\nu_T\cdot x=0$ on $\mathcal C_{ij}$. For any $y\in e_{ij}$,
  \begin{align*}
    \Psi_T(y)\in\mathcal C_{ij},
    \quad
    \nu_T\cdot\Psi_T(y)
    =\frac{Rd_T}{|y|}>0.
  \end{align*}
  Therefore,
  \begin{align*}
    \Psi_T(e_{ij})\subset\mathcal C_{ij}^{+}.
  \end{align*}
  The restriction $\Psi_T|_{e_{ij}}$ is continuous and injective, and it fixes the endpoints because $\Psi_T(p_i)=p_i$ and $\Psi_T(p_j)=p_j$. Because $\mathcal C_{ij}^{+}$ is homeomorphic to an open interval, its image is exactly the closed arc in $\mathcal C_{ij}^{+}$ joining $p_i$ and $p_j$. This arc has length strictly less than $\pi R$ and is therefore the closed shorter great-circle arc. 
  
  Finally, $\Phi_T$ maps each edge of $\widehat T$ bijectively onto the corresponding edge of $T$. The assertion for $F_K$ follows from $F_K=\Psi_T\circ\Phi_T$.
\end{proof}

In fact, the radial map extends to a smooth diffeomorphism from $\aff(T)$ onto the open hemisphere $U_T:=\{x\in\Sph_R^2:\nu_T\cdot x>0\}$, with inverse $x\mapsto d_Tx/(\nu_T\cdot x)$. Therefore, the interior of $T$ relative to $\aff(T)$ is mapped onto the interior of $K$ relative to $\Sph_R^2$. The set $K$ is geodesically convex, since the radial normalisation of any positive linear combination of two points in $K$ is the image of a convex combination of their preimages in $T$.

The support-plane distance is computable from the vertices:
\begin{align}
  d_T
  =\frac{|\det(p_1,p_2,p_3)|}
  {|(p_2-p_1)\times(p_3-p_1)|}
  =\frac{|\det(p_1,p_2,p_3)|}{2|T|_2},
  \label{eq:support-plane-distance-formula}
\end{align}
where $|T|_2=\mathcal H^2(T)$ is the two-dimensional Hausdorff measure of the chordal triangle. 

\subsection{Affine factorisation and anisotropic coordinates}
\label{subsec:planar-factorisation}
With the canonical labelling fixed in Convention~\ref{conv:chordal-labelling}, the affine map $\Phi_T$ is factorised into a planar anisotropic map and an isometric placement in the chordal plane. We first choose orthonormal coordinates in the two-dimensional space $V_T$.

Because all three vertices belong to $\aff (T)=c_T+V_T$, their differences $p_2-p_1$ and $p_3-p_1$ belong to $V_T$. We set
\begin{align*}
  q_{T,1}:=\frac{p_2-p_1}{h_{T,1}}\in V_T,
\end{align*}
so that $|q_{T,1}|=1$. The component of $p_3-p_1$ orthogonal to $q_{T,1}$ is
\begin{align*}
  r_{T,2}:=(p_3-p_1) - \left((p_3-p_1)\cdot q_{T,1}\right)q_{T,1}.
\end{align*}
Then, $r_{T,2}\in V_T$ and $r_{T,2}\cdot q_{T,1}=0$. If $r_{T,2}=0$, then $p_3-p_1$ is parallel to $p_2-p_1$, contrary to the non-collinearity of the vertices. Therefore, $r_{T,2}\neq0$, and we may define
\begin{align}
  q_{T,2}:=\frac{r_{T,2}}{|r_{T,2}|},
  \quad
  Q_T:=\begin{pmatrix}q_{T,1}&q_{T,2}\end{pmatrix}
  \in\R^{3\times2}.
  \label{eq:chordal-orthonormal-basis}
\end{align}
Therefore, $q_{T,1},q_{T,2}$ form an orthonormal basis of $V_T$, and $Q_T$ is an isometric isomorphism from $\R^2$ onto $V_T$. In particular,
\begin{align}
  Q_T^{\top}Q_T=I_2,
  \quad
  |Q_Tz|=|z|
  \quad\text{for any }z\in\R^2.
  \label{eq:chordal-basis-orthonormality}
\end{align}
Furthermore,
\begin{align*}
  (p_3-p_1)\cdot q_{T,2}
  =r_{T,2}\cdot q_{T,2}
  =|r_{T,2}|>0.
\end{align*}
The sign of $q_{T,2}$ is therefore fixed by the position of $p_3$ and does not depend on the orientation chosen for $\nu_T$. We set
\begin{align}
  s_T:=\frac{(p_3-p_1)\cdot q_{T,1}}{h_{T,2}}=\cos\beta_T,
  \quad
  t_T:=\frac{(p_3-p_1)\cdot q_{T,2}}{h_{T,2}}
  =\frac{|r_{T,2}|}{h_{T,2}}=\sin\beta_T>0.
  \label{eq:planar-shear-parameters}
\end{align}
The orthogonal decomposition used above now reads
\begin{align}
  p_3-p_1
  =h_{T,2}\bigl(s_Tq_{T,1}+t_Tq_{T,2}\bigr).
  \label{eq:third-edge-orthonormal-decomposition}
\end{align}
Because $|p_3-p_1|=h_{T,2}$ and $q_{T,1},q_{T,2}$ are orthonormal, this identity also gives $s_T^2+t_T^2=1$.

We define
\begin{align}
  A_{\widetilde T}
  :=\begin{pmatrix}
    h_{T,1}&h_{T,2}s_T\\
    0&h_{T,2}t_T
  \end{pmatrix},
  \quad
  \Phi_{\widetilde T}(\hat x):=A_{\widetilde T}\hat x,
  \quad
  \widetilde T:=\Phi_{\widetilde T}(\widehat T),
  \quad
  \iota_T(z):=p_1+Q_Tz,
  \quad z\in\R^2.
  \label{eq:planar-anisotropic-matrix}
\end{align}
Because $\det A_{\widetilde T}=h_{T,1}h_{T,2}t_T>0$, $A_{\widetilde T}\in\R^{2\times2}$ is invertible. From \eqref{eq:third-edge-orthonormal-decomposition},
\begin{align*}
  Q_TA_{\widetilde T}
  &=
  \begin{pmatrix}
    h_{T,1}q_{T,1}
    &h_{T,2}(s_Tq_{T,1}+t_Tq_{T,2})
  \end{pmatrix}
  =
  \begin{pmatrix}p_2-p_1&p_3-p_1\end{pmatrix}
  =B_T.
\end{align*}
Therefore,
\begin{align}
  \Phi_T=\iota_T\circ\Phi_{\widetilde T},
  \quad
  B_T=Q_TA_{\widetilde T},
  \quad
  F_K=\Psi_T\circ\iota_T\circ\Phi_{\widetilde T}.
  \label{eq:affine-stage-factorisation}
\end{align}
Here, $A_{\widetilde T}\in\R^{2\times2}$ contains the planar scaling and shear, while $Q_T\in\R^{3\times2}$ has orthonormal columns. Thus, the complete element map has the four-level factorisation
\begin{align}
  \widehat T
  \xrightarrow{\ \Phi_{\widetilde T}\ }
  \widetilde T
  \xrightarrow{\ \iota_T\ }
  T
  \xrightarrow{\ \Psi_T\ }
  K.
  \label{eq:four-level-factorisation}
\end{align}
The first map carries the directional scales and angle dependence, $\iota_T$ is an isometric placement in the chordal plane, and $\Psi_T$ supplies the radial correction. Figure~\ref{fig:four-level-element-map} summarises this factorisation.

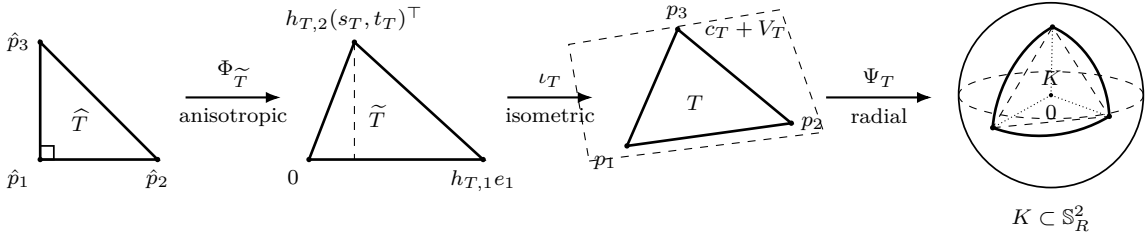
\begin{figure}[t]
  \centering
  \begin{tikzpicture}[
    x=1cm,y=1cm,
    >={Latex[length=2mm]},
    map/.style={->,line width=0.8pt},
    geom/.style={line width=0.65pt},
    strong/.style={line width=1pt},
    aux/.style={thin,dashed},
    ray/.style={thin,densely dotted},
    every node/.style={font=\scriptsize}
  ]
    % Reference triangle
    \coordinate (r1) at (0,0);
    \coordinate (r2) at (1.55,0);
    \coordinate (r3) at (0,1.55);
    \draw[strong] (r1)--(r2)--(r3)--cycle;
    \draw[geom] (0,0)--(0.18,0)--(0.18,0.18)--(0,0.18);
    \fill (r1) circle (1pt) (r2) circle (1pt) (r3) circle (1pt);
    \node[below left] at (r1) {$\hat p_1$};
    \node[below] at (r2) {$\hat p_2$};
    \node[left] at (r3) {$\hat p_3$};
    \node at (0.52,0.52) {$\widehat T$};

    % Planar anisotropic representative
    \coordinate (a1) at (3.55,0);
    \coordinate (a2) at (5.85,0);
    \coordinate (a3) at (4.15,1.55);
    \draw[strong] (a1)--(a2)--(a3)--cycle;
    \draw[aux] (a3)--(4.15,0);
    \fill (a1) circle (1pt) (a2) circle (1pt) (a3) circle (1pt);
    \node[below left] at (a1) {$0$};
    \node[below] at (a2) {$h_{T,1}e_1$};
    \node[above] at (a3) {$h_{T,2}(s_T,t_T)^{\top}$};
    \node at (4.45,0.58) {$\widetilde T$};

    % Chordal affine core
    \coordinate (c1) at (7.75,0.18);
    \coordinate (c2) at (9.92,0.48);
    \coordinate (c3) at (8.42,1.72);
    \draw[aux] (7.25,-0.02)--(10.35,0.36)--(9.78,1.98)--(6.98,1.52)--cycle;
    \draw[strong] (c1)--(c2)--(c3)--cycle;
    \fill (c1) circle (1pt) (c2) circle (1pt) (c3) circle (1pt);
    \node[below left] at (c1) {$p_1$};
    \node[right] at (c2) {$p_2$};
    \node[above] at (c3) {$p_3$};
    \node at (8.65,0.75) {$T$};
    \node at (9.32,1.72) {$c_T+V_T$};

    % Spherical element
    \coordinate (O) at (13.35,0.85);
    \draw[geom] (O) circle[radius=1.22];
    \draw[aux] ($(O)+(1.22,0)$)
      arc[start angle=0,end angle=180,x radius=1.22,y radius=0.31];
    \draw[aux] ($(O)+(-1.22,0)$)
      arc[start angle=180,end angle=360,x radius=1.22,y radius=0.31];
    \coordinate (s1) at (12.58,0.42);
    \coordinate (s2) at (14.12,0.57);
    \coordinate (s3) at (13.37,1.75);
    \draw[aux] (s1)--(s2)--(s3)--cycle;
    \draw[strong] (s1)..controls(13.05,0.25)and(13.62,0.27)..(s2);
    \draw[strong] (s2)..controls(14.10,1.10)and(13.78,1.57)..(s3);
    \draw[strong] (s3)..controls(12.96,1.53)and(12.62,0.96)..(s1);
    \draw[ray] (O)--(s1) (O)--(s2) (O)--(s3);
    \fill (s1) circle (1pt) (s2) circle (1pt) (s3) circle (1pt)
      (O) circle (0.8pt);
    \node[below] at (O) {$0$};
    \node at (13.35,1.08) {$K$};
    \node[below=3pt] at ($(O)+(0,-1.22)$) {$K\subset\Sph_R^2$};

    % Maps
    \draw[map] (1.92,0.82)--(3.18,0.82)
      node[midway,above] {$\Phi_{\widetilde T}$}
      node[midway,below] {anisotropic};
    \draw[map] (6.17,0.82)--(7.30,0.82)
      node[midway,above] {$\iota_T$}
      node[midway,below] {isometric};
    \draw[map] (10.38,0.82)--(11.78,0.82)
      node[midway,above] {$\Psi_T$}
      node[midway,below] {radial};
  \end{tikzpicture}
  \caption{Four-level factorisation of the spherical element map. The planar map carries the anisotropic scales, the affine isometry places the triangle in its chordal support plane, and the radial projection maps the chordal core onto the exact geodesic spherical triangle.}
  \label{fig:four-level-element-map}
\end{figure}

\subsection{Tangential differential operators on the sphere}
\label{subsec:spherical-calculus}
We next fix the spherical differential notation used below.

For $x\in\Sph_R^2$, let $n(x):=x/R$ be the outward unit normal. We identify the tangent space at $x$ with
\begin{align*}
  T_x\Sph_R^2
  :=\{\tau\in\R^3:\tau\cdot n(x)=0\}
  =n(x)^\perp.
\end{align*}
The orthogonal projection onto $T_x\Sph_R^2$ is
\begin{align*}
  P(x):=I-n(x)n(x)^\top,
\end{align*}
where $I$ is the $3\times3$ identity matrix. Throughout the paper, $D$ and $D^2$ denote the first and second Fr\'echet derivatives of Euclidean-valued maps on open subsets of Euclidean spaces.

Let $v:\Sph_R^2\to\R$ be smooth. We choose an open set $U\subset\R^3$ containing $\Sph_R^2$ and a smooth extension
\begin{align*}
  v^{\mathrm e}:U\longrightarrow\R,
  \quad
  v^{\mathrm e}|_{\Sph_R^2}=v.
\end{align*}
The tangential gradient is
\begin{align*}
  \nabla_{\Sph}v(x):=P(x)\nabla v^{\mathrm e}(x).
\end{align*}
We set $g:=\nabla_{\Sph}v:\Sph_R^2\to\R^3$ and choose a smooth extension
\begin{align*}
  g^{\mathrm e}:U\longrightarrow\R^3,
  \quad
  g^{\mathrm e}|_{\Sph_R^2}=g.
\end{align*}
The covariant Hessian is the symmetric bilinear form
\begin{align*}
  \nabla_{\Sph}^2v(x)[\tau,\eta]
  :=\tau\cdot P(x)Dg^{\mathrm e}(x)\eta,
  \quad \tau,\eta\in T_x\Sph_R^2.
\end{align*}
Both definitions are independent of the extensions: the difference between two extensions vanishes on the sphere, as do its tangential derivatives; see \cite{DziukElliott2013}.

We use the Frobenius norm
\begin{align*}
  |\nabla_{\Sph}^2v(x)|^2
  :=\sum_{i,j=1}^2
  \left|\nabla_{\Sph}^2v(x)[e_i,e_j]\right|^2,
\end{align*}
where $\{e_1,e_2\}$ is any orthonormal basis of $T_x\Sph_R^2$.

Let $\mathcal O\subset\Sph_R^2$ be a Lipschitz surface domain.
All $L^2(\mathcal O)$ norms use the surface measure $\diff\sigma$.
We define
\begin{align*}
  \|v\|_{H^1(\mathcal O)}^2
  &:=\|v\|_{L^2(\mathcal O)}^2
  +\|\nabla_{\Sph}v\|_{L^2(\mathcal O)}^2,
  \\
  \|v\|_{H^2(\mathcal O)}^2
  &:=\|v\|_{H^1(\mathcal O)}^2
  +\|\nabla_{\Sph}^2v\|_{L^2(\mathcal O)}^2,
\end{align*}
with derivatives understood in the weak sense. These norms are equivalent to the usual Sobolev norms in local coordinates. On chordal planes, $\diff s$ denotes two-dimensional Hausdorff measure.

The geodesic distance between $x,y\in\Sph_R^2$ is
\begin{align*}
  d_{\Sph}(x,y):=R\arccos\left(\frac{x\cdot y}{R^2}\right).
\end{align*}
For distinct, non-antipodal $x,y$, ``the arc joining $x$ and $y$'' means the unique shorter great-circle arc, whose length is $d_{\Sph}(x,y)$.

\subsection{Transfer identities for the affine isometry}
\label{subsec:affine-transfer}
We first fix the chordal differential notation and then give the exact transfer identities under the affine isometry.

Because $\widetilde T=\Phi_{\widetilde T}(\widehat T)$ and $T=\Phi_T(\widehat T)$, the factorisation
\begin{align*}
  \Phi_T=\iota_T\circ\Phi_{\widetilde T},
  \quad
  \iota_T(z)=p_1+Q_Tz,
  \quad
  B_T=Q_TA_{\widetilde T},
  \quad
  Q_T^{\top}Q_T=I_2,
\end{align*}
shows that $\iota_T:\widetilde T\to T$ is an affine isometry. Indeed, for any $z_1,z_2\in\R^2$,
\begin{align*}
  |\iota_T(z_1)-\iota_T(z_2)|
  =|Q_T(z_1-z_2)|
  =|z_1-z_2|.
\end{align*}
In particular, the surface Jacobian of $\iota_T$ is one. For $w\in H^1(T)$, let $\nabla_Tw\in L^2(T;V_T)$ denote its intrinsic weak gradient in the affine plane $\aff(T)$. The intrinsic weak differential $D_Tw(y)$ is the linear functional on $V_T$ given, for almost every $y\in T$, by
\begin{align*}
  D_Tw(y)[\xi]
  :=\partial_\xi w(y)
  =\nabla_Tw(y)\cdot\xi,
  \quad \xi\in V_T,
\end{align*}
where $\partial_\xi$ denotes weak differentiation in the constant in-plane direction $\xi$. Thus, these derivatives are taken within the flat chordal plane $\aff(T)$. We set
\begin{align*}
  \tilde w:=w\circ\iota_T\in H^1(\widetilde T).
\end{align*}
The change-of-variables formula for the isometry gives
\begin{align}
  \|\tilde w\|_{L^2(\widetilde T)}
  =\|w\|_{L^2(T)}.
  \label{eq:isometric-L2-transfer}
\end{align}
Furthermore, for almost every $z\in\widetilde T$ and any $\hat\xi\in\R^2$, the Sobolev chain rule gives
\begin{align}
  \nabla\tilde w(z)
  &=Q_T^{\top}\nabla_Tw(\iota_T(z)),
  \label{eq:isometric-gradient-pullback}
  \\
  |\nabla\tilde w(z)|
  &=|\nabla_Tw(\iota_T(z))|,
  \label{eq:isometric-gradient-norm}
  \\
  \partial_{\hat\xi}\tilde w(z)
  &=(Q_T\hat\xi)\cdot\nabla_Tw(\iota_T(z)).
  \label{eq:isometric-directional-transfer}
\end{align}
The matrix $Q_TQ_T^{\top}$ is the orthogonal projection onto $V_T$. Because the intrinsic gradient takes values in $V_T$, this proves \eqref{eq:isometric-gradient-norm}.

For $w\in H^2(T)$, we write $D_T^2w$ for its intrinsic weak Hessian, viewed almost everywhere as a symmetric bilinear form on $V_T\times V_T$. For almost every $y\in T$ and constant vectors $\xi,\eta\in V_T$,
\begin{align*}
  D_T^2w(y)[\xi,\eta]
  =\partial_\xi\partial_\eta w(y).
\end{align*}
We use the same symbol for the associated self-adjoint linear map $V_T\to V_T$, defined through the Euclidean inner product by
\begin{align*}
  \xi\cdot\bigl(D_T^2w(y)\eta\bigr)
  =D_T^2w(y)[\xi,\eta].
\end{align*}
Its pointwise norm is the Frobenius norm
\begin{align*}
  |D_T^2w(y)|^2
  :=\sum_{i,j=1}^2
  \left|D_T^2w(y)[e_i,e_j]\right|^2,
\end{align*}
where $\{e_1,e_2\}$ is any orthonormal basis of $V_T$; the corresponding $L^2(T)$ norm uses the planar measure $\diff s$.

For almost every $z\in\widetilde T$ and any $\hat\xi,\hat\eta\in\R^2$, one has
\begin{align}
  D^2\tilde w(z)[\hat\xi,\hat\eta]
  =D_T^2w(\iota_T(z))[Q_T\hat\xi,Q_T\hat\eta].
  \label{eq:isometric-hessian-transfer}
\end{align}
This identity follows first for smooth functions from the second-order chain rule and extends to $H^2(T)$ by density. The inverse of the full affine map is
\begin{align}
  \Phi_T^{-1}(y)
  =A_{\widetilde T}^{-1}Q_T^{\top}(y-p_1)
  =(B_T^{\top}B_T)^{-1}B_T^{\top}(y-p_1),
  \quad y\in T.
  \label{eq:affine-inverse-map}
\end{align}
The constant surface Jacobian of $\Phi_T$ is
\begin{align}
  J_{\Phi_T}
  =\sqrt{\det(B_T^{\top}B_T)}
  =\det A_{\widetilde T}
  =h_{T,1}h_{T,2}t_T
  =2|T|_2.
  \label{eq:affine-surface-jacobian}
\end{align}
Thus, for any $f\in L^1(T)$,
\begin{align}
  \int_Tf\,\diff s
  =\int_{\widehat T}(f\circ\Phi_T)J_{\Phi_T}\,\diff\hat x.
  \label{eq:affine-change-of-variables}
\end{align}
Finally, if $w\in H^1(T)$ and $\hat w:=w\circ\Phi_T$, we write $\widehat\nabla$ for the gradient with respect to the reference variable $\hat x$. Then,
\begin{align}
  \widehat\nabla\hat w
  &=B_T^{\top}(\nabla_Tw)\circ\Phi_T,
  \label{eq:affine-gradient-pullback}
  \\
  (\nabla_Tw)\circ\Phi_T
  &=B_T(B_T^{\top}B_T)^{-1}\widehat\nabla\hat w
   =Q_TA_{\widetilde T}^{-\top}\widehat\nabla\hat w.
  \label{eq:affine-gradient-pushforward}
\end{align}
These identities follow from $Q_T^{\top}Q_T=I_2$ and the Sobolev chain rule. Thus, $Q_T$ introduces no geometric distortion, and all non-trivial affine scaling is contained in $A_{\widetilde T}$.

\subsection{Differential properties of the radial projection}
We establish the distortion bounds needed for regularity and norm transfer; the sharper directional identities are derived in Section~\ref{sec:local-interpolation}.

To define the Euclidean Fr\'echet derivatives of the radial correction, we consider the ambient map
\begin{align*}
  \Psi:\R^3\setminus\{0\}\longrightarrow\Sph_R^2,
  \quad
  \Psi(y):=R\frac{y}{|y|},
\end{align*}
so that $\Psi_T=\Psi|_T$. Because $|y|\geq d_T>0$ for any $y\in T$, $\Psi$ is $\mathcal C^\infty$ on an open neighbourhood of $T$. For $y\in T$, we set
\begin{align*}
  x:=\Psi_T(y),
  \quad
  n_y:=\frac{y}{|y|}=\frac{x}{R},
  \quad
  P_y:=I-n_yn_y^{\top}=P(x).
\end{align*}
Thus, $P_y$ is the orthogonal projection of $\R^3$ onto $T_x\Sph_R^2$.  Direct differentiation gives, for $\xi,\eta\in\R^3$,
\begin{align}
  D\Psi(y)\xi
  &=\frac{R}{|y|}P_y\xi,
  \label{eq:radial-first-derivative}
  \\
  D^2\Psi(y)[\xi,\eta]
  &=-\frac{R}{|y|^3}
  \left((y\cdot\eta)\xi+(y\cdot\xi)\eta+(\xi\cdot\eta)y\right)
  +\frac{3R}{|y|^5}(y\cdot\xi)(y\cdot\eta)y.
  \label{eq:radial-second-derivative}
\end{align}
In particular,
\begin{align*}
  \operatorname{range}D\Psi(y)=T_x\Sph_R^2,
  \quad
  \ker D\Psi(y)=\operatorname{span}\{y\}.
\end{align*}
Furthermore,
\begin{align}
  d_T\leq|y|\leq R,
  \quad
  \chi_T=\frac{R}{d_T}>1.
  \label{eq:radial-radius-and-distortion}
\end{align}
The lower bound follows from $\nu_T\cdot y=d_T$ and the upper bound from convexity of the Euclidean norm. Furthermore, $d_T<R$. A support plane at distance $R$ would be tangent to the sphere and could not contain three distinct non-collinear points.

For notational simplicity, $D\Psi_T(y)$ and $D^2\Psi_T(y)$ denote the ambient Euclidean derivatives of $\Psi(y)=Ry/|y|$, evaluated at $y\in T$. Their arguments will lie in $V_T$ whenever they are used intrinsically.

\begin{lemma}[Metric and area distortion of the radial correction]
  \label{lem:metric-distortion}
  Let $y\in T$, set $x:=\Psi_T(y)$ and $r:=|y|$, and regard the restricted derivative as a linear map
  \begin{align*}
    D\Psi(y)|_{V_T}:
    (V_T,|\cdot|)\longrightarrow(T_x\Sph_R^2,|\cdot|).
  \end{align*}
  Then, for any $\xi\in V_T$,
  \begin{align}
    \chi_T^{-1}|\xi|
    \leq|D\Psi(y)\xi|
    \leq\chi_T|\xi|.
    \label{eq:radial-metric-bounds}
  \end{align}
  The two-dimensional surface Jacobian of $\Psi_T:T\to K$ is
  \begin{align}
    J_T(y)
    :=\sqrt{\det\!\left(
      (D\Psi(y)|_{V_T})^{\top}(D\Psi(y)|_{V_T})
    \right)}
    =\frac{R^2d_T}{|y|^3},
    \label{eq:radial-surface-jacobian}
  \end{align}
  where the determinant is computed with respect to any orthonormal basis of $V_T$.  Consequently,
  \begin{align}
    \chi_T^{-1}\leq J_T(y)\leq\chi_T^2.
    \label{eq:radial-jacobian-bounds}
  \end{align}
  Finally, equip the space of bilinear maps $\R^3\times\R^3\to\R^3$ with the operator norm
  \begin{align}
    \|D^2\Psi(y)\|_{\mathrm{op}}
    :=\sup_{\substack{\xi,\eta\in\R^3\\|\xi|=|\eta|=1}}
    |D^2\Psi(y)[\xi,\eta]|.
    \label{eq:radial-second-derivative-operator-norm}
  \end{align}
  With this norm, the following crude bound holds
  \begin{align}
    \|D^2\Psi(y)\|_{\mathrm{op}}
    \leq\frac{6\chi_T^2}{R}.
    \label{eq:radial-second-derivative-bound}
  \end{align}
\end{lemma}

\begin{proof}
  Write $y=d_T\nu_T+z$ with $z\in V_T$.  Suppose first that $z\neq0$, and choose
  \begin{align*}
    e_{\mathrm{rad}}:=\frac{z}{|z|}
  \end{align*}
  and a unit vector $e_{\mathrm{tan}}\in V_T$ orthogonal to
  $e_{\mathrm{rad}}$.  Since
  \begin{align*}
    n_y=\frac{d_T}{r}\nu_T+\frac{|z|}{r}e_{\mathrm{rad}},
  \end{align*}
  one has
  \begin{align*}
    |P_ye_{\mathrm{rad}}|^2
    =1-(e_{\mathrm{rad}}\cdot n_y)^2
    =\frac{d_T^2}{r^2},
    \quad
    P_ye_{\mathrm{tan}}=e_{\mathrm{tan}},
  \end{align*}
  and the two images are orthogonal. Thus, the singular values of $P_y|_{V_T}$ are $d_T/r$ and $1$, and those of $D\Psi(y)|_{V_T}=(R/r)P_y|_{V_T}$ are
  \begin{align}
    \sigma_{\min}(y)=\frac{Rd_T}{r^2},
    \quad
    \sigma_{\max}(y)=\frac{R}{r}.
    \label{eq:radial-singular-values}
  \end{align}
  If $z=0$, then $r=d_T$, $n_y=\nu_T$, and $P_y|_{V_T}=I_{V_T}$, so the same singular-value formulas remain valid. Because $d_T\leq r\leq R$,
  \begin{align*}
    \sigma_{\min}(y)
    \geq\frac{d_T}{R}=\chi_T^{-1},
    \quad
    \sigma_{\max}(y)
    \leq\frac{R}{d_T}=\chi_T.
  \end{align*}
  This proves \eqref{eq:radial-metric-bounds}. The product of the singular values gives \eqref{eq:radial-surface-jacobian}, and
  \begin{align*}
    \chi_T^{-1}
    =\frac{d_T}{R}
    \leq\frac{R^2d_T}{r^3}
    \leq\frac{R^2}{d_T^2}
    =\chi_T^2,
  \end{align*}
  proving \eqref{eq:radial-jacobian-bounds}. Finally, \eqref{eq:radial-second-derivative} and Cauchy--Schwarz give
  \begin{align*}
    |D^2\Psi(y)[\xi,\eta]|
    \leq\frac{6R}{r^2}|\xi|\,|\eta|
    \leq\frac{6\chi_T^2}{R}|\xi|\,|\eta|,
  \end{align*}
  which proves \eqref{eq:radial-second-derivative-bound}.
\end{proof}

\subsection{Properties of the complete element map}
The preceding bounds for $\Psi_T$ combine with the affine map $\Phi_T$ to give the corresponding properties of the complete map $F_K$.
\begin{theorem}[Regularity of the complete spherical element map]
  \label{thm:complete-element-map}
  The map $F_K:\widehat T\to K$ is a homeomorphism and restricts to a $\mathcal C^\infty$-diffeomorphism from $\operatorname{int}\widehat T$ onto the interior of $K$ relative to $\Sph_R^2$. Its inverse is
  \begin{align}
    F_K^{-1}(x)
    =A_{\widetilde T}^{-1}Q_T^{\top}
    \left(\frac{d_T}{\nu_T\cdot x}x-p_1\right),
    \quad x\in K.
    \label{eq:complete-element-inverse}
  \end{align}
  For $\hat x\in\widehat T$ and $y:=\Phi_T(\hat x)$,
  \begin{align}
    &DF_K(\hat x) =D\Psi_T(y)B_T,
    \quad
    D^2F_K(\hat x)[\hat\xi,\hat\eta]
    =D^2\Psi_T(y)[B_T\hat\xi,B_T\hat\eta],
    \quad \hat\xi,\hat\eta\in\R^2,
    \notag\\
    &\operatorname{rank}DF_K(\hat x) =2.
    \label{eq:complete-map-differentials}
  \end{align}
  The first fundamental matrix and surface Jacobian are, respectively,
  \begin{align*}
    g_K(\hat x):=DF_K(\hat x)^{\top}DF_K(\hat x),
    \quad
    J_{F_K}(\hat x):=\sqrt{\det g_K(\hat x)},
  \end{align*}
  and the metric satisfies
  \begin{align}
    \chi_T^{-2}B_T^{\top}B_T
    \preceq g_K(\hat x)
    \preceq\chi_T^2B_T^{\top}B_T.
    \label{eq:complete-element-metric-comparison}
  \end{align}
  Its surface Jacobian and second derivative satisfy
  \begin{align}
    &J_{F_K}(\hat x)
    =J_T(y)J_{\Phi_T}
     =\frac{R^2d_T}{|y|^3}\,2|T|_2,
    \quad
    \chi_T^{-1}J_{\Phi_T}
    \leq J_{F_K}(\hat x)\leq\chi_T^2J_{\Phi_T},
    \notag\\
    &|D^2F_K(\hat x)[\hat\xi,\hat\eta]|
    \leq\frac{6\chi_T^2}{R}|B_T\hat\xi|\,|B_T\hat\eta|.
    \label{eq:complete-element-bounds}
  \end{align}
\end{theorem}

\begin{proof}
  The affine map $\Phi_T:\widehat T\to T$ is a homeomorphism, and Lemma~\ref{lem:radial-basic} shows that $\Psi_T:T\to K$ is a homeomorphism. Therefore, $F_K=\Psi_T\circ\Phi_T$ is a homeomorphism, and its inverse is obtained by composing \eqref{eq:affine-inverse-map} with \eqref{eq:radial-inverse-map}.

  The map $\Psi$ is smooth on a neighbourhood of $T$, and $\Phi_T$ is affine. Furthermore, $\nu_T\cdot x>0$ on $K$, so the displayed inverse is smooth on the interior of $K$ relative to $\Sph_R^2$. Thus, the restriction of $F_K$ to $\operatorname{int}\widehat T$ is a $\mathcal C^\infty$-diffeomorphism. Because $D\Phi_T=B_T$ and $D^2\Phi_T=0$, the first- and second-order chain rules give \eqref{eq:complete-map-differentials}.

  For any $\hat\xi\in\R^2$, \eqref{eq:radial-metric-bounds} gives
  \begin{align*}
    \chi_T^{-2}|B_T\hat\xi|^2
    \leq
    \hat\xi^{\top}g_K(\hat x)\hat\xi
    =|D\Psi_T(y)B_T\hat\xi|^2
    \leq
    \chi_T^2|B_T\hat\xi|^2.
  \end{align*}
  Because $B_T$ has rank two, this proves both $\operatorname{rank}DF_K(\hat x)=2$ and the Loewner comparison \eqref{eq:complete-element-metric-comparison}. Multiplicativity of the two-dimensional surface Jacobian under composition gives
  \begin{align*}
    J_{F_K}(\hat x)
    =J_T(\Phi_T(\hat x))J_{\Phi_T},
  \end{align*}
  which proves the Jacobian identity and bounds in \eqref{eq:complete-element-bounds}. Finally, the second-derivative identity and \eqref{eq:radial-second-derivative-bound} give
  \begin{align*}
    |D^2F_K(\hat x)[\hat\xi,\hat\eta]|
    \leq\frac{6\chi_T^2}{R}
    |B_T\hat\xi|\,|B_T\hat\eta|,
  \end{align*}
  proving the last estimate in \eqref{eq:complete-element-bounds}.
\end{proof}

\subsection{Geometric interface: maximum angles and radial locality}
\label{subsec:geometric-interface}
We connect the three geometric controls used later: the intrinsic and chordal maximum-angle parameters and the radial locality factor. We state both angular implications and show that relative refinement forces $\chi_T\to1$.

Let $\{\mathscr K_h\}_h$ be a family of finite collections of non-degenerate exact spherical triangles, and let $T_K$ be the chordal core of $K\in\mathscr K_h$.

For each $K$, let $\alpha_K$ be its largest intrinsic angle. The angle at a vertex is measured between the tangent directions of the two incident shorter great-circle arcs.
\begin{definition}[Intrinsic spherical maximum-angle condition]
  \label{def:intrinsic-spherical-maximum-angle}
  The family $\{\mathscr K_h\}_h$ satisfies the \emph{intrinsic spherical maximum-angle condition} if there is a constant $\delta\in(0,\pi)$, independent of $h$ and $K$, such that
  \begin{align}
    \alpha_K\leq\pi-\delta
    \quad\text{for any }K\in\mathscr K_h\text{ and any }h.
    \label{eq:intrinsic-spherical-angle-condition}
  \end{align}
\end{definition}

This condition can be checked from the vertices. Under Convention~\ref{conv:chordal-labelling}, the longest chordal edge is $[p_2,p_3]$. Chordal and geodesic edge lengths satisfy
\begin{align*}
  |p_i-p_j|=2R\sin\left(\frac{d_{\Sph}(p_i,p_j)}{2R}\right).
\end{align*}
Because $s\mapsto2R\sin(s/(2R))$ is strictly increasing on $[0,\pi R]$, the corresponding geodesic edge is also longest. By \cite[Lemma~3.2]{Ishizaka2026SphericalMaxAngle}, the largest intrinsic angle therefore occurs at $p_1$. Define
\begin{align}
  t_K^{\mathrm{sr}}
  :=\frac{R|\det(p_1,p_2,p_3)|}
  {|p_1\times p_2|\,|p_1\times p_3|},
  \quad
  \mu_K^{\mathrm{sr}}:=\frac2{t_K^{\mathrm{sr}}}.
  \label{eq:intrinsic-spherical-sr-interface}
\end{align}
Then, \cite[Theorem~3.4]{Ishizaka2026SphericalMaxAngle} gives
\begin{align}
  t_K^{\mathrm{sr}}=\sin\alpha_K,
  \quad
  \mu_K^{\mathrm{sr}}=\frac2{\sin\alpha_K}.
  \label{eq:intrinsic-spherical-maximum-angle-interface}
\end{align}
The triangles considered here are geodesically convex and lie in an open hemisphere. Their angles lie in $(0,\pi)$ and have sum greater than $\pi$, hence $\pi/3<\alpha_K<\pi$. The condition in Definition~\ref{def:intrinsic-spherical-maximum-angle} can therefore be written as
\begin{align}
  \inf_h\inf_{K\in\mathscr K_h}t_K^{\mathrm{sr}}
  \geq c_{\mathrm{sr}}>0.
  \label{eq:intrinsic-spherical-uniform-condition}
\end{align}
Indeed, \eqref{eq:intrinsic-spherical-angle-condition} gives this bound with $c_{\mathrm{sr}}=\min\{\sqrt3/2,\sin\delta\}$; conversely, \eqref{eq:intrinsic-spherical-uniform-condition} gives \eqref{eq:intrinsic-spherical-angle-condition} with $\delta=\arcsin c_{\mathrm{sr}}$.

We use \eqref{eq:intrinsic-spherical-uniform-condition} as the algebraic form of the \emph{intrinsic spherical maximum-angle condition}. Equivalently, there is a constant $\gamma_{\mathrm{sr}}\geq2$, independent of $h$ and $K$, such that
\begin{align*}
  \sup_h\sup_{K\in\mathscr K_h}\mu_K^{\mathrm{sr}}
  \leq\gamma_{\mathrm{sr}}.
\end{align*}

The affine interpolation estimates use the geometry of the chordal core. Let $\mathcal R_T$ denote the Euclidean circumradius of $T$, and we set $H_T:=4\mathcal R_T$. The extended sine rule gives $H_T=\mu_T h_T$. The circumradius formula therefore gives
\begin{align}
  H_T=\frac{h_{T,1}h_{T,2}}{|T|_2}h_T,
  \quad
  \mu_T=\frac{H_T}{h_T}
  =\frac{h_{T,1}h_{T,2}}{|T|_2}.
  \label{eq:chordal-geometric-parameters}
\end{align}
Because $\beta_T$ is the largest chordal angle and
\begin{align*}
  |T|_2=\frac12h_{T,1}h_{T,2}\sin\beta_T,
\end{align*}
we have
\begin{align}
  \mu_T=\frac2{\sin\beta_T}.
  \label{eq:mu-sine-identity}
\end{align}
Thus,
\begin{align}
  t_T=\sin\beta_T=\frac2{\mu_T},
  \quad
  \mu_Tt_T=2,
  \quad
  \mu_T\geq2.
  \label{eq:angular-factor-identity}
\end{align}
\begin{definition}[Chordal semi-regularity]
  \label{def:semiregularity}
  A family is \emph{chordally semi-regular} if there is a constant $\gamma_0$ such that $\mu_T\leq\gamma_0$ for all its chordal cores. By \eqref{eq:mu-sine-identity}, this is the maximum-angle condition for the planar triangles $T$.
\end{definition}
Thus, $\mu_T$ measures the chordal angle $\beta_T$, whereas $\mu_K^{\mathrm{sr}}$ measures the spherical angle $\alpha_K$. In statements involving $\mu_T$, ``semi-regularity'' means chordal semi-regularity. The sharp comparison from \cite[Theorem~4.5]{Ishizaka2026SphericalMaxAngle},
\begin{align}
  \mu_{T_K}\leq\frac2{\sqrt3}\mu_K^{\mathrm{sr}},
  \label{eq:spherical-chordal-comparison}
\end{align}
shows that the intrinsic spherical maximum-angle condition implies chordal semi-regularity.

Radial distortion is measured by $\chi_T=R/d_T$. The following identity expresses it in terms of $H_T$.

\begin{theorem}[Support-plane identity]
  \label{thm:support-plane-localisation}
  Let $c_T=d_T\nu_T$ be the orthogonal projection of the origin onto $\aff (T)$. Then,
  \begin{align}
    |p_i-c_T|&=\mathcal R_T=\frac{H_T}{4},
    \quad i=1,2,3,
    \notag\\
    d_T^2&=R^2-\frac{H_T^2}{16},
    \quad
    \chi_T=\left(1-\frac{H_T^2}{16R^2}\right)^{-1/2},
    \label{eq:support-plane-localisation}
  \end{align}
  where $\mathcal R_T$ is the Euclidean circumradius and $0<H_T<4R$.
\end{theorem}

\begin{proof}
Because $p_i-c_T\in V_T$ and $c_T\perp V_T$,
\begin{align*}
  |p_i-c_T|^2=R^2-d_T^2,
  \quad i=1,2,3.
\end{align*}
Thus, $c_T$ is the circumcentre of $T$. The circumradius formula gives
\begin{align*}
  \mathcal R_T
  =\frac{h_{T,1}h_{T,2}h_T}{4|T|_2}
  =\frac{H_T}{4}.
\end{align*}
Substituting this into $d_T^2+\mathcal R_T^2=R^2$ gives \eqref{eq:support-plane-localisation}, since $\chi_T=R/d_T$. Finally, $\mathcal R_T>0$ and $d_T>0$ give $0<H_T<4R$.
\end{proof}

If $y=\sum_{i=1}^3\lambda_ip_i\in T$, where $\lambda_i\geq0$ and $\sum_{i=1}^3\lambda_i=1$, convexity gives $|y-c_T|\leq\sum_i\lambda_i|p_i-c_T|=H_T/4$, while equality is attained at any vertex. Therefore,
\begin{align}
  \max_{y\in T}|y-c_T|=\frac{H_T}{4}.
  \label{eq:support-plane-distance-bound}
\end{align}
We set
\begin{align}
  \vartheta_T:=\frac{H_T}{4R},
  \quad
  0<\vartheta_T<1,
  \quad
  \chi_T=(1-\vartheta_T^2)^{-1/2}.
  \label{eq:locality-parameter}
\end{align}

\begin{definition}[Spherical locality]
  \label{def:spherical-locality}
  A family is called \emph{uniformly spherically local} if $\vartheta_T\leq\eta_0<1$ uniformly over its elements. From \eqref{eq:locality-parameter}, this is equivalent to a uniform bound on $\chi_T$.
\end{definition}

The following estimate follows from \cite[Theorem~4.3]{Ishizaka2026SphericalMaxAngle}. We include a short proof using the radial differential.

\begin{proposition}[Reverse angular comparison]
  \label{prop:reverse-angular-comparison}
  For any exact spherical triangle $K$ with chordal core $T$,
  \begin{align}
    \mu_K^{\mathrm{sr}}\leq\chi_T\mu_T.
    \label{eq:reverse-angular-comparison}
  \end{align}
  Consequently, if $\mu_T\leq\gamma_0$ and $\chi_T\leq\chi_0$ uniformly over a family, where $\gamma_0\geq2$ and $\chi_0\geq1$, then
  \begin{align}
    t_K^{\mathrm{sr}}\geq\frac2{\gamma_0\chi_0}>0,
    \quad
    \alpha_K\leq\pi-\arcsin\left(\frac2{\gamma_0\chi_0}\right).
    \label{eq:reverse-uniform-angle-bound}
  \end{align}
  In particular, chordal semi-regularity and uniform spherical locality imply the intrinsic spherical maximum-angle condition.
\end{proposition}

\begin{proof}
With the canonical labelling, we set
\begin{align*}
  A:=D\Psi(p_1)|_{V_T},
  \quad
  \xi_1:=p_2-p_1,
  \quad
  \xi_2:=p_3-p_1.
\end{align*}
The vectors $\xi_1,\xi_2$ form the angle $\beta_T$, while $A\xi_1,A\xi_2$ point along the two spherical edges leaving $p_1$ and form the angle $\alpha_K$. Because $|p_1|=R$, \eqref{eq:radial-singular-values} gives singular values $d_T/R$ and $1$ for $A$. Therefore,
\begin{align*}
  |A\xi_1\times A\xi_2|
  =\frac{d_T}{R}|\xi_1\times\xi_2|,
  \quad
  |A\xi_i|\leq|\xi_i|,\quad i=1,2.
\end{align*}
It follows that
\begin{align*}
  \sin\alpha_K
  &=\frac{|A\xi_1\times A\xi_2|}{|A\xi_1|\,|A\xi_2|}
  \geq\frac{d_T}{R}
    \frac{|\xi_1\times\xi_2|}{|\xi_1|\,|\xi_2|}
  =\frac{\sin\beta_T}{\chi_T}.
\end{align*}
Using $\mu_K^{\mathrm{sr}}=2/\sin\alpha_K$ and $\mu_T=2/\sin\beta_T$ gives \eqref{eq:reverse-angular-comparison}. The uniform bounds then give \eqref{eq:reverse-uniform-angle-bound}.
\end{proof}

The spherical maximum-angle condition gives uniform chordal angle control. Relative refinement also makes the radial distortion tend to one, as the following corollary shows.

\begin{corollary}[Intrinsic spherical criterion for the interpolation geometry]
  \label{cor:intrinsic-partII-geometric-interface}
  Suppose \eqref{eq:intrinsic-spherical-uniform-condition} holds. Then,
  \begin{align}
    \sup_h\max_{K\in\mathscr K_h}\mu_{T_K}
    \leq\frac4{\sqrt3\,c_{\mathrm{sr}}}.
    \label{eq:intrinsic-partII-chordal-control}
  \end{align}
  If also
  $q_h:=\max_{K\in\mathscr K_h}h_{T_K}/R\to0$, then
  \begin{align}
    \max_{K\in\mathscr K_h}\vartheta_{T_K}
    \leq\frac{q_h}{\sqrt3\,c_{\mathrm{sr}}}\longrightarrow0,
    \quad
    \max_{K\in\mathscr K_h}\chi_{T_K}\longrightarrow1.
    \label{eq:intrinsic-partII-radial-control}
  \end{align}
\end{corollary}

\begin{proof}
From \eqref{eq:spherical-chordal-comparison},
\begin{align*}
  \mu_{T_K}
  \leq\frac2{\sqrt3}\mu_K^{\mathrm{sr}}
  =\frac4{\sqrt3\,t_K^{\mathrm{sr}}}
  \leq\frac4{\sqrt3\,c_{\mathrm{sr}}}.
\end{align*}
Because $\vartheta_{T_K}=(\mu_{T_K}/4)(h_{T_K}/R)$, this gives the bound for $\vartheta_{T_K}$ in \eqref{eq:intrinsic-partII-radial-control}. The limit for $\chi_{T_K}$ follows from $\chi_{T_K}=(1-\vartheta_{T_K}^2)^{-1/2}$.
\end{proof}

\begin{corollary}[Automatic locality under chordal semi-regular refinement]
  \label{cor:refining-locality}
  Suppose
  \begin{align*}
    \sup_h\max_{K\in\mathscr K_h}\mu_{T_K}\leq\gamma_0.
  \end{align*}
  Then, for any $h$ and any $K\in\mathscr K_h$,
  \begin{align}
    \vartheta_{T_K}
    =\frac{\mu_{T_K}}4\frac{h_{T_K}}R
    \leq\frac{\gamma_0}{4}\frac{h_{T_K}}R.
    \label{eq:locality-refinement-bound}
  \end{align}
  Consequently, if
  $\max_{K\in\mathscr K_h}h_{T_K}/R\to0$, then
  \begin{align}
    \max_{K\in\mathscr K_h}\vartheta_{T_K}\longrightarrow0,
    \quad
    \max_{K\in\mathscr K_h}\chi_{T_K}\longrightarrow1.
    \label{eq:uniform-locality-convergence}
  \end{align}
\end{corollary}

\begin{proof}
  The equality in \eqref{eq:locality-refinement-bound} follows from $H_{T_K}=\mu_{T_K}h_{T_K}$, and its inequality follows from the uniform bound $\mu_{T_K}\leq\gamma_0$. Taking maxima over $K\in\mathscr K_h$ gives $\max_{K\in\mathscr K_h}\vartheta_{T_K}\leq(\gamma_0/4) \max_{K\in\mathscr K_h} h_{T_K}/R\to0$; continuity in \eqref{eq:locality-parameter} then gives $\max_{K\in\mathscr K_h}\chi_{T_K}\to1$.
\end{proof}

Corollary~\ref{cor:refining-locality} and Proposition~\ref{prop:reverse-angular-comparison} show that uniform chordal semi-regularity, together with $\max_{K\in\mathscr K_h}h_{T_K}/R\to0$, implies the intrinsic spherical maximum-angle condition for any sufficiently fine collections.

\begin{remark}[Need for locality]
Chordal semi-regularity alone does not control the spherical maximum angle. The equilateral family in Section~\ref{subsec:radial-exponent-sharpness} has support plane $x_3=\varepsilon$, circumcentre $(0,0,\varepsilon)$, and circumradius $\sqrt{R^2-\varepsilon^2}$. Its parameters satisfy
\begin{align*}
  \mu_{T_\varepsilon^{\mathrm{loc}}}=\frac4{\sqrt3},
  \quad
  \chi_{T_\varepsilon^{\mathrm{loc}}}=\frac R\varepsilon
  \longrightarrow\infty.
\end{align*}
The three spherical angles are equal, and their common value
$\alpha_\varepsilon$ satisfies
\begin{align*}
  \cos\alpha_\varepsilon
  =\frac{3\varepsilon^2-R^2}{R^2+3\varepsilon^2}
  \longrightarrow-1.
\end{align*}
Thus, $\alpha_\varepsilon\to\pi$ even though the chordal cores are all equilateral.
\end{remark}

\section{Local \texorpdfstring{$\mathbb P^1$}{P1} Lagrange interpolation on an exact spherical triangle}
\label{sec:local-interpolation}
We transfer the affine anisotropic estimates to the exact spherical element while keeping $h_{T,1}$, $h_{T,2}$, $\mu_T$, and $\chi_T$ explicit. The weighted estimates hold on any non-degenerate element with $d_T>0$. Uniform locality is needed only to bound the radial factors uniformly; the local $L^2$ estimate requires no semi-regularity assumption, whereas the standard $H^1$ rate requires uniform control of $\mu_T$.

The argument is organised in four stages. We first define the mapped finite-element space and identify the spherical interpolation error with the affine interpolation error on the chordal core. We then derive the sharp radial differential identities and the corresponding norm-transfer formulas. Next, we establish the affine-core anisotropic estimates. Finally, these ingredients are combined in Theorem~\ref{thm:final-local-anisotropic} and then rewritten in intrinsic spherical form.

\subsection{Notation for transported directional quantities}
Recall that $w=v\circ\Psi_T$ is the pullback of $v$ from $K$ to the chordal core $T$. We write $\tilde w:=w\circ\iota_T$ for its planar representative on $\widetilde T$. For a generic scalar function $u$ used in the norm-transfer lemma, we instead write $u_T:=u\circ\Psi_T$.

For a non-zero chordal direction $\xi\in V_T$, $(\tau_\xi,a_\xi,e_\xi)$ denotes its transported tangent vector, its length, and the corresponding unit tangent direction. The weight $\omega_\xi$ occurs in the exact first-directional $L^2$ transfer, and $\mathscr Q_{\xi,\eta}$ in the second-directional transfer. We write $\mathscr E_{ij}$ and $\mathscr G_i$ for the weighted norms used in the estimates, and $\mathscr H_{\xi,\eta}$ and $\mathscr P_\xi$ for the corresponding unweighted intrinsic Hessian and gradient data.

\subsection{Mapped polynomial space and nodal interpolant}
Recall the barycentric coordinates $\hat\lambda_1,\hat\lambda_2,\hat\lambda_3$ of $\widehat T$,
\begin{align*}
  \hat\lambda_1(\hat x)=1-\hat x_1-\hat x_2,
  \quad
  \hat\lambda_2(\hat x)=\hat x_1,
  \quad
  \hat\lambda_3(\hat x)=\hat x_2.
\end{align*}
For any integer $m\geq0$, the intrinsic polynomial space of degree at most $m$ on the chordal triangle is
\begin{align*}
  \mathbb P^m(T)
  :=
  \{\hat p\circ\Phi_T^{-1}:\ \hat p\in\mathbb P^m(\widehat T)\}.
\end{align*}
In particular, the barycentric basis of $\mathbb P^1(T)$ is
\begin{align}
  \lambda_i^T:=\hat\lambda_i\circ\Phi_T^{-1},
  \quad i=1,2,3,
  \label{eq:chordal-barycentric-functions}
\end{align}
and the corresponding mapped basis functions on the exact spherical triangle are
\begin{align}
  \lambda_i^K:=\hat\lambda_i\circ F_K^{-1}
  =\lambda_i^T\circ G_T,
  \quad i=1,2,3.
  \label{eq:spherical-barycentric-functions}
\end{align}
They satisfy
\begin{align}
  \lambda_i^K(p_j)=\delta_{ij},
  \quad
  \sum_{i=1}^3\lambda_i^K=1
  \quad\text{on }K.
  \label{eq:spherical-barycentric-properties}
\end{align}

\begin{definition}[Local mapped $\mathbb P^1$ space]
  The local mapped polynomial space on $K$ is
  \begin{align}
    \mathbb P_F^1(K)
    :=\{\hat p\circ F_K^{-1}:\ \hat p\in\mathbb P^1(\widehat T)\}
    =\operatorname{span}\{\lambda_1^K,\lambda_2^K,\lambda_3^K\}.
    \label{eq:mapped-linear-space}
  \end{align}
\end{definition}

For $v\in \mathcal{C}^0(K)$, the local spherical Lagrange interpolant is defined as
\begin{align}
  I_Kv:=\sum_{i=1}^3v(p_i)\lambda_i^K\in\mathbb P_F^1(K).
  \label{eq:nodal-interpolant}
\end{align}
Let $I_{\widehat T}$ and $I_T$ denote the ordinary nodal $\mathbb P^1$ interpolation operators on $\widehat T$ and $T$, respectively.

The following commuting relation reduces the spherical interpolation error exactly to interpolation on the chordal core.

\begin{proposition}[Commuting representation of the local interpolant]
  \label{prop:interpolant-commuting}
  For any $v\in \mathcal{C}^0(K)$,
  \begin{align}
    (I_Kv)\circ F_K=I_{\widehat T}(v\circ F_K),
    \label{eq:interpolant-commuting-reference}
  \end{align}
  and
  \begin{align}
    I_Kv
    =\bigl(I_T(v\circ\Psi_T)\bigr)\circ G_T.
    \label{eq:interpolant-commuting-chordal}
  \end{align}
  Consequently,
  \begin{align}
    (v-I_Kv)\circ\Psi_T
    =w-I_Tw,
    \quad
    w=v\circ\Psi_T.
    \label{eq:interpolation-error-pullback}
  \end{align}
  In particular, the three vertex values are unisolvent for $\mathbb P_F^1(K)$, and $I_Kv=v$ for any $v\in\mathbb P_F^1(K)$.
\end{proposition}

\begin{proof}
Because $\lambda_i^K\circ F_K=\hat\lambda_i$ and $F_K(\hat p_i)=p_i$, the reference-element expansion gives
\begin{align*}
  (I_Kv)\circ F_K
  &=
  \sum_{i=1}^3v(p_i)(\lambda_i^K\circ F_K)
  =
  \sum_{i=1}^3(v\circ F_K)(\hat p_i)\hat\lambda_i
  =
  I_{\widehat T}(v\circ F_K).
\end{align*}
Similarly, $\lambda_i^K=\lambda_i^T\circ G_T$ and $\Psi_T(p_i)=p_i$ imply
\begin{align*}
  \bigl(I_T(v\circ\Psi_T)\bigr)\circ G_T
  &=
  \sum_{i=1}^3v(\Psi_T(p_i))(\lambda_i^T\circ G_T)
  =
  \sum_{i=1}^3v(p_i)\lambda_i^K
  =
  I_Kv.
\end{align*}
Composing the second identity with $\Psi_T$ and using $G_T\circ\Psi_T=\operatorname{id}_T$ gives \eqref{eq:interpolation-error-pullback}. 

If $q\in\mathbb P_F^1(K)$ vanishes at the three vertices, then $q\circ F_K\in\mathbb P^1(\widehat T)$ vanishes at the three reference vertices and is therefore zero. Because $F_K$ is bijective, it follows that $q=0$ on $K$. This proves unisolvence. Finally, if $v\in\mathbb P_F^1(K)$, then $v-I_Kv\in\mathbb P_F^1(K)$ and vanishes at the three vertices. Unisolvence therefore gives $I_Kv=v$.
\end{proof}

\begin{remark}[Relation to the standard surface finite-element lift]
  For $y\neq0$, $Ry/|y|$ is the unique closest point to $y$ on $\Sph_R^2$. Therefore, $\Psi_T$ is the closest-point projection restricted to the chordal triangle, and
  \begin{align*}
    \mathbb P_F^1(K)
    =
    \{p_T\circ G_T:p_T\in\mathbb P^1(T)\}
  \end{align*}
  is the usual closest-point lift of the affine $\mathbb P^1$ space on $T$; compare \cite{DziukElliott2013}. The estimates below quantify its dependence on the directional scales, semi-regularity, and radial distortion. If the variational problem and its integrals are posed on the exact sphere, no polyhedral-surface consistency term arises; quadrature error is a separate issue.
\end{remark}

\subsection{Sharp differential identities for the radial correction}
The distortion bounds in Section~\ref{sec:geometry} are sufficient for regularity and norm comparison, but the interpolation analysis requires a sharper second-order identity whose curvature contribution vanishes in the local flat limit.

Let $v\in\mathcal C^2(K)$ and set $w=v\circ\Psi_T$.  For $y\in T$, we define
\begin{align}
  x:=\Psi_T(y),
  \quad
  z:=y-c_T,
  \quad
  r:=|y|.
  \label{eq:radial-local-variables}
\end{align}
Because $y\in\aff (T)=c_T+V_T$, one has $z\in V_T$ and $y=c_T+z$.  Furthermore, \eqref{eq:support-plane-distance-bound} gives
\begin{align}
  |z|
  \leq
  \frac{H_T}{4}.
  \label{eq:radial-offset-bound}
\end{align}
For $\xi\in V_T$, we define
\begin{align*}
  \tau_\xi(y):=D\Psi_T(y)\xi,
  \quad
  a_\xi(y):=|\tau_\xi(y)|.
\end{align*}
If $\xi\neq0$, then $a_\xi(y)>0$, and we define
\begin{align*}
  e_\xi(x):=\frac{\tau_\xi(y)}{a_\xi(y)}.
\end{align*}
By \eqref{eq:radial-metric-bounds},
\begin{align*}
  a_\xi(y)
  =
  |D\Psi_T(y)\xi|
  \geq
  \chi_T^{-1}|\xi|
  >0,
\end{align*}
so $e_\xi(x)$ is well defined. Furthermore, $\tau_\xi(y)\in T_x\Sph_R^2$, and hence $e_\xi(x)$ is a unit tangent vector at $x$. Because $\Psi_T:T\to K$ is bijective, $y=G_T(x)$ is uniquely determined by $x$, so $e_\xi$ is well defined on $K$. We write
\begin{align*}
  \partial_{e_\xi}^{\Sph}v(x)
  :=
  \nabla_{\Sph}v(x)\cdot e_\xi(x).
\end{align*}

\begin{lemma}[Tangential second derivative of the radial map]
  \label{lem:tangential-second-radial}
  Let $y\in T$, and let $x,z,r$ be defined by \eqref{eq:radial-local-variables}. Recall from Section~\ref{subsec:spherical-calculus} that
\begin{align*}
  P(x)=I-n(x)n(x)^{\top}.
\end{align*}
 Then, for any $\xi,\eta\in V_T$,
  \begin{align}
    P(x)D^2\Psi_T(y)[\xi,\eta]
    =
    -\frac{1}{r^2}
    \left(
      (z\cdot\eta)\tau_\xi(y)
      +(z\cdot\xi)\tau_\eta(y)
    \right).
    \label{eq:tangential-second-derivative-identity}
  \end{align}
  Consequently,
  \begin{align}
    |P(x)D^2\Psi_T(y)[\xi,\eta]|
    \leq
    \frac{H_T\chi_T^3}{2R^2}
    |\xi|\,|\eta|.
    \label{eq:tangential-second-derivative-bound}
  \end{align}
\end{lemma}

\begin{proof}
Because $x=\Psi_T(y)=Ry/r$, one has
\begin{align*}
  n(x)=\frac{y}{r},
 \end{align*}
and hence
\begin{align*}
  P(x)=I-\frac{yy^{\top}}{r^2}=P_y.
 \end{align*}
  Applying $P_y$ to \eqref{eq:radial-second-derivative} and using $P_yy=0$ gives
  \begin{align*}
    P(x)D^2\Psi_T(y)[\xi,\eta]
    =
    -\frac{R}{r^3}
    \left(
      (y\cdot\eta)P_y\xi
      +(y\cdot\xi)P_y\eta
    \right).
  \end{align*}
  Because $y=c_T+z$, $c_T=d_T\nu_T$, and $\xi,\eta\in V_T=\nu_T^\perp$,
  \begin{align*}
    y\cdot\xi=z\cdot\xi,
    \quad
    y\cdot\eta=z\cdot\eta.
  \end{align*}
  Furthermore,
  \begin{align*}
    \tau_\xi(y)
    =
    D\Psi_T(y)\xi
    =
    \frac{R}{r}P_y\xi,
  \end{align*}
  and similarly for $\tau_\eta(y)$. Therefore, \eqref{eq:tangential-second-derivative-identity} follows.

  Taking norms gives
  \begin{align*}
    |P(x)D^2\Psi_T(y)[\xi,\eta]|
    \leq
    \frac{|z|}{r^2}
    \left(
      |\eta|\,|\tau_\xi(y)|
      +
      |\xi|\,|\tau_\eta(y)|
    \right)
    \leq
    \frac{2R|z|}{r^3}
    |\xi|\,|\eta|.
  \end{align*}
  From \eqref{eq:radial-offset-bound},
  \begin{align*}
    |z|\leq\frac{H_T}{4},
  \end{align*}
  while
  \begin{align*}
    r\geq d_T=\frac{R}{\chi_T}.
  \end{align*}
  Therefore,
  \begin{align*}
    |P(x)D^2\Psi_T(y)[\xi,\eta]|
    \leq
    \frac{H_T\chi_T^3}{2R^2}
    |\xi|\,|\eta|,
  \end{align*}
  which proves \eqref{eq:tangential-second-derivative-bound}.
\end{proof}

\begin{lemma}[Exact first and second directional derivatives]
  \label{lem:exact-directional-derivatives}
  Let $v\in\mathcal C^2(K)$ and set $w=v\circ\Psi_T$. Let $y\in T$, and let $x,z,r$ be defined by \eqref{eq:radial-local-variables}. Let $\xi,\eta\in V_T\setminus\{0\}$. Then,
  \begin{align}
    \partial_\xi w(y)
    =
    a_\xi(y)\partial_{e_\xi}^{\Sph}v(x),
    \label{eq:radial-first-directional-derivative}
  \end{align}
  and
  \begin{align}
    D_T^2w(y)[\xi,\eta]
    &=
    a_\xi(y)a_\eta(y)
    \nabla_{\Sph}^2v(x)[e_\xi(x),e_\eta(x)]
    \nonumber \\
    &\quad -\frac{1}{r^2}
    \left(
      (z\cdot\eta)a_\xi(y)\partial_{e_\xi}^{\Sph}v(x)
      +(z\cdot\xi)a_\eta(y)\partial_{e_\eta}^{\Sph}v(x)
    \right).
    \label{eq:radial-second-directional-derivative}
  \end{align}
  In particular,
  \begin{align}
    |D_T^2w(y)[\xi,\eta]|
    \leq
    \left(
      \chi_T^2|\nabla_{\Sph}^2v(x)|
      +
      \frac{H_T\chi_T^3}{2R^2}|\nabla_{\Sph}v(x)|
    \right)
    |\xi|\,|\eta|.
    \label{eq:pullback-hessian-bound}
  \end{align}
\end{lemma}

\begin{proof}
The first identity is the chain rule $\partial_\xi w=\nabla_{\Sph}v\cdot D\Psi_T\xi$. To justify the second identity, work locally in the interior of $K$ and take the $0$-homogeneous extension
\begin{align*}
  v^{\mathrm e}(x):=v(Rx/|x|)
\end{align*}
on the cone over a spherical neighbourhood contained in the interior of $K$. Because $x\cdot\nabla v^{\mathrm e}(x)=0$, its gradient is tangential to the concentric spheres and agrees with $\nabla_{\Sph}v$ on that neighbourhood. Differentiating along tangent vectors and projecting onto the tangent plane gives
\begin{align*}
  \nabla_{\Sph}^2v(x)[\tau,\eta]
  =D^2v^{\mathrm e}(x)[\tau,\eta],
  \quad \tau,\eta\in T_x\Sph_R^2.
\end{align*}
The Euclidean chain rule for $w=v^{\mathrm e}\circ\Psi_T$ therefore gives
\begin{align*}
  D_T^2w[\xi,\eta]
  =\nabla_{\Sph}^2v[\tau_\xi,\tau_\eta]
   +\nabla_{\Sph}v\cdot D^2\Psi_T[\xi,\eta].
\end{align*}
Because $\nabla_{\Sph}v$ is tangential,
\begin{align*}
  \nabla_{\Sph}v\cdot D^2\Psi_T[\xi,\eta]
  =
  \nabla_{\Sph}v\cdot P(x)D^2\Psi_T[\xi,\eta],
\end{align*}
so Lemma~\ref{lem:tangential-second-radial} applies to the second term. Substituting $\tau_\xi=a_\xi e_\xi$ and $\tau_\eta=a_\eta e_\eta$ yields \eqref{eq:radial-second-directional-derivative}. The bounds
\begin{align*}
  |\nabla_{\Sph}^2v[e_\xi,e_\eta]|
  &\leq|\nabla_{\Sph}^2v|,
  \quad
  |\tau_\zeta(y)|
  \leq\chi_T|\zeta|,
  \quad \zeta\in V_T,
\end{align*}
together with \eqref{eq:tangential-second-derivative-bound}, give \eqref{eq:pullback-hessian-bound}.
\end{proof}

The identities in Lemma~\ref{lem:exact-directional-derivatives} were written first for $v\in\mathcal C^2(K)$. They also have the Sobolev interpretation used below. Indeed, $\Psi_T$ and $G_T$ are the restrictions of smooth maps defined on neighbourhoods of the closed elements, and all their derivatives up to order two are bounded because $d_T>0$. Therefore, pullback by $\Psi_T$ maps $H^m(K)$ continuously into $H^m(T)$ for $m=1,2$. Smooth functions are dense in these intrinsic Sobolev spaces, so the first identity extends almost everywhere to $v\in H^1(K)$ and the second identity extends, in the weak-derivative sense, to $v\in H^2(K)$. This density argument also justifies all directional transfer identities below at their stated Sobolev regularity.

The improvement over the earlier $O(R^{-1})$ bound is essential. Under a uniform spherical-locality bound, the tangential component of $D^2\Psi_T$ is $O(H_T/R^2)$. Therefore, for the fixed sphere radius $R$, it vanishes in the local flat limit $H_T/R\to0$, while the leading second-order term retains the intrinsic covariant Hessian of $v$.

\subsection{Exact energy splitting and directional-derivative transfer}
We now convert the pointwise identities of the preceding subsection into exact $L^2$ and $H^1$ transfer formulas, first for the energy and then for the directional derivatives used in the interpolation estimates.

For a scalar function $u$ on $K$, set $u_T:=u\circ\Psi_T$. We use
\begin{align*}
  |u|_{H^1(K)}:=\|\nabla_{\Sph}u\|_{L^2(K)},
  \quad
  |u_T|_{H^1(T)}:=\|\nabla_Tu_T\|_{L^2(T)}.
\end{align*}

\begin{lemma}[Local norm transfer and exact energy splitting]
  \label{lem:sharp-norm-transfer}
  Let $u\in H^1(K)$ and set
  $u_T:=u\circ\Psi_T\in H^1(T)$. Then,
  \begin{align}
    \chi_T^{-1/2}\|u_T\|_{L^2(T)}
    \leq
    \|u\|_{L^2(K)}
    \leq
    \chi_T\|u_T\|_{L^2(T)},
    \label{eq:L2-norm-transfer}
  \end{align}
  and
  \begin{align}
    \chi_T^{-1/2}|u_T|_{H^1(T)}
    \leq
    |u|_{H^1(K)}
    \leq
    \chi_T^{1/2}|u_T|_{H^1(T)}.
    \label{eq:H1-seminorm-transfer}
  \end{align}
  More precisely, for almost every $y\in T$, we set
  \begin{align*}
    x:=\Psi_T(y),
    \quad
    r:=|y|,
    \quad
    z:=y-c_T,
    \quad
    g:=\nabla_Tu_T(y).
  \end{align*}
  If $z\neq0$, we define
  \begin{align*}
    e_{\mathrm{rad}}(y)
    :=
    \frac{z}{|z|}
  \end{align*}
  and we choose a unit vector
  $e_{\mathrm{tan}}(y)\in V_T$ orthogonal to
  $e_{\mathrm{rad}}(y)$. Then,
  \begin{align}
    J_T(y)|\nabla_{\Sph}u(x)|^2
    =
    \frac{r}{d_T}
    |g\cdot e_{\mathrm{rad}}(y)|^2
    +
    \frac{d_T}{r}
    |g\cdot e_{\mathrm{tan}}(y)|^2.
    \label{eq:pointwise-energy-splitting}
  \end{align}
  If $z=0$, the same identity holds after replacing
  $e_{\mathrm{rad}},e_{\mathrm{tan}}$ by any orthonormal basis of
  $V_T$.
\end{lemma}

\begin{proof}
The area formula gives
\begin{align*}
  \|u\|_{L^2(K)}^2
  =
  \int_TJ_T(y)|u_T(y)|^2\,\diff s(y),
\end{align*}
so \eqref{eq:radial-jacobian-bounds} yields \eqref{eq:L2-norm-transfer}. For the energy identity, we set
\begin{align*}
  A_y:=D\Psi_T(y)|_{V_T},
  \quad
  q:=\nabla_{\Sph}u(x),
  \quad
  g:=\nabla_Tu_T(y)=A_y^{\top}q.
\end{align*}
Suppose first that $z\neq0$. In the orthonormal right-singular basis $(e_{\mathrm{rad}},e_{\mathrm{tan}})$ from the proof of Lemma~\ref{lem:metric-distortion}, the corresponding singular values are
\begin{align*}
  s_{\mathrm{rad}}=\frac{Rd_T}{r^2},
  \quad
  s_{\mathrm{tan}}=\frac{R}{r}.
\end{align*}
Let $(f_{\mathrm{rad}},f_{\mathrm{tan}})$ be the corresponding orthonormal left-singular basis, so that
\begin{align*}
  A_ye_{\mathrm{rad}}
  &=s_{\mathrm{rad}}f_{\mathrm{rad}},
  \quad
  A_ye_{\mathrm{tan}}
  =s_{\mathrm{tan}}f_{\mathrm{tan}}.
\end{align*}
Because $g=A_y^{\top}q$,
\begin{align*}
  g\cdot e_{\mathrm{rad}}
  &=
  s_{\mathrm{rad}}\,q\cdot f_{\mathrm{rad}},
  \quad
  g\cdot e_{\mathrm{tan}}
  =
  s_{\mathrm{tan}}\,q\cdot f_{\mathrm{tan}}.
\end{align*}
Using $J_T=s_{\mathrm{rad}}s_{\mathrm{tan}}=R^2d_T/r^3$, we obtain
\begin{align*}
  J_T|q|^2
  &=
  \frac{J_T}{s_{\mathrm{rad}}^2}
  |g\cdot e_{\mathrm{rad}}|^2
  +
  \frac{J_T}{s_{\mathrm{tan}}^2}
  |g\cdot e_{\mathrm{tan}}|^2
  =
  \frac{r}{d_T}|g\cdot e_{\mathrm{rad}}|^2
  +
  \frac{d_T}{r}|g\cdot e_{\mathrm{tan}}|^2,
\end{align*}
which is \eqref{eq:pointwise-energy-splitting}. Because
\begin{align*}
  \chi_T^{-1}\le \frac{d_T}{r}\le1\le\frac r{d_T}\le\chi_T,
\end{align*}
integration proves \eqref{eq:H1-seminorm-transfer}. If $z=0$, then $r=d_T$ and both singular values are equal to $R/d_T$. Therefore, \eqref{eq:pointwise-energy-splitting} holds with any orthonormal basis of $V_T$.
\end{proof}

The factors in \eqref{eq:H1-seminorm-transfer} follow from the exact energy-density eigenvalues; they are not obtained by separately combining the extremal singular values of $D\Psi_T$ and the Jacobian.

\begin{lemma}[First directional derivative transfer]
  \label{lem:first-direction-transfer}
  Let $v\in H^1(K)$, set $w=v\circ\Psi_T$, and let $\xi\in V_T$ be a unit vector. For $x\in K$, let $y:=G_T(x)$ and define
\begin{align}
  \omega_\xi(x)
  :=
  \frac{a_\xi(y)^2}{J_T(y)}.
  \label{eq:first-direction-weight}
\end{align}
Then,
  \begin{align}
    \|\partial_\xi w\|_{L^2(T)}^2
    =\int_K\omega_\xi(x)
      |\partial_{e_\xi}^{\Sph}v(x)|^2\,\diff\sigma(x),
    \label{eq:first-direction-transfer-identity}
  \end{align}
  and
  \begin{align}
    \chi_T^{-1}\leq\omega_\xi(x)\leq\chi_T.
    \label{eq:first-direction-weight-bounds}
  \end{align}
  Consequently,
  \begin{align}
    \chi_T^{-1/2}\|\partial_{e_\xi}^{\Sph}v\|_{L^2(K)}
    \leq\|\partial_\xi w\|_{L^2(T)}
    \leq\chi_T^{1/2}\|\partial_{e_\xi}^{\Sph}v\|_{L^2(K)}.
    \label{eq:first-direction-L2-transfer}
  \end{align}
\end{lemma}

\begin{proof}
Equation \eqref{eq:radial-first-directional-derivative} and the area formula
give
\begin{align*}
  \|\partial_\xi w\|_{L^2(T)}^2
  &=
  \int_Ta_\xi(y)^2
  |\partial_{e_\xi}^{\Sph}v(\Psi_T(y))|^2\,\diff s(y)
  =
  \int_K\frac{a_\xi(G_T(x))^2}{J_T(G_T(x))}
  |\partial_{e_\xi}^{\Sph}v(x)|^2\,\diff\sigma(x),
\end{align*}
which is \eqref{eq:first-direction-transfer-identity}. For $|\xi|=1$, $y=d_T\nu_T+z$ implies
\begin{align*}
  \omega_\xi
  =\frac r{d_T}\left(1-\frac{(z\cdot\xi)^2}{r^2}\right).
\end{align*}
Because the bracket is at most one and $r\leq R$,
\begin{align*}
  \omega_\xi
  \leq\frac{r}{d_T}
  \leq\frac{R}{d_T}
  =\chi_T.
\end{align*}
On the other hand,
\begin{align*}
  (z\cdot\xi)^2
  \leq|z|^2
  =r^2-d_T^2,
\end{align*}
and hence
\begin{align*}
  \omega_\xi
  \geq
  \frac{r}{d_T}\frac{d_T^2}{r^2}
  =
  \frac{d_T}{r}
  \geq
  \frac{d_T}{R}
  =
  \chi_T^{-1}.
\end{align*}
This proves \eqref{eq:first-direction-weight-bounds} and hence \eqref{eq:first-direction-L2-transfer}.
\end{proof}

For a scalar function $f$ on $K$, we define the weighted norm
\begin{align}
  \|f\|_{0,K;T}
  :=\left(
    \int_K\frac{|f(x)|^2}{J_T(G_T(x))}\,\diff\sigma(x)
  \right)^{1/2}.
  \label{eq:weighted-surface-norm}
\end{align}
From the area formula, this is exactly the pullback norm
\begin{align}
  \|f\|_{0,K;T}
  =
  \|f\circ\Psi_T\|_{L^2(T)}.
  \label{eq:weighted-surface-norm-pullback}
\end{align}
For $v\in H^2(K)$ and unit vectors $\xi,\eta\in V_T$, we define, for almost every $x\in K$,
\begin{align}
  \mathscr Q_{\xi,\eta}v(x)
  :=a_\xi a_\eta
    \nabla_{\Sph}^2v[e_\xi,e_\eta]
  -\frac{1}{r^2}
    \left(
      (z\cdot\eta)a_\xi\partial_{e_\xi}^{\Sph}v
      +(z\cdot\xi)a_\eta\partial_{e_\eta}^{\Sph}v
    \right),
  \label{eq:transported-second-direction-quantity}
\end{align}
where $r$, $z$, $a_\xi$, and $a_\eta$ are evaluated at $y=G_T(x)$, while the spherical derivatives and transported directions are evaluated at $x$.

\begin{lemma}[Second directional derivative transfer with curvature correction]
  \label{lem:second-direction-transfer}
  Let $v\in H^2(K)$, set $w=v\circ\Psi_T$, and let $\xi,\eta\in V_T$ be unit vectors. Then,
  \begin{align}
    \|\partial_\xi\partial_\eta w\|_{L^2(T)}
    =\|\mathscr Q_{\xi,\eta}v\|_{0,K;T}.
    \label{eq:second-direction-transfer-identity}
  \end{align}
  Furthermore, with
  \begin{align}
    \kappa_T:=\frac{H_T\chi_T^2}{4R^2},
    \label{eq:curvature-transfer-coefficient}
  \end{align}
  one has
  \begin{align}
    \|\mathscr Q_{\xi,\eta}v\|_{0,K;T}
    \leq
    \chi_T
    \|\nabla_{\Sph}^2v[e_\xi,e_\eta]\|_{L^2(K)}
    +\kappa_T
    \left(
      \|\partial_{e_\xi}^{\Sph}v\|_{L^2(K)}
      +\|\partial_{e_\eta}^{\Sph}v\|_{L^2(K)}
    \right).
    \label{eq:second-direction-transfer-bound}
  \end{align}
\end{lemma}

\begin{proof}
By \eqref{eq:radial-second-directional-derivative},
\begin{align*}
  \partial_\xi\partial_\eta w(y)
  =
  \mathscr Q_{\xi,\eta}v(\Psi_T(y))
\end{align*}
for almost every $y\in T$. The area formula and \eqref{eq:weighted-surface-norm-pullback} therefore give \eqref{eq:second-direction-transfer-identity}.

Because $\xi$ and $\eta$ are unit vectors, \eqref{eq:radial-first-derivative} gives
\begin{align*}
  a_\xi\leq\frac{R}{r},
  \quad
  a_\eta\leq\frac{R}{r}.
\end{align*}
Furthermore, \eqref{eq:radial-surface-jacobian} gives
\begin{align*}
  J_T^{-1/2}
  =
  \frac{r^{3/2}}{R\sqrt{d_T}}.
\end{align*}
For the Hessian term,
\begin{align*}
  J_T^{-1/2}a_\xi a_\eta
  \leq \frac{R}{\sqrt{d_Tr}}
  \leq \chi_T.
\end{align*}
For either first-derivative term, \eqref{eq:radial-offset-bound} gives
\begin{align*}
  J_T^{-1/2}\frac{|z\cdot\eta|a_\xi}{r^2}
  \le\frac{|z|}{\sqrt{d_T}r^{3/2}}
  \le\frac{H_T}{4d_T^2}=\kappa_T,
\end{align*}
and similarly with $\xi$ and $\eta$ interchanged. The triangle inequality proves \eqref{eq:second-direction-transfer-bound}.
\end{proof}

The coefficient $\kappa_T$ measures the contribution of the nonlinearity of the radial closest-point map. Its dimensionless refinement dependence is displayed by
\begin{align*}
  \kappa_T
  =
  \frac{\chi_T^2}{4R}\frac{H_T}{R}.
\end{align*}
Under a uniform spherical-locality bound, it therefore vanishes as $H_T/R\to0$ on the fixed sphere.

\subsection{Affine-core anisotropic estimates}
The estimates below are intrinsic to the affine plane $c_T+V_T$. The isometry $\iota_T:\widetilde T\to T$ identifies them with the corresponding planar estimates. Recall that the unit chordal edge directions satisfy
\begin{align}
  \ell_{T,1}=q_{T,1},
  \quad
  \ell_{T,2}=\frac{p_3-p_1}{h_{T,2}}
  =s_Tq_{T,1}+t_Tq_{T,2}.
  \label{eq:affine-core-edge-directions}
\end{align}
They need not be orthogonal; $q_{T,1},q_{T,2}$ remain an orthonormal basis used to resolve the intrinsic $H^1$-seminorm. These are intrinsic edge directions and need not be singular-vector directions of the affine map. The isometry $\iota_T:\widetilde T\to T$ introduces no additional geometric factor.

The $L^2$ estimate below follows directly from the affine anisotropic interpolation theory. For the $H^1$ estimate, we keep the reference derivative components separate; this is what confines the maximum-angle factor $\mu_T$ to the first directional scale.

\begin{lemma}[Affine-core $\mathbb P^1$ interpolation]
  \label{lem:affine-core-interpolation}
  There exist constants $C_{\mathrm{aff},0},C_{\mathrm{aff},1}>0$, depending only on the fixed reference triangle and independent of the element size and shape, such that every $w\in H^2(T)$ satisfies
  \begin{align}
    \|w-I_Tw\|_{L^2(T)}
    \leq C_{\mathrm{aff},0}
    \sum_{i,j=1}^2h_{T,i}h_{T,j}
    \|\partial_{\ell_{T,i}}\partial_{\ell_{T,j}}w\|_{L^2(T)}.
    \label{eq:affine-L2}
  \end{align}
  Furthermore, with
  \begin{align*}
    H_{T,1}^{\mathrm{dir}}:=h_{T,1},
    \quad
    H_{T,2}^{\mathrm{dir}}:=h_{T,2}t_T,
  \end{align*}
  one has
  \begin{align}
    |w-I_Tw|_{H^1(T)}
    \leq C_{\mathrm{aff},1}\mu_T
    \sum_{i=1}^2H_{T,i}^{\mathrm{dir}}
    |\partial_{\ell_{T,i}}w|_{H^1(T)}.
    \label{eq:affine-H1-directional-scales}
  \end{align}
  Equivalently, because $\mu_Tt_T=2$,
  \begin{align}
    |w-I_Tw|_{H^1(T)}
    \leq C_{\mathrm{aff},1}
    \left(
      \mu_Th_{T,1}|\partial_{\ell_{T,1}}w|_{H^1(T)}
      +2h_{T,2}|\partial_{\ell_{T,2}}w|_{H^1(T)}
    \right).
    \label{eq:affine-H1-sharp}
  \end{align}
\end{lemma}

\begin{proof}
  Because $T$ is two-dimensional, $H^2(T)\hookrightarrow\mathcal C^0(T)$ and the nodal interpolant is well defined. We set $\tilde w:=w\circ\iota_T\in H^2(\widetilde T)$, and let $I_{\widetilde T}$ denote the ordinary nodal $\mathbb P^1$ interpolation operator on $\widetilde T$. We define
  \begin{align*}
    \tilde\ell_{T,1}:=(1,0)^{\top},
    \quad
    \tilde\ell_{T,2}:=(s_T,t_T)^{\top}.
  \end{align*}
  The isometry $\iota_T(z)=p_1+Q_Tz$ maps the vertices and these two unit directions of $\widetilde T$ onto those of $T$. Unisolvence, the chain rule, and $Q_T^{\top}Q_T=I_2$ therefore give
  \begin{align}
    (I_Tw)\circ\iota_T &=I_{\widetilde T}\tilde w,
    \quad
    \|w-I_Tw\|_{L^2(T)}
      =\|\tilde w-I_{\widetilde T}\tilde w\|_{L^2(\widetilde T)},
    \notag\\
    |w-I_Tw|_{H^1(T)}
     &=|\tilde w-I_{\widetilde T}\tilde w|_{H^1(\widetilde T)},
    \quad
    |\partial_{\ell_{T,i}}w|_{H^1(T)}
      =|\partial_{\tilde\ell_{T,i}}\tilde w|_{H^1(\widetilde T)},
    \notag\\
    \|\partial_{\ell_{T,i}}\partial_{\ell_{T,j}}w\|_{L^2(T)}
      &=\|\partial_{\tilde\ell_{T,i}}
        \partial_{\tilde\ell_{T,j}}\tilde w\|_{L^2(\widetilde T)}.
    \label{eq:affine-core-isometry-summary}
  \end{align}
  Indeed, $\iota_T$ maps the vertices of $\widetilde T$ onto the corresponding vertices of $T$. Therefore, $(I_Tw)\circ\iota_T$ and $I_{\widetilde T}\tilde w$ are affine functions on $\widetilde T$ with the same three nodal values, and unisolvence proves the first identity. Applying \eqref{eq:isometric-L2-transfer} and \eqref{eq:isometric-gradient-norm} to the resulting interpolation-error identity gives the $L^2$- and $H^1$-isometries. Furthermore,
\begin{align*}
  Q_T\tilde\ell_{T,i}=\ell_{T,i},
  \quad i=1,2.
\end{align*}
Therefore, \eqref{eq:isometric-directional-transfer} and \eqref{eq:isometric-hessian-transfer}, followed by the corresponding $L^2$- and $H^1$-isometries, give the last two identities.  These relations hold first for smooth functions and then for $H^2$ functions by density.

  For the $L^2$ estimate, apply Theorem~B, equation~(58), and Corollary~1 of \cite{IshizakaKobayashiTsuchiya2023} with $d=2$, $k=1$, $\ell=2$, $p=q=2$, and $m=0$. The edge lengths and directions are $h_{T,i}$ and $\tilde\ell_{T,i}$, respectively. Thus,
  \begin{align*}
    \|\tilde w-I_{\widetilde T}\tilde w\|_{L^2(\widetilde T)}
    \leq C_{\mathrm{aff},0}
    \sum_{i,j=1}^2h_{T,i}h_{T,j}
    \|\partial_{\tilde\ell_{T,i}}
      \partial_{\tilde\ell_{T,j}}\tilde w\|_{L^2(\widetilde T)}.
  \end{align*}
  The identities in \eqref{eq:affine-core-isometry-summary} prove \eqref{eq:affine-L2}.

For the $H^1$ estimate, we work componentwise on the fixed reference triangle using the affine matrix $A_{\widetilde T}$ defined in Section~\ref{sec:geometry}. We set
\begin{align*}
  J_{\mathrm{aff}}:=\det A_{\widetilde T}=h_{T,1}h_{T,2}t_T,
  \quad
  W:=\tilde w\circ A_{\widetilde T},
  \quad E:=W-I_{\widehat T}W.
\end{align*}
For the reference coordinates $(\hat x_1,\hat x_2)$, write $W_i:=\partial_{\hat x_i}W$, $W_{ij}:=\partial_{\hat x_i}\partial_{\hat x_j}W$, and $E_i:=\partial_{\hat x_i}E$. The two-dimensional componentwise interpolation estimate of
\cite[Lemma~2.4]{Apel1999}, specialised to linear Lagrange interpolation, gives
\begin{align}
  \|E_1\|_{L^2(\widehat T)}
  &\leq C_{\mathrm{ref}}
    \bigl(\|W_{11}\|_{L^2(\widehat T)}
          +\|W_{12}\|_{L^2(\widehat T)}\bigr),
  \notag\\
  \|E_2\|_{L^2(\widehat T)}
  &\leq C_{\mathrm{ref}}
    \bigl(\|W_{12}\|_{L^2(\widehat T)}
          +\|W_{22}\|_{L^2(\widehat T)}\bigr).
  \label{eq:affine-reference-component-bounds}
\end{align}
For a direct proof, let $\widehat e_1$ and $\widehat e_2$ be the coordinate edges of $\widehat T$. Since $E$ vanishes at the vertices,
\begin{align*}
  \int_{\widehat e_1}E_1\,\diff s=0,
  \quad
  \int_{\widehat e_2}E_2\,\diff s=0.
\end{align*}
The Poincar\'e--Friedrichs inequality with zero mean on a fixed boundary edge gives
\begin{align*}
  \|E_i\|_{L^2(\widehat T)}
  \leq C|E_i|_{H^1(\widehat T)},\quad i=1,2.
\end{align*}
Because $I_{\widehat T}W$ is affine, $D^2E=D^2W$, which proves \eqref{eq:affine-reference-component-bounds}. The argument extends from smooth functions to $H^2(\widehat T)$ by density.

The two columns of $A_{\widetilde T}^{-\top}$ have Euclidean norms $1/(h_{T,1}t_T)$ and $1/(h_{T,2}t_T)$, respectively. Consequently,
\begin{align*}
  |\tilde w-I_{\widetilde T}\tilde w|_{H^1(\widetilde T)}
  \leq J_{\mathrm{aff}}^{1/2}
  \left(\frac{\|E_1\|_{L^2(\widehat T)}}{h_{T,1}t_T}
       +\frac{\|E_2\|_{L^2(\widehat T)}}{h_{T,2}t_T}\right).
\end{align*}
In the orthonormal coordinates $(x_1,x_2)$ on $\widetilde T$, write
\begin{align*}
  D_i:=|\partial_{\tilde\ell_{T,i}}\tilde w|_{H^1(\widetilde T)},
  \quad i=1,2.
\end{align*}
The chain rule and $\tilde\ell_{T,2}=(s_T,t_T)^{\top}$ yield
\begin{align*}
  W_{11}
  &=h_{T,1}^2\tilde w_{x_1x_1}\circ A_{\widetilde T},\\
  W_{12}
  &=h_{T,1}h_{T,2}
    (s_T\tilde w_{x_1x_1}+t_T\tilde w_{x_1x_2})\circ A_{\widetilde T},\\
  W_{22}
  &=h_{T,2}^2\bigl[s_T(s_T\tilde w_{x_1x_1}+t_T\tilde w_{x_1x_2})
      +t_T\partial_{x_2}\partial_{\tilde\ell_{T,2}}\tilde w\bigr]
      \circ A_{\widetilde T}.
\end{align*}
Because $(s_T,t_T)^{\top}$ is a unit vector, the exact change of variables gives
\begin{align*}
  J_{\mathrm{aff}}^{1/2}\|W_{11}\|_{L^2(\widehat T)}
    &\leq h_{T,1}^2D_1,\\
  J_{\mathrm{aff}}^{1/2}\|W_{12}\|_{L^2(\widehat T)}
    &\leq h_{T,1}h_{T,2}D_1,\\
  J_{\mathrm{aff}}^{1/2}\|W_{22}\|_{L^2(\widehat T)}
    &\leq h_{T,2}^2\bigl(|s_T|D_1+t_TD_2\bigr).
\end{align*}
Substituting these bounds and using $h_{T,2}\leq h_{T,1}$ and $|s_T|\leq1$, we obtain
\begin{align*}
  |\tilde w-I_{\widetilde T}\tilde w|_{H^1(\widetilde T)}
  &\leq C_{\mathrm{ref}}
    \left(\frac{2h_{T,1}+h_{T,2}(1+|s_T|)}{t_T}D_1+h_{T,2}D_2\right)\\
  &\leq C_{\mathrm{ref}}
    \left(\frac{4h_{T,1}}{t_T}D_1+h_{T,2}D_2\right)\\
  &\leq C_{\mathrm{aff},1}
    \left(\mu_Th_{T,1}D_1+2h_{T,2}D_2\right),
\end{align*}
with $C_{\mathrm{aff},1}:=2C_{\mathrm{ref}}$, since $\mu_T=2/t_T$.
Thus,
\begin{align*}
  |\tilde w-I_{\widetilde T}\tilde w|_{H^1(\widetilde T)}
  \leq C_{\mathrm{aff},1}
  \left(\mu_Th_{T,1}
      |\partial_{\tilde\ell_{T,1}}\tilde w|_{H^1(\widetilde T)}
      +2h_{T,2}
      |\partial_{\tilde\ell_{T,2}}\tilde w|_{H^1(\widetilde T)}\right),
\end{align*}
where $C_{\mathrm{aff},1}$ depends only on $\widehat T$.
The isometry identities in
\eqref{eq:affine-core-isometry-summary} prove
\eqref{eq:affine-H1-sharp}, and $\mu_Tt_T=2$ proves the equivalent
form \eqref{eq:affine-H1-directional-scales}.
\end{proof}

Keeping the derivative components separate in the affine scaling confines the maximum-angle factor $\mu_T$ to the first directional term. The identity $\mu_Tt_T=2$ gives the second scale $2h_{T,2}$ and the equivalent form \eqref{eq:affine-H1-directional-scales}.

\subsection{Complete local anisotropic interpolation estimates}
We combine the commuting interpolation identity, the exact radial transfer formulas, and the affine-core estimates.

For $v\in H^2(K)$, let $w=v\circ\Psi_T$ denote the pullback of $v$ to the chordal triangle $T$. For $i,j\in\{1,2\}$, we define
\begin{align}
  \mathscr E_{ij}(v;K)
  :=
  \|\mathscr Q_{\ell_{T,i},\ell_{T,j}}v\|_{0,K;T},
  \label{eq:local-directional-second-data}
\end{align}
and, for $i\in\{1,2\}$,
\begin{align}
  \mathscr G_i(v;K)
  :=
  \left(
    \sum_{k=1}^2
    \|\mathscr Q_{\ell_{T,i},q_{T,k}}v\|_{0,K;T}^2
  \right)^{1/2}.
  \label{eq:local-directional-H1-data}
\end{align}
From Lemma~\ref{lem:second-direction-transfer},
\begin{align}
  \mathscr E_{ij}(v;K)
  =
  \|\partial_{\ell_{T,i}}
    \partial_{\ell_{T,j}}w\|_{L^2(T)}.
  \label{eq:local-second-direction-identity}
\end{align}
Furthermore, because $w\in H^2(T)$, its weak Hessian is symmetric. Therefore, using also the fact that $q_{T,1},q_{T,2}$ form an orthonormal basis of $V_T$,
\begin{align*}
  \mathscr G_i(v;K)^2
  &=
  \sum_{k=1}^2
  \|\partial_{\ell_{T,i}}
    \partial_{q_{T,k}}w\|_{L^2(T)}^2
  =
  \sum_{k=1}^2
  \|\partial_{q_{T,k}}
    \partial_{\ell_{T,i}}w\|_{L^2(T)}^2
  =
  |\partial_{\ell_{T,i}}w|_{H^1(T)}^2.
\end{align*}
Therefore,
\begin{align}
  \mathscr G_i(v;K)
  =
  |\partial_{\ell_{T,i}}w|_{H^1(T)}.
  \label{eq:local-directional-H1-identity}
\end{align}

The affine-core estimates and the exact directional transfer identities give the following local anisotropic bound. No uniform bound on $\mu_T$ or $\chi_T$ is imposed.

\begin{theorem}[Directional anisotropic interpolation on an exact spherical triangle]
  \label{thm:final-local-anisotropic}
  Let $K=\Psi_T(T)$ be an exact spherical triangle and let $v\in H^2(K)$. Then, the mapped linear Lagrange interpolant satisfies
  \begin{align}
    \|v-I_Kv\|_{L^2(K)}
    \leq C_{\mathrm{aff},0}\chi_T
    \sum_{i,j=1}^2
      h_{T,i}h_{T,j}\mathscr E_{ij}(v;K),
    \label{eq:local-anisotropic-L2}
  \end{align}
  and
  \begin{align}
    |v-I_Kv|_{H^1(K)}
    \leq C_{\mathrm{aff},1}\chi_T^{1/2}
    \left(
      \mu_Th_{T,1}\mathscr G_1(v;K)
      +2h_{T,2}\mathscr G_2(v;K)
    \right).
    \label{eq:local-anisotropic-H1}
  \end{align}
  Here, the constants $C_{\mathrm{aff},0}$ and $C_{\mathrm{aff},1}$ are independent of $R$, $h_T$, $H_T$, $\mu_T$, $\chi_T$, and the aspect ratio of $T$.
\end{theorem}

\begin{proof}
Because $K$ is two-dimensional, the Sobolev embedding $H^2(K)\hookrightarrow\mathcal C^0(K)$ shows that the nodal interpolant $I_Kv$ is well defined. We set $w=v\circ\Psi_T$. Because $d_T>0$, the maps $\Psi_T$ and $G_T$ extend smoothly to neighbourhoods of $T$ and $K$, respectively. Therefore, the standard Sobolev pullback result gives $w\in H^2(T)$. Proposition~\ref{prop:interpolant-commuting} yields
\begin{align*}
  (v-I_Kv)\circ\Psi_T=w-I_Tw.
\end{align*}
Applying \eqref{eq:L2-norm-transfer}, \eqref{eq:affine-L2}, and \eqref{eq:local-second-direction-identity}, we obtain
\begin{align*}
  \|v-I_Kv\|_{L^2(K)}
  &\leq
  \chi_T\|w-I_Tw\|_{L^2(T)}
  \\
  &\leq
  C_{\mathrm{aff},0}\chi_T
  \sum_{i,j=1}^2
  h_{T,i}h_{T,j}
  \|\partial_{\ell_{T,i}}\partial_{\ell_{T,j}}w\|_{L^2(T)}
  =
  C_{\mathrm{aff},0}\chi_T
  \sum_{i,j=1}^2
  h_{T,i}h_{T,j}\mathscr E_{ij}(v;K).
\end{align*}
This proves \eqref{eq:local-anisotropic-L2}. Likewise, \eqref{eq:H1-seminorm-transfer}, \eqref{eq:affine-H1-sharp}, and \eqref{eq:local-directional-H1-identity} give
\begin{align*}
  |v-I_Kv|_{H^1(K)}
  &\leq
  \chi_T^{1/2}|w-I_Tw|_{H^1(T)}
  \\
  &\leq
  C_{\mathrm{aff},1}\chi_T^{1/2}
  \left(
    \mu_Th_{T,1}
    |\partial_{\ell_{T,1}}w|_{H^1(T)}
    +
    2h_{T,2}
    |\partial_{\ell_{T,2}}w|_{H^1(T)}
  \right)
  \\
  &=
  C_{\mathrm{aff},1}\chi_T^{1/2}
  \left(
    \mu_Th_{T,1}\mathscr G_1(v;K)
    +
    2h_{T,2}\mathscr G_2(v;K)
  \right),
\end{align*}
which proves \eqref{eq:local-anisotropic-H1}.
\end{proof}

Theorem~\ref{thm:final-local-anisotropic} is the sharp transported directional form of the local estimate. The next theorem bounds its data by intrinsic spherical derivatives, and the following corollaries recover isotropic local rates under progressively stronger geometric assumptions.

\begin{theorem}[Expanded directional anisotropic form]
  \label{thm:standard-directional-form}
  Under the assumptions of Theorem~\ref{thm:final-local-anisotropic}, for
  nonzero vectors $\xi,\eta\in V_T$, set
  \begin{align}
    \mathscr H_{\xi,\eta}(v)
    :=\|\nabla_{\Sph}^2v[e_\xi,e_\eta]\|_{L^2(K)},
    \quad
    \mathscr P_\xi(v)
    :=\|\partial_{e_\xi}^{\Sph}v\|_{L^2(K)}.
    \label{eq:intrinsic-directional-data}
  \end{align}
  Then,
  \begin{align}
    \mathscr E_{ij}(v;K)
    \leq \mathscr E_{ij}^{\mathrm{std}}(v;K)
    &:={\chi_T\mathscr H_{\ell_{T,i},\ell_{T,j}}(v)}
    +\kappa_T
      \left (
        \mathscr P_{\ell_{T,i}}(v)
        +\mathscr P_{\ell_{T,j}}(v)
      \right),
    \label{eq:directional-second-data-bound}
  \end{align}
  and
  \begin{align}
    \mathscr G_i(v;K)
    \leq \mathscr G_i^{\mathrm{std}}(v;K)
    &:={}
    \left
    \{
      \sum_{k=1}^2
      \left[
        \chi_T\mathscr H_{\ell_{T,i},q_{T,k}}(v)
        +\kappa_T
        \left(
          \mathscr P_{\ell_{T,i}}(v)
          +\mathscr P_{q_{T,k}}(v)
        \right)
      \right]^2
    \right\}^{1/2}.
    \label{eq:directional-H1-data-bound}
  \end{align}
  Consequently,
  \begin{align}
    \|v-I_Kv\|_{L^2(K)}
    \leq C_{\mathrm{aff},0}\chi_T
    \sum_{i,j=1}^2
      h_{T,i}h_{T,j}\mathscr E_{ij}^{\mathrm{std}}(v;K),
    \label{eq:expanded-local-anisotropic-L2}
  \end{align}
  and
  \begin{align}
    |v-I_Kv|_{H^1(K)}
    \leq C_{\mathrm{aff},1}\chi_T^{1/2}
    \left(
      \mu_Th_{T,1}\mathscr G_1^{\mathrm{std}}(v;K)
      +2h_{T,2}\mathscr G_2^{\mathrm{std}}(v;K)
    \right).
    \label{eq:expanded-local-anisotropic-H1}
  \end{align}
\end{theorem}

\begin{proof}
Because $\ell_{T,1}$, $\ell_{T,2}$, $q_{T,1}$, and $q_{T,2}$ are unit vectors in $V_T$, we may apply \eqref{eq:second-direction-transfer-bound} with $(\xi,\eta)=(\ell_{T,i},\ell_{T,j})$ to obtain \eqref{eq:directional-second-data-bound}, and with $(\xi,\eta)=(\ell_{T,i},q_{T,k})$ for $k=1,2$. Because all quantities are nonnegative, squaring the resulting bounds for $k=1,2$, summing in $k$, and taking the square root gives \eqref{eq:directional-H1-data-bound}. Finally, substituting \eqref{eq:directional-second-data-bound} and \eqref{eq:directional-H1-data-bound} into \eqref{eq:local-anisotropic-L2} and \eqref{eq:local-anisotropic-H1}, respectively, proves \eqref{eq:expanded-local-anisotropic-L2} and \eqref{eq:expanded-local-anisotropic-H1}.
\end{proof}

Replacing the directional scales by $h_T$ gives the following isotropic consequences.

\begin{corollary}[Isotropic local consequence under spherical locality]
  \label{cor:locality-only-orders}
  Let $\eta_0\in(0,1)$ be fixed. Under the assumptions of Theorem~\ref{thm:final-local-anisotropic}, suppose in addition that
  \begin{align*}
    \vartheta_T\leq\eta_0.
  \end{align*}
  Then, there exists a constant $C=C(\eta_0)$, independent of $R$, $h_T$, $\mu_T$, and the aspect ratio of $T$, such that
  \begin{align}
    \|v-I_Kv\|_{L^2(K)}
    \leq Ch_T^2
    \left(
      \|\nabla_{\Sph}^2v\|_{L^2(K)}
      +\frac{H_T}{R^2}\|\nabla_{\Sph}v\|_{L^2(K)}
    \right),
    \label{eq:isotropic-local-L2}
  \end{align}
  whereas
  \begin{align}
    |v-I_Kv|_{H^1(K)}
    \leq C\mu_Th_T
    \left(
      \|\nabla_{\Sph}^2v\|_{L^2(K)}
      +\frac{H_T}{R^2}\|\nabla_{\Sph}v\|_{L^2(K)}
    \right).
    \label{eq:isotropic-local-H1}
  \end{align}
  Thus, no uniform upper bound on $\mu_T$ is required for the $L^2$ estimate, whereas the $H^1$ estimate retains the chordal semi-regularity factor $\mu_T$ explicitly. Because $H_T\leq4\eta_0R$, the lower-order term in \eqref{eq:isotropic-local-L2} may also be bounded by
\begin{align*}
    \frac{C(\eta_0)h_T^2}{R}
    \|\nabla_{\Sph}v\|_{L^2(K)}.
 \end{align*}
\end{corollary}

\begin{proof}
We set
\begin{align*}
  \chi_0:=(1-\eta_0^2)^{-1/2}.
\end{align*}
By spherical locality,
\begin{align*}
  \chi_T\leq\chi_0,
\end{align*}
and, because
\begin{align*}
  \kappa_T
  =
  \frac{H_T\chi_T^2}{4R^2},
\end{align*}
one has
\begin{align*}
  \kappa_T
  \leq
  \frac{\chi_0^2}{4}
  \frac{H_T}{R^2}.
\end{align*}
Because the transported directions $e_\xi$ are unit tangent vectors,
\begin{align*}
  \mathscr H_{\xi,\eta}(v)
  &\leq
  \|\nabla_{\Sph}^2v\|_{L^2(K)}, \quad
  \mathscr P_\xi(v)
  \leq
  \|\nabla_{\Sph}v\|_{L^2(K)}.
\end{align*}
Therefore, \eqref{eq:directional-second-data-bound} gives
\begin{align}
  \mathscr E_{ij}(v;K)
  \leq
  \chi_T\|\nabla_{\Sph}^2v\|_{L^2(K)}
  +
  2\kappa_T\|\nabla_{\Sph}v\|_{L^2(K)}.
  \label{eq:proof-isotropic-E-bound}
\end{align}
Because $h_{T,i}\leq h_T$,
\begin{align*}
  \sum_{i,j=1}^2
  h_{T,i}h_{T,j}\mathscr E_{ij}(v;K)
  \leq
  4h_T^2
  \left(
    \chi_T\|\nabla_{\Sph}^2v\|_{L^2(K)}
    +
    2\kappa_T\|\nabla_{\Sph}v\|_{L^2(K)}
  \right).
\end{align*}
Combining this with \eqref{eq:local-anisotropic-L2}, $\chi_T\leq\chi_0$, and the preceding bound for $\kappa_T$ proves \eqref{eq:isotropic-local-L2}. Similarly, \eqref{eq:directional-H1-data-bound} yields
\begin{align*}
  \mathscr G_i(v;K)
  \leq
  \sqrt2
  \left(
    \chi_T\|\nabla_{\Sph}^2v\|_{L^2(K)}
    +
    2\kappa_T\|\nabla_{\Sph}v\|_{L^2(K)}
  \right),
  \quad i=1,2.
\end{align*}
Using
\begin{align*}
  h_{T,1}\leq h_T,
  \quad
  h_{T,2}\leq h_T,
  \quad
  \mu_T\geq2,
\end{align*}
we obtain
\begin{align*}
  \mu_Th_{T,1}\mathscr G_1(v;K)
  +2h_{T,2}\mathscr G_2(v;K)
  \leq
  2\sqrt2\,\mu_Th_T
  \left(
    \chi_T\|\nabla_{\Sph}^2v\|_{L^2(K)}
    +
    2\kappa_T\|\nabla_{\Sph}v\|_{L^2(K)}
  \right).
\end{align*}
Substitution into \eqref{eq:local-anisotropic-H1}, followed again by the locality bounds on $\chi_T$ and $\kappa_T$, proves \eqref{eq:isotropic-local-H1}.

Finally,
\begin{align*}
  \vartheta_T=\frac{H_T}{4R}\leq\eta_0
\end{align*}
implies
\begin{align*}
  \frac{H_T}{R^2}\leq\frac{4\eta_0}{R},
\end{align*}
which gives the stated simplified form of the lower-order $L^2$ term.
\end{proof}

\begin{corollary}[Standard local rates under chordal semi-regularity]
  \label{cor:standard-local-orders}
  Under the assumptions of Corollary~\ref{cor:locality-only-orders}, let $\gamma_0\geq2$ be fixed and suppose in addition that
  \begin{align*}
    \mu_T\leq\gamma_0.
  \end{align*}
  Then, there exists a constant $C=C(\gamma_0,\eta_0)$, independent of $R$, $h_T$, and the aspect ratio of $T$, such that
\begin{subequations} \label{eq:standard-local-rates}
\begin{align}
\displaystyle
\|v-I_Kv\|_{L^2(K)}
    &\leq C
    \left(
      h_T^2\|\nabla_{\Sph}^2v\|_{L^2(K)}
      +\frac{h_T^3}{R^2}\|\nabla_{\Sph}v\|_{L^2(K)}
    \right), \label{eq:standard-local-rates=a}
    \\
    |v-I_Kv|_{H^1(K)}
    &\leq C
    \left(
      h_T\|\nabla_{\Sph}^2v\|_{L^2(K)}
      +\frac{h_T^2}{R^2}\|\nabla_{\Sph}v\|_{L^2(K)}
    \right). \label{eq:standard-local-rates=b}
\end{align}
\end{subequations}
\end{corollary}

\begin{proof}
  Apply Corollary~\ref{cor:locality-only-orders}. Because
  \begin{align*}
    \mu_T\leq\gamma_0,
    \quad
    H_T=\mu_Th_T\leq\gamma_0h_T,
  \end{align*}
  the $L^2$ estimate follows immediately. For the $H^1$ estimate, one additionally uses
  \begin{align*}
    \mu_Th_T\leq\gamma_0h_T,
    \quad
    \mu_Th_TH_T
    \leq
    \gamma_0^2h_T^2.
  \end{align*}
  Absorbing the powers of $\gamma_0$ into $C=C(\gamma_0,\eta_0)$ proves \eqref{eq:standard-local-rates}.
\end{proof}

\begin{remark}[Roles of the geometric conditions]
The chordal factor $\mu_T$ appears explicitly in the $H^1$ interpolation scale and is the relevant shape parameter for a uniform standard $H^1$ bound. No uniform bound on $\mu_T$ is needed for the local $L^2$ estimate. By Corollary~\ref{cor:intrinsic-partII-geometric-interface}, the intrinsic condition $t_K^{\mathrm{sr}}\geq c_{\mathrm{sr}}>0$ gives the required uniform control of $\mu_T$.

The radial factor $\chi_T$ measures the distortion of the projection, and $\kappa_T=H_T\chi_T^2/(4R^2)$ measures its second-order contribution. On a fixed sphere, uniform chordal semi-regularity together with $\max_K h_{T_K}/R\to0$ gives $\max_K\chi_{T_K}\to1$ and $\max_K\kappa_{T_K}\to0$ by Corollary~\ref{cor:refining-locality}. The converse is false: refinement and radial locality alone do not control the maximum angle; see Subsection~\ref{subsec:maximum-angle-sharpness}.
\end{remark}

\section{Sharpness of the geometric mechanisms}
\label{sec:sharpness}
We use four families of exact spherical triangles to test the geometric factors in the local interpolation estimates. The first family becomes thin while its chordal maximum angle remains a right angle. The second becomes flat symmetrically, and the third becomes flat with unequal directional lengths. In both cases, the chordal maximum angle tends to $\pi$. All three refine with radial distortion tending to one. The fourth is a coarse family: its chordal triangles remain equilateral, while their supporting planes approach the sphere centre. Table~\ref{tab:sharpness-families} records the distinct roles of these examples.

\begin{table}[ht]
  \centering
  \small
  \setlength{\tabcolsep}{4pt}
  \renewcommand{\arraystretch}{1.15}
  \begin{tabular}{@{}>{\centering\arraybackslash}p{0.12\textwidth}>{\raggedright\arraybackslash}p{0.25\textwidth}>{\raggedright\arraybackslash}p{0.23\textwidth}>{\raggedright\arraybackslash}p{0.30\textwidth}@{}}
    \toprule
    Family & Geometry & Quantity tested & Sharpness conclusion \\
    \midrule
    Right
      & $\mu_T=2$, $\chi_T\to1$, and $h_{T,1}/h_{T,2}\to\infty$
      & directional scales versus the diameter
      & the directional $L^2$ and $H^1$ scales remain sharp, whereas diameter replacement loses anisotropic order \\
    Isosceles
      & $\chi_T\to1$ and $\mu_T\to\infty$
      & size and exponent of the maximum-angle factor
      & the factor $\mu_T$ is necessary in the $H^1$ scale, and an exponent below one fails \\
    Asymmetric
      & $\chi_T\to1$, $\mu_T\to\infty$, and $h_{T,1}/h_{T,2}\to\infty$
      & location of $\mu_T$
      & the factor $\mu_T$ cannot be moved to the shorter directional scale \\
    Equilateral
      & $\mu_T=4/\sqrt3$ and $\chi_T\to\infty$
      & upper $H^1$ radial-transfer factor
      & the exponent $1/2$ of $\chi_T$ is optimal \\
    \bottomrule
  \end{tabular}
  \caption{The four sharpness families and the geometric mechanisms they isolate.}
  \label{tab:sharpness-families}
\end{table}

Within each subsection, we write $T_\varepsilon$ and $K_\varepsilon$ for the chordal triangle and its radial image, and $p_{i,\varepsilon}$, $w_\varepsilon$, and $v_\varepsilon$ for the vertices and test functions of that family. The family is determined by the subsection, so no family superscript is needed. The notation $T_\varepsilon^{\mathrm{loc}}$ used in Section~\ref{sec:geometry} denotes the equilateral family of Subsection~\ref{subsec:radial-exponent-sharpness}. We use the abbreviations
\begin{align*}
  h_\varepsilon&:=h_{T_\varepsilon},
  \quad h_{i,\varepsilon}:=h_{T_\varepsilon,i},
  \quad \ell_{i,\varepsilon}:=\ell_{T_\varepsilon,i},
  \quad i=1,2,\\
  \beta_\varepsilon&:=\beta_{T_\varepsilon},
  \quad \mu_\varepsilon:=\mu_{T_\varepsilon},
  \quad \chi_\varepsilon:=\chi_{T_\varepsilon}.
\end{align*}
Here, $h_{1,\varepsilon}\geq h_{2,\varepsilon}$ are the two canonical directional lengths. We also reuse $\mathcal A_{0,\varepsilon}$ and $\mathcal A_{1,\varepsilon}$ for the $L^2$ and asymmetric $H^1$ directional scales; these quantities are defined separately for each family.

The first three families are analysed on their chordal cores, with $\chi_\varepsilon\to1$. The following lemma transfers their affine error ratios and directional scales to the sphere.

\begin{lemma}[Affine-to-spherical transfer for sharpness families]
  \label{lem:sharpness-transfer-principle}
  Let $T_\varepsilon$ be a chordal triangle with $d_{T_\varepsilon}>0$, set $K_\varepsilon:=\Psi_{T_\varepsilon}(T_\varepsilon)$, and let $w_\varepsilon$ be the restriction of a quadratic polynomial to $T_\varepsilon$. We define
  \begin{align*}
    v_\varepsilon
    :=w_\varepsilon\circ G_{T_\varepsilon},
    \quad
    \mathcal E_\varepsilon
    :=w_\varepsilon-I_{T_\varepsilon}w_\varepsilon.
  \end{align*}
  We set
  \begin{align*}
    U_\varepsilon:=U_{T_\varepsilon}.
  \end{align*}
  Then, $U_\varepsilon$ is an open neighbourhood of $K_\varepsilon$ in $\Sph_R^2$, and $v_\varepsilon$ admits the explicit $\mathcal C^\infty$ extension
  \begin{align*}
    \tilde v_\varepsilon(x)
    :=w_\varepsilon\!\left(
      \frac{d_{T_\varepsilon}}
           {\nu_{T_\varepsilon}\cdot x}\,x
    \right),
    \quad x\in U_\varepsilon.
  \end{align*}
  Furthermore,
  \begin{align}
    (v_\varepsilon-I_{K_\varepsilon}v_\varepsilon)
    \circ\Psi_{T_\varepsilon}
    =\mathcal E_\varepsilon.
    \label{eq:sharpness-common-commuting}
  \end{align}
  For any positive normalising scale $S_\varepsilon$,
  \begin{align}
    \chi_\varepsilon^{-1/2}
    \frac{\|\mathcal E_\varepsilon\|_{L^2(T_\varepsilon)}}{S_\varepsilon}
    &\leq
    \frac{\|v_\varepsilon-I_{K_\varepsilon}v_\varepsilon\|_{L^2(K_\varepsilon)}}
         {S_\varepsilon}
    \leq
    \chi_\varepsilon
    \frac{\|\mathcal E_\varepsilon\|_{L^2(T_\varepsilon)}}{S_\varepsilon},
    \label{eq:sharpness-common-L2-transfer}
    \\
    \chi_\varepsilon^{-1/2}
    \frac{|\mathcal E_\varepsilon|_{H^1(T_\varepsilon)}}{S_\varepsilon}
    &\leq
    \frac{|v_\varepsilon-I_{K_\varepsilon}v_\varepsilon|_{H^1(K_\varepsilon)}}
         {S_\varepsilon}
    \leq
    \chi_\varepsilon^{1/2}
    \frac{|\mathcal E_\varepsilon|_{H^1(T_\varepsilon)}}{S_\varepsilon}.
    \label{eq:sharpness-common-H1-transfer}
  \end{align}
  Furthermore, for unit vectors $\xi,\eta\in V_{T_\varepsilon}$,
  \begin{align}
    \|\mathscr Q_{\xi,\eta}v_\varepsilon\|_{0,K_\varepsilon;T_\varepsilon}
    &=
    \|\partial_\xi\partial_\eta w_\varepsilon\|_{L^2(T_\varepsilon)},
    \notag\\
    \left(
      \sum_{k=1}^2
      \|\mathscr Q_{\xi,q_{T_\varepsilon,k}}v_\varepsilon
      \|_{0,K_\varepsilon;T_\varepsilon}^2
    \right)^{1/2}
    &=
    |\partial_\xi w_\varepsilon|_{H^1(T_\varepsilon)}.
    \label{eq:sharpness-common-directional-transfer}
  \end{align}
  Therefore, any scale assembled from these affine directional norms is exactly the corresponding transported spherical scale. If $\chi_\varepsilon\to1$, any affine ratio with a finite positive limit has the same limit after lifting, and any positive affine asymptotic equivalence is preserved.
\end{lemma}

\begin{proof}
  Because $w_\varepsilon$ is the restriction of a quadratic polynomial to the supporting plane of $T_\varepsilon$, the same polynomial is defined smoothly on the whole supporting plane. For $x\in U_\varepsilon$, we set
  \begin{align*}
    \widetilde G_{T_\varepsilon}(x)
    :=\frac{d_{T_\varepsilon}}
            {\nu_{T_\varepsilon}\cdot x}\,x.
  \end{align*}
  Because $\nu_{T_\varepsilon}\cdot x>0$ on $U_\varepsilon$, this map is $\mathcal C^\infty$ there, and
  \begin{align*}
    \nu_{T_\varepsilon}\cdot
    \widetilde G_{T_\varepsilon}(x)
    =d_{T_\varepsilon},
  \end{align*}
  so its image lies in the supporting plane. On $K_\varepsilon$, $\widetilde G_{T_\varepsilon}$ agrees with the inverse radial map $G_{T_\varepsilon}$. Consequently, the extension defined in the lemma satisfies
  \begin{align*}
    \tilde v_\varepsilon
    =w_\varepsilon\circ\widetilde G_{T_\varepsilon}
    \in\mathcal C^\infty(U_\varepsilon),
    \quad
    \tilde v_\varepsilon|_{K_\varepsilon}=v_\varepsilon.
  \end{align*}
  Thus, the regularity asserted in the lemma follows from this explicit construction, without invoking a general extension theorem. Proposition~\ref{prop:interpolant-commuting} gives \eqref{eq:sharpness-common-commuting}. The two-sided estimates are \eqref{eq:L2-norm-transfer} and \eqref{eq:H1-seminorm-transfer} applied to that identity. Finally, \eqref{eq:local-second-direction-identity} and \eqref{eq:local-directional-H1-identity} give \eqref{eq:sharpness-common-directional-transfer}. Whenever $\|\mathcal E_\varepsilon\|_{L^2(T_\varepsilon)}>0$, \eqref{eq:sharpness-common-L2-transfer} with $S_\varepsilon=\|\mathcal E_\varepsilon\|_{L^2(T_\varepsilon)}$ gives
\begin{align*}
  \chi_\varepsilon^{-1/2}
  \leq
  \frac{\|v_\varepsilon-I_{K_\varepsilon}v_\varepsilon\|_{L^2(K_\varepsilon)}}
       {\|\mathcal E_\varepsilon\|_{L^2(T_\varepsilon)}}
  \leq
  \chi_\varepsilon.
\end{align*}
Likewise, whenever $|\mathcal E_\varepsilon|_{H^1(T_\varepsilon)}>0$, \eqref{eq:sharpness-common-H1-transfer} with $S_\varepsilon=|\mathcal E_\varepsilon|_{H^1(T_\varepsilon)}$ yields
\begin{align*}
  \chi_\varepsilon^{-1/2}
  \leq
  \frac{|v_\varepsilon-I_{K_\varepsilon}v_\varepsilon|_{H^1(K_\varepsilon)}}
       {|\mathcal E_\varepsilon|_{H^1(T_\varepsilon)}}
  \leq
  \chi_\varepsilon^{1/2}.
\end{align*}
Therefore, if $\chi_\varepsilon\to1$, both spherical-to-affine error quotients tend to one. It follows that an affine ratio with a finite positive limit has the same limit after lifting, and any positive affine asymptotic equivalence is preserved.
\end{proof}

\subsection{Pure anisotropy under uniform chordal semi-regularity}
The chordal triangles in this family are right triangles with one leg much shorter than the other. Their largest angle remains $\pi/2$, and their spherical images shrink towards the north pole. This example isolates the loss caused by replacing the two directional lengths with the diameter.
Let $e_1,e_2,e_3$ denote the standard basis of $\R^3$. Let $0<\varepsilon<1/2$ and set
\begin{align*}
  r_\varepsilon:=R\varepsilon,
  \quad
  d_\varepsilon:=R\sqrt{1-\varepsilon^2},
  \quad
  c_\varepsilon:=d_\varepsilon e_3.
\end{align*}
We define
\begin{align}
  p_{1,\varepsilon}
  &:=
  c_\varepsilon+r_\varepsilon
  (\cos\varepsilon\,e_1+\sin\varepsilon\,e_2),
  \quad
  p_{2,\varepsilon}
  :=
  c_\varepsilon-r_\varepsilon e_1,
  \quad
  p_{3,\varepsilon}
  :=c_\varepsilon+r_\varepsilon e_1,
  \label{eq:an-vertices}
\end{align}
Because $r_\varepsilon^2+d_\varepsilon^2=R^2$, all three vertices lie on $\Sph_R^2$. Because $\sin\varepsilon>0$, they are non-collinear. Let $T_\varepsilon$ be their chordal triangle. They lie in the plane $\{x_3=d_\varepsilon\}$, so
\begin{align*}
  d_{T_\varepsilon}
  =d_\varepsilon>0.
\end{align*}
The edge $p_{2,\varepsilon}p_{3,\varepsilon}$ is a diameter of the common planar circumcircle, so the angle at $p_{1,\varepsilon}$ is $\pi/2$. Under the canonical labelling,
\begin{align}
  h_\varepsilon&=2r_\varepsilon,
  \quad h_{1,\varepsilon} =2r_\varepsilon\cos\frac\varepsilon2,
  \quad
  h_{2,\varepsilon} =2r_\varepsilon\sin\frac\varepsilon2, \notag\\
  \mu_\varepsilon&=2,
  \quad
  \vartheta_{T_\varepsilon} =\varepsilon, \quad
  \chi_\varepsilon =(1-\varepsilon^2)^{-1/2}.
  \label{eq:an-geometry}
\end{align}
Consequently, $h_\varepsilon\to0$, $\chi_\varepsilon\to1$, and the side-length aspect ratio
\begin{align}
  \rho_\varepsilon
  :=\frac{h_{1,\varepsilon}}{h_{2,\varepsilon}}
  =\cot\frac\varepsilon2
  \longrightarrow\infty.
  \label{eq:an-aspect-ratio}
\end{align}
Thus, $\{T_\varepsilon\}_{\varepsilon\downarrow0}$ is a uniformly chordally semi-regular refining family for which the radial distortion tends to one, while the side-length aspect ratio diverges. Let $\ell_{1,\varepsilon}$ and $\ell_{2,\varepsilon}$ be the unit edge directions from $p_{1,\varepsilon}$. They are orthogonal, and
\begin{align*}
  T_\varepsilon
  =\left\{
    p_{1,\varepsilon}
    +{x_1}\ell_{1,\varepsilon}
    +{x_2}\ell_{2,\varepsilon}: \
    {x_1},{x_2}\geq0,\quad
    \frac{{x_1}}{h_{1,\varepsilon}}+
    \frac{{x_2}}{h_{2,\varepsilon}}\leq1
  \right\}.
\end{align*}
On this triangle set $w_\varepsilon({x_1},{x_2}):={x_2}^2$. Its nodal interpolant is $h_{2,\varepsilon}{x_2}$, so
\begin{align}
  \mathcal E_\varepsilon
  :=w_\varepsilon
    -I_{T_\varepsilon}w_\varepsilon
  ={x_2}({x_2}-h_{2,\varepsilon}).
  \label{eq:an-test-error}
\end{align}

\begin{lemma}[Exact directional scaling on the right-triangle family]
  \label{lem:pure-anisotropy-affine}
 We set $A_\varepsilon:=|T_\varepsilon|_2 =h_{1,\varepsilon}h_{2,\varepsilon}/2$ and we define
  \begin{align*}
    \mathcal A_{0,\varepsilon}
    &:=
    \sum_{i,j=1}^2h_{i,\varepsilon}h_{j,\varepsilon}
    \|\partial_{\ell_{i,\varepsilon}}
      \partial_{\ell_{j,\varepsilon}}
      w_\varepsilon\|_{L^2(T_\varepsilon)},
    \\
    \mathcal A_{1,\varepsilon}
    &:=
    \mu_\varepsilon h_{1,\varepsilon}
    |\partial_{\ell_{1,\varepsilon}}
      w_\varepsilon|_{H^1(T_\varepsilon)}
    +2h_{2,\varepsilon}
    |\partial_{\ell_{2,\varepsilon}}
      w_\varepsilon|_{H^1(T_\varepsilon)}.
  \end{align*}
  Then,
  \begin{align}
    \|\mathcal E_\varepsilon\|_{L^2(T_\varepsilon)}
    &=h_{2,\varepsilon}^2\sqrt{\frac{A_\varepsilon}{30}},
    \quad
    |\mathcal E_\varepsilon|_{H^1(T_\varepsilon)}
    =h_{2,\varepsilon}\sqrt{\frac{A_\varepsilon}{3}},
    \notag\\
    \mathcal A_{0,\varepsilon}
    &=2h_{2,\varepsilon}^2\sqrt{A_\varepsilon},
    \quad
    \mathcal A_{1,\varepsilon}
    =4h_{2,\varepsilon}\sqrt{A_\varepsilon}.
    \label{eq:an-errors-and-scales}
  \end{align}
  Consequently,
  \begin{align}
    \frac{\|\mathcal E_\varepsilon\|_{L^2(T_\varepsilon)}}
         {\mathcal A_{0,\varepsilon}}
    =\frac1{\sqrt{120}},
    \quad
    \frac{|\mathcal E_\varepsilon|_{H^1(T_\varepsilon)}}
         {\mathcal A_{1,\varepsilon}}
    =\frac1{4\sqrt3}.
    \label{eq:an-sharp-ratios}
  \end{align}
  Furthermore,
  \begin{align}
    \frac{h_\varepsilon^2
      \|D_T^2w_\varepsilon\|_{L^2(T_\varepsilon)}}
      {\mathcal A_{0,\varepsilon}}
    &=\sec^2\frac\varepsilon2
      (\rho_\varepsilon)^2,
    \quad
    \frac{h_\varepsilon
      \|D_T^2w_\varepsilon\|_{L^2(T_\varepsilon)}}
      {\mathcal A_{1,\varepsilon}}
    =\frac12\sec\frac\varepsilon2
      \rho_\varepsilon.
    \label{eq:an-isotropic-loss}
  \end{align}
\end{lemma}

\begin{proof}
  Integrating \eqref{eq:an-test-error} and its gradient $(2{x_2}-h_{2,\varepsilon})\ell_{2,\varepsilon}$ over the right triangle gives
  \begin{align*}
    \|\mathcal E_\varepsilon\|_{L^2(T_\varepsilon)}^2
    =\frac{h_{1,\varepsilon}h_{2,\varepsilon}^5}{60},
    \quad
    |\mathcal E_\varepsilon|_{H^1(T_\varepsilon)}^2
    =\frac{h_{1,\varepsilon}h_{2,\varepsilon}^3}{6}.
  \end{align*}
  Only $\partial_{\ell_{2,\varepsilon}}^2 w_\varepsilon=2$ is non-zero, and $|\partial_{\ell_{2,\varepsilon}} w_\varepsilon|_{H^1(T_\varepsilon)}=2\sqrt{A_\varepsilon}$. This proves \eqref{eq:an-errors-and-scales} and \eqref{eq:an-sharp-ratios}. Finally, $\|D_T^2w_\varepsilon\|_{L^2(T_\varepsilon)} =2\sqrt{A_\varepsilon}$ and $h_\varepsilon/h_{2,\varepsilon} =\sec(\varepsilon/2)\rho_\varepsilon$, which gives \eqref{eq:an-isotropic-loss}.
\end{proof}

We set $K_\varepsilon:= \Psi_{T_\varepsilon}(T_\varepsilon)$ and $v_\varepsilon:=w_\varepsilon \circ G_{T_\varepsilon}$.  From Lemma~\ref{lem:sharpness-transfer-principle}, the two normalising quantities in Lemma~\ref{lem:pure-anisotropy-affine} are exactly the corresponding transported spherical directional scales. Because $\chi_\varepsilon\to1$,
\begin{align}
  \frac{\|v_\varepsilon
      -I_{K_\varepsilon}v_\varepsilon\|_{L^2(K_\varepsilon)}}
       {\mathcal A_{0,\varepsilon}}
  &\longrightarrow\frac1{\sqrt{120}},
  \quad
  \frac{|v_\varepsilon
      -I_{K_\varepsilon}v_\varepsilon|_{H^1(K_\varepsilon)}}
       {\mathcal A_{1,\varepsilon}}
  \longrightarrow\frac1{4\sqrt3}.
  \label{eq:an-spherical-limits}
\end{align}
The loss also occurs in the intrinsic isotropic bounds. Lemma~\ref{lem:first-direction-transfer} and the bounds on the singular values of $D\Psi_{T_\varepsilon}$ give
\begin{align*}
  \|\nabla_{\Sph}v_\varepsilon\|_{L^2(K_\varepsilon)}
  &=(1+O(\varepsilon^2))
    \|\nabla_Tw_\varepsilon\|_{L^2(T_\varepsilon)},\\
  \|\nabla_Tw_\varepsilon\|_{L^2(T_\varepsilon)}
  &\leq h_{2,\varepsilon}
    \|D_T^2w_\varepsilon\|_{L^2(T_\varepsilon)}.
\end{align*}
Since $H_{T_\varepsilon}=4R\varepsilon$ and $h_{2,\varepsilon}=O(R\varepsilon^2)$,
\begin{align*}
  \frac{H_{T_\varepsilon}}{R^2}
  \|\nabla_{\Sph}v_\varepsilon\|_{L^2(K_\varepsilon)}
  =O(\varepsilon^3)
    \|D_T^2w_\varepsilon\|_{L^2(T_\varepsilon)}.
\end{align*}
Lemma~\ref{lem:exact-directional-derivatives} and $\chi_\varepsilon=1+O(\varepsilon^2)$ then give
\begin{align*}
  \|\nabla_{\Sph}^2v_\varepsilon\|_{L^2(K_\varepsilon)}
  =(1+O(\varepsilon^2))
    \|D_T^2w_\varepsilon\|_{L^2(T_\varepsilon)}.
\end{align*}
Therefore, the scales on the right-hand sides of \eqref{eq:isotropic-local-L2}--\eqref{eq:isotropic-local-H1} exceed the actual spherical errors by factors of exact order $\rho_\varepsilon^2$ and $\rho_\varepsilon$, respectively. The directional scales retain the correct order as the aspect ratio diverges.

\subsection{Maximum-angle degeneration}
\label{subsec:maximum-angle-sharpness}
The chordal triangles are isosceles, and their apex approaches the midpoint of the base relative to the base length. The two directional lengths remain equal, but the apex angle tends to $\pi$. Their spherical images again refine with $\chi_\varepsilon\to1$. This example tests the size and exponent of the maximum-angle factor. We define a refining family that is not uniformly semi-regular as follows. For $0<\varepsilon<1/2$, we set
\begin{align*}
  r_\varepsilon:=R\varepsilon,
  \quad d_\varepsilon:=R\sqrt{1-\varepsilon^2},
  \quad \delta_\varepsilon:=\varepsilon,
  \quad b_1:=(0,1,0)^{\top},
  \quad b_2:=(1,0,0)^{\top},
\end{align*}
and
\begin{align}
  \lambda_\varepsilon
  :=r_\varepsilon\sin\delta_\varepsilon,
  \quad
  \zeta_\varepsilon
  :=r_\varepsilon(1-\cos\delta_\varepsilon),
  \quad
  m_\varepsilon
  :=(r_\varepsilon\cos\delta_\varepsilon,0,d_\varepsilon)^{\top}.
  \label{eq:ma-coordinate-scales}
\end{align}
Furthermore,
\begin{align*}
  p_{1,\varepsilon}
  &:=m_\varepsilon+\zeta_\varepsilon b_2,
  \quad
  p_{2,\varepsilon}
  :=m_\varepsilon-\lambda_\varepsilon b_1,
  \quad
  p_{3,\varepsilon}
  :=m_\varepsilon+\lambda_\varepsilon b_1,
\end{align*}
and
\begin{align}
  T_\varepsilon
  :=\conv\bigl\{
    p_{1,\varepsilon},
    p_{2,\varepsilon},
    p_{3,\varepsilon}
  \bigr\}.
  \label{eq:ma-family-definition}
\end{align}
Because $\lambda_\varepsilon>0$ and $\zeta_\varepsilon>0$, these vertices are non-collinear. Furthermore,
\begin{align*}
  r_\varepsilon\cos\delta_\varepsilon+\zeta_\varepsilon
  =r_\varepsilon,
  \quad
  \lambda_\varepsilon=r_\varepsilon\sin\delta_\varepsilon,
\end{align*}
and therefore
\begin{align*}
  |p_{i,\varepsilon}|^2
  =r_\varepsilon^2+d_\varepsilon^2
  =R^2,
  \quad i=1,2,3.
\end{align*}
Thus, all three vertices lie on $\Sph_R^2$.  Their support plane is $x_3=d_\varepsilon$, with
\begin{align*}
  \nu_{T_\varepsilon}=(0,0,1)^{\top},
  \quad
  d_{T_\varepsilon}=d_\varepsilon>0,
\end{align*}
and $b_1,b_2$ form an orthonormal basis of
\begin{align*}
  V_{T_\varepsilon}
  =\{\xi\in\R^3:\xi_3=0\}.
\end{align*}
In the corresponding $(b_1,b_2)$ coordinates, the apex $p_{1,\varepsilon}$ and the two base vertices are $(0,\zeta_\varepsilon)^{\top}$, $(-\lambda_\varepsilon,0)^{\top}$, and $(\lambda_\varepsilon,0)^{\top}$, respectively, relative to $m_\varepsilon$. The two edges issuing from the apex have common length
\begin{align*}
  L_\varepsilon
  :=\sqrt{\lambda_\varepsilon^2+\zeta_\varepsilon^2}
  =2r_\varepsilon\sin\frac{\delta_\varepsilon}{2}.
\end{align*}
Furthermore,
\begin{align*}
  2\lambda_\varepsilon
  =2L_\varepsilon\cos\frac{\delta_\varepsilon}{2}
  >L_\varepsilon,
\end{align*}
so the base is the longest edge. If $\beta_\varepsilon$ denotes the opposite angle, then
\begin{align*}
  \cos\beta_\varepsilon
  =\frac{\zeta_\varepsilon^2-\lambda_\varepsilon^2}
  {\zeta_\varepsilon^2+\lambda_\varepsilon^2}
  =-\cos\delta_\varepsilon,
\end{align*}
so $\beta_\varepsilon =\pi-\delta_\varepsilon$. Consequently, the above labelling agrees with Convention~\ref{conv:chordal-labelling}, and
\begin{align}
  h_{1,\varepsilon}
   &=h_{2,\varepsilon} =L_\varepsilon,
  \quad
  h_\varepsilon=2\lambda_\varepsilon,
  \quad
  \mu_\varepsilon
  =\frac2{\sin\delta_\varepsilon},
  \quad
  H_{T_\varepsilon}=4R\varepsilon,
  \quad
  \vartheta_{T_\varepsilon}=\varepsilon,
  \notag\\
  \frac{\lambda_\varepsilon}{\zeta_\varepsilon}
  &=\cot\frac{\delta_\varepsilon}{2},
  \quad
  \chi_\varepsilon
  =(1-\varepsilon^2)^{-1/2}.
  \label{eq:ma-geometry}
\end{align}
Thus, $h_\varepsilon\to0$, $\mu_\varepsilon\to\infty$, $\vartheta_{T_\varepsilon}\to0$, and $\chi_\varepsilon\to1$. By the canonical labelling, the unit directions $\ell_{i,\varepsilon}$ point from the apex to the two base vertices and satisfy
\begin{align}
  \ell_{1,\varepsilon}
  =\frac{-\lambda_\varepsilon b_1-\zeta_\varepsilon b_2}{L_\varepsilon},
  \quad
  \ell_{2,\varepsilon}
  =\frac{\lambda_\varepsilon b_1-\zeta_\varepsilon b_2}{L_\varepsilon}.
  \label{eq:ma-directions}
\end{align}
For $y=m_\varepsilon+\sigma b_1+\tau b_2$, we define $w_\varepsilon(y):=\sigma^2$. Because its base values are $\lambda_\varepsilon^2$ and its apex value is zero,
\begin{align}
  \mathcal E_\varepsilon
  :=w_\varepsilon
    -I_{T_\varepsilon}w_\varepsilon
  =\sigma^2-\lambda_\varepsilon^2
   +\frac{\lambda_\varepsilon^2}{\zeta_\varepsilon}\tau.
  \label{eq:ma-test-error}
\end{align}

\begin{lemma}[Exact maximum-angle error ratios]
  \label{lem:maximum-angle-affine}
  We define
  \begin{align*}
    \mathcal A_{0,\varepsilon}
    &:=
    \sum_{i,j=1}^2L_\varepsilon^2
    \|\partial_{\ell_{i,\varepsilon}}
      \partial_{\ell_{j,\varepsilon}}
      w_\varepsilon\|_{L^2(T_\varepsilon)},
    \quad
    \mathcal U_{1,\varepsilon}
    :=
    \sum_{i=1}^2L_\varepsilon
    |\partial_{\ell_{i,\varepsilon}}
      w_\varepsilon|_{H^1(T_\varepsilon)},
    \\
    \mathcal A_{1,\varepsilon}
    &:=
    \mu_\varepsilon L_\varepsilon
    |\partial_{\ell_{1,\varepsilon}}
      w_\varepsilon|_{H^1(T_\varepsilon)}
    +2L_\varepsilon
    |\partial_{\ell_{2,\varepsilon}}
      w_\varepsilon|_{H^1(T_\varepsilon)}.
  \end{align*}
  Then,
  \begin{align}
    \|\mathcal E_\varepsilon\|_{L^2(T_\varepsilon)}^2
    &=\frac3{10}\lambda_\varepsilon^5\zeta_\varepsilon,
    \quad
    |\mathcal E_\varepsilon|_{H^1(T_\varepsilon)}^2
    =\frac{\lambda_\varepsilon^5}{\zeta_\varepsilon}
      +\frac23\lambda_\varepsilon^3\zeta_\varepsilon,
    \notag \\
    \mathcal A_{0,\varepsilon}
    &=8\lambda_\varepsilon^2\sqrt{\lambda_\varepsilon\zeta_\varepsilon},
    \quad
    \mathcal U_{1,\varepsilon}
    =4\lambda_\varepsilon\sqrt{\lambda_\varepsilon\zeta_\varepsilon},
    \notag\\
    \mathcal A_{1,\varepsilon}
    &=2\lambda_\varepsilon\sqrt{\lambda_\varepsilon\zeta_\varepsilon}
      (\mu_\varepsilon+2).
    \label{eq:ma-errors-and-scales}
  \end{align}
  Consequently,
  \begin{align}
    \frac{\|\mathcal E_\varepsilon\|_{L^2(T_\varepsilon)}}
         {\mathcal A_{0,\varepsilon}}
    &=\frac18\sqrt{\frac3{10}},
    \quad
    \frac{|\mathcal E_\varepsilon|_{H^1(T_\varepsilon)}}
         {\mathcal U_{1,\varepsilon}}
    =\frac14\sqrt{\cot^2\frac{\delta_\varepsilon}{2}+\frac23}
      \sim\frac14\mu_\varepsilon,
    \notag\\
    \frac{|\mathcal E_\varepsilon|_{H^1(T_\varepsilon)}}
         {\mathcal A_{1,\varepsilon}}
    &=\frac{\sqrt{\cot^2(\delta_\varepsilon/2)+2/3}}
      {2(\mu_\varepsilon+2)}
      \longrightarrow\frac12.
    \label{eq:ma-ratios}
  \end{align}
\end{lemma}

\begin{proof}
  The coordinate domain is $0\leq \tau\leq\zeta_\varepsilon$ and $|\sigma|\leq\lambda_\varepsilon(1-\tau/\zeta_\varepsilon)$. Direct integration of \eqref{eq:ma-test-error} and $\nabla_T\mathcal E_\varepsilon =2\sigma b_1+(\lambda_\varepsilon^2/\zeta_\varepsilon)b_2$ gives the first line of \eqref{eq:ma-errors-and-scales}. Because $D_T^2w_\varepsilon=2b_1\otimes b_1$, \eqref{eq:ma-directions} and $|T_\varepsilon|_2 =\lambda_\varepsilon\zeta_\varepsilon$ give
  \begin{align*}
    |\partial_{\ell_{i,\varepsilon}}
      \partial_{\ell_{j,\varepsilon}}
      w_\varepsilon|
    =\frac{2\lambda_\varepsilon^2}{L_\varepsilon^2},
    \quad
    |\partial_{\ell_{i,\varepsilon}}
      w_\varepsilon|_{H^1(T_\varepsilon)}
    =\frac{2\lambda_\varepsilon}{L_\varepsilon}
      \sqrt{\lambda_\varepsilon\zeta_\varepsilon}.
  \end{align*}
  This proves the remaining scale identities. Division gives \eqref{eq:ma-ratios}, because $\lambda_\varepsilon/\zeta_\varepsilon =\cot(\delta_\varepsilon/2)$ and
  \begin{align*}
    \frac{\cot(\delta_\varepsilon/2)}
         {\mu_\varepsilon}
    =\cos^2\frac{\delta_\varepsilon}{2}\longrightarrow1.
  \end{align*}
\end{proof}

We set $K_\varepsilon :=\Psi_{T_\varepsilon}(T_\varepsilon)$ and $v_\varepsilon :=w_\varepsilon\circ G_{T_\varepsilon}$. From \eqref{eq:sharpness-common-directional-transfer}, the three affine scales above are exactly their transported spherical counterparts.

\begin{theorem}[Sharp maximum-angle scaling]
  \label{thm:sharpness-semiregularity}
  As $\varepsilon\downarrow0$,
  \begin{align}
    \frac{\|v_\varepsilon
      -I_{K_\varepsilon}v_\varepsilon\|_{L^2(K_\varepsilon)}}
      {\mathcal A_{0,\varepsilon}}
    &\longrightarrow\frac18\sqrt{\frac3{10}},
    \quad
    \frac{|v_\varepsilon
      -I_{K_\varepsilon}v_\varepsilon|_{H^1(K_\varepsilon)}}
      {\mathcal A_{1,\varepsilon}}
    \longrightarrow\frac12,
    \notag\\
    \frac{|v_\varepsilon
      -I_{K_\varepsilon}v_\varepsilon|_{H^1(K_\varepsilon)}}
      {\mathcal U_{1,\varepsilon}}
    &\sim\frac14\mu_\varepsilon
      \longrightarrow\infty.
    \label{eq:ma-spherical-limits}
  \end{align}
  Thus, the unweighted directional scale cannot control the $H^1$ error with a constant independent of the chordal maximum angle, whereas the asymmetric scale has the correct order.
\end{theorem}

\begin{proof}
  Apply Lemma~\ref{lem:sharpness-transfer-principle} with the three positive normalising scales in Lemma~\ref{lem:maximum-angle-affine}. Because $\chi_\varepsilon\to1$, the affine ratios \eqref{eq:ma-ratios} are preserved after lifting.
\end{proof}

Indeed, in this symmetric family the two directional $H^1$ data are equal, and therefore
\begin{align*}
  h_{1,\varepsilon}\mathscr G_1
  +2h_{2,\varepsilon}\mathscr G_2
  =\frac32\mathcal U_{1,\varepsilon}.
\end{align*}
The corresponding error ratio is asymptotic to $\mu_\varepsilon/6$ and diverges, proving that the factor $\mu_T$ is needed in a uniform local bound. More generally, for any fixed $\theta<1$, equality of the two directional lengths and data gives
\begin{align*}
  \frac{|v_\varepsilon-I_{K_\varepsilon}v_\varepsilon|_{H^1(K_\varepsilon)}}
    {\mu_\varepsilon^{\theta}h_{1,\varepsilon}\mathscr G_1
      +2h_{2,\varepsilon}\mathscr G_2}
  =\frac{\mu_\varepsilon}
    {2(\mu_\varepsilon^{\theta}+2)}(1+o(1))
  \longrightarrow\infty.
\end{align*}
Thus, this symmetric family already rules out reducing the exponent of $\mu_T$ below one; the next family identifies which directional term must carry that factor. Convergence along a particular refining sequence may still occur when element-size decay compensates for the growth of $\mu_T$.

\begin{remark}[Shape dependence of the quadratic $H^1$ problem]
  \label{rem:quadratic-constant-asymmetry}
  The quadratic optimisation problem associated with \eqref{eq:affine-H1-sharp} retains the shape of the triangle.  For $w\in\mathbb P^2(T)$, set $M:=D_T^2w$ and
  \begin{align*}
    \mathcal A_1(w;T)
    :=|T|_2^{1/2}
    \left(
      \mu_Th_{T,1}|M\ell_{T,1}|
      +2h_{T,2}|M\ell_{T,2}|
    \right),
  \end{align*}
  while $\nabla_T(w-I_Tw)$ also retains the directional lengths and angles of $T$. Therefore, the triangle shape does not cancel. Let $C_{\mathrm{aff},1}^{\mathrm{opt}}$ be the optimal universal constant in \eqref{eq:affine-H1-sharp}, and define
  \begin{align*}
    C_{\mathrm{aff},1}^{\mathrm{quad}}
    :=
    \sup_T
    \sup_{w\in\mathbb P^2(T)\setminus\mathbb P^1(T)}
    \frac{|w-I_Tw|_{H^1(T)}}{\mathcal A_1(w;T)},
  \end{align*}
  where the first supremum is taken over all non-degenerate affine triangles labelled according to Convention~\ref{conv:chordal-labelling}. Then,
  \begin{align*}
    C_{\mathrm{aff},1}^{\mathrm{opt}}
    \geq C_{\mathrm{aff},1}^{\mathrm{quad}}
    \geq\frac12.
  \end{align*}
  Indeed, $\cot(\delta_\varepsilon/2)<\mu_\varepsilon$, so the last ratio in \eqref{eq:ma-ratios} is strictly below $1/2$ for each non-degenerate member of the family and tends to $1/2$ as $\varepsilon\downarrow0$. We leave open whether $C_{\mathrm{aff},1}^{\mathrm{quad}}=1/2$, whether this quadratic supremum is attained on a non-degenerate triangle, and what the exact value of $C_{\mathrm{aff},1}^{\mathrm{opt}}$ is.
\end{remark}

\subsection{Asymmetry of the maximum-angle penalty}
\label{subsec:asymmetry-sharpness}
The symmetric family in Subsection~\ref{subsec:maximum-angle-sharpness} cannot distinguish which directional term carries the maximum-angle factor. We now use an asymmetric flat triangle: one edge from the apex is much shorter than the other, while the apex angle still tends to $\pi$. This separates the two scales
\begin{align*}
  \mu_Th_{T,1}\mathscr G_1+2h_{T,2}\mathscr G_2
  \quad\text{and}\quad
  2h_{T,1}\mathscr G_1+\mu_Th_{T,2}\mathscr G_2.
\end{align*}
The example shows that the factor $\mu_T$ must remain on the longer directional term within this two-term estimate. It also tests whether its exponent can be reduced.

For $0<\varepsilon<1/2$, we set
\begin{align}
  \varphi_{1,\varepsilon}:=\varepsilon,
  \quad
  \varphi_{2,\varepsilon}:=\varepsilon^3,
  \quad
  r_\varepsilon:=R\varepsilon,
  \quad
  d_\varepsilon:=R\sqrt{1-\varepsilon^2},
  \quad
  c_\varepsilon:=d_\varepsilon e_3,
  \label{eq:as-parameters}
\end{align}
and
\begin{align}
  p_{1,\varepsilon}
  &:=c_\varepsilon+r_\varepsilon e_1,
  \notag\\
  p_{2,\varepsilon}
  &:=c_\varepsilon
    +r_\varepsilon\bigl(
      \cos\varphi_{1,\varepsilon}\,e_1
      +\sin\varphi_{1,\varepsilon}\,e_2
    \bigr),
  \notag\\
  p_{3,\varepsilon}
  &:=c_\varepsilon
    +r_\varepsilon\bigl(
      \cos\varphi_{2,\varepsilon}\,e_1
      -\sin\varphi_{2,\varepsilon}\,e_2
    \bigr),
  \quad
  T_\varepsilon
  :=\conv\bigl\{
    p_{1,\varepsilon},
    p_{2,\varepsilon},
    p_{3,\varepsilon}
  \bigr\}.
  \label{eq:as-family-definition}
\end{align}
The three vertices lie on the circle of radius $r_\varepsilon$ centred at $c_\varepsilon$ in the plane $\{x_3=d_\varepsilon\}$. Because $r_\varepsilon^2+d_\varepsilon^2=R^2$, they lie on $\Sph_R^2$, and because $0<\varphi_{2,\varepsilon}<\varphi_{1,\varepsilon}<\pi/2$, they are distinct and non-collinear. Their support plane is $x_3=d_\varepsilon$, so that
\begin{align*}
  \nu_{T_\varepsilon}=(0,0,1)^{\top},
  \quad
  d_{T_\varepsilon}=d_\varepsilon>0.
\end{align*}
From
\begin{align}
  p_{2,\varepsilon}-p_{1,\varepsilon}
  &=2r_\varepsilon\sin\frac{\varphi_{1,\varepsilon}}2
   \left(
     -\sin\frac{\varphi_{1,\varepsilon}}2,\
     \cos\frac{\varphi_{1,\varepsilon}}2,\
     0
   \right)^{\top},
  \notag\\
  p_{3,\varepsilon}-p_{1,\varepsilon}
  &=2r_\varepsilon\sin\frac{\varphi_{2,\varepsilon}}2
   \left(
     -\sin\frac{\varphi_{2,\varepsilon}}2,\
     -\cos\frac{\varphi_{2,\varepsilon}}2,\
     0
   \right)^{\top},
  \label{eq:as-edge-vectors}
\end{align}
we obtain the unit edge directions
\begin{align}
  \ell_{1,\varepsilon}
  =\left(
    -\sin\frac{\varphi_{1,\varepsilon}}2,\
    \cos\frac{\varphi_{1,\varepsilon}}2,\
    0
  \right)^{\top},
  \quad
  \ell_{2,\varepsilon}
  =\left(
    -\sin\frac{\varphi_{2,\varepsilon}}2,\
    -\cos\frac{\varphi_{2,\varepsilon}}2,\
    0
  \right)^{\top},
  \label{eq:as-directions}
\end{align}
so
\begin{align*}
  \cos\beta_\varepsilon
  =\ell_{1,\varepsilon}\cdot\ell_{2,\varepsilon}
  =-\cos\frac{\varphi_{1,\varepsilon}+\varphi_{2,\varepsilon}}2 .
\end{align*}
The two base vertices subtend the central angle $\varphi_{1,\varepsilon}+\varphi_{2,\varepsilon}$, so the opposite edge has length $2r_\varepsilon\sin((\varphi_{1,\varepsilon}+\varphi_{2,\varepsilon})/2)$. Because $0<\varphi_{2,\varepsilon}<\varphi_{1,\varepsilon}$ and $\varphi_{1,\varepsilon}+\varphi_{2,\varepsilon}<\pi/2$, the sine is strictly increasing on the relevant range, so that edge is the longest one and $h_{2,\varepsilon}\leq h_{1,\varepsilon}$. The labelling in \eqref{eq:as-family-definition} therefore agrees with Convention~\ref{conv:chordal-labelling}, and
\begin{align}
  h_{1,\varepsilon}
  &=2r_\varepsilon\sin\frac{\varphi_{1,\varepsilon}}2,
  \quad
  h_{2,\varepsilon}
  =2r_\varepsilon\sin\frac{\varphi_{2,\varepsilon}}2,
  \notag\\
  h_\varepsilon
  &=2r_\varepsilon
    \sin\frac{\varphi_{1,\varepsilon}+\varphi_{2,\varepsilon}}2,
  \quad
  \beta_\varepsilon
  =\pi-\frac{\varphi_{1,\varepsilon}+\varphi_{2,\varepsilon}}2,
  \quad
  \mu_\varepsilon
  =\frac2{\sin\bigl((\varphi_{1,\varepsilon}+\varphi_{2,\varepsilon})/2\bigr)},
  \notag\\
  |T_\varepsilon|_2
  &=2r_\varepsilon^2
    \sin\frac{\varphi_{1,\varepsilon}}2
    \sin\frac{\varphi_{2,\varepsilon}}2
    \sin\frac{\varphi_{1,\varepsilon}+\varphi_{2,\varepsilon}}2,
  \quad
  H_{T_\varepsilon}=4R\varepsilon,
  \quad
  \vartheta_{T_\varepsilon}=\varepsilon,
  \notag\\
  \chi_\varepsilon
  &=(1-\varepsilon^2)^{-1/2},
  \quad
  \rho_\varepsilon
  :=\frac{h_{1,\varepsilon}}{h_{2,\varepsilon}}
  =\frac{\sin(\varphi_{1,\varepsilon}/2)}{\sin(\varphi_{2,\varepsilon}/2)}.
  \label{eq:as-geometry}
\end{align}
Thus, $h_\varepsilon\to0$, $\chi_\varepsilon\to1$, $\mu_\varepsilon\to\infty$, and $\rho_\varepsilon\to\infty$: maximum-angle degeneration and aspect-ratio degeneration occur simultaneously, while spherical locality is retained.

The vertex $p_{3,\varepsilon}$ approaches $p_{1,\varepsilon}$ faster than $p_{2,\varepsilon}$ does. Both angular separations tend to zero, however, and the two edge directions become opposite. If $\varphi_{1,\varepsilon}$ were fixed at a positive value while only $\varphi_{2,\varepsilon}$ tended to zero, the maximum angle would remain bounded away from $\pi$. The present family lets both parameters tend to zero. Thus, the short edge does not prevent maximum-angle degeneration. Indeed,
\begin{align*}
  h_{1,\varepsilon}\sim R\varepsilon^2,
  \quad h_{2,\varepsilon}\sim R\varepsilon^4,
  \quad \beta_\varepsilon
  =\pi-\frac{\varepsilon+\varepsilon^3}{2}\longrightarrow\pi.
\end{align*}

We use the altitude frame of $T_\varepsilon$. Let
\begin{align}
  b_{1,\varepsilon}
  :=\frac{p_{2,\varepsilon}-p_{3,\varepsilon}}
         {h_\varepsilon}
  =\left(
    -\sin\frac{\varphi_{1,\varepsilon}-\varphi_{2,\varepsilon}}2,\
    \cos\frac{\varphi_{1,\varepsilon}-\varphi_{2,\varepsilon}}2,\
    0
  \right)^{\top},
  \label{eq:as-altitude-frame}
\end{align}
and let $b_{2,\varepsilon}\in V_{T_\varepsilon}$ be the unit vector orthogonal to $b_{1,\varepsilon}$ with $(p_{1,\varepsilon}-p_{3,\varepsilon}) \cdot b_{2,\varepsilon}>0$. We write
\begin{align*}
  y=m_\varepsilon+{x_1} b_{1,\varepsilon}+{x_2} b_{2,\varepsilon},
\end{align*}
where $m_\varepsilon$ is the foot of the perpendicular from $p_{1,\varepsilon}$ to the line through the two base vertices. Using \eqref{eq:as-edge-vectors} and \eqref{eq:as-altitude-frame},
\begin{align}
  \sigma_{2,\varepsilon}
  &:=(p_{2,\varepsilon}
      -p_{1,\varepsilon})\cdot b_{1,\varepsilon}
  =2r_\varepsilon\sin\frac{\varphi_{1,\varepsilon}}2
    \cos\frac{\varphi_{2,\varepsilon}}2,
  \notag\\
  \sigma_{3,\varepsilon}
  &:=-(p_{3,\varepsilon}
      -p_{1,\varepsilon})\cdot b_{1,\varepsilon}
  =2r_\varepsilon\sin\frac{\varphi_{2,\varepsilon}}2
    \cos\frac{\varphi_{1,\varepsilon}}2,
  \notag\\
  \zeta_\varepsilon
  &:=\frac{2|T_\varepsilon|_2}{h_\varepsilon}
  =2r_\varepsilon\sin\frac{\varphi_{1,\varepsilon}}2
    \sin\frac{\varphi_{2,\varepsilon}}2,
  \label{eq:as-altitude-coordinates}
\end{align}
so that the three vertices are $m_\varepsilon+\sigma_{2,\varepsilon}b_{1,\varepsilon}$, $m_\varepsilon-\sigma_{3,\varepsilon}b_{1,\varepsilon}$ and $m_\varepsilon+\zeta_\varepsilon b_{2,\varepsilon}$, and
\begin{align}
  \sigma_{2,\varepsilon}+\sigma_{3,\varepsilon}
  =h_\varepsilon,
  \quad
  \gamma_\varepsilon
  :=\frac{\sigma_{2,\varepsilon}\sigma_{3,\varepsilon}}{\zeta_\varepsilon}
  =2r_\varepsilon
    \cos\frac{\varphi_{1,\varepsilon}}2
    \cos\frac{\varphi_{2,\varepsilon}}2 .
  \label{eq:as-thin-direction-coefficient}
\end{align}
On $T_\varepsilon$, we set $w_\varepsilon(y):={x_1}^2$. Its nodal values are $\sigma_{2,\varepsilon}^2$, $\sigma_{3,\varepsilon}^2$ and $0$, whence
\begin{align}
  \mathcal E_\varepsilon
  :=w_\varepsilon
    -I_{T_\varepsilon}w_\varepsilon
  ={x_1}^2
   -(\sigma_{2,\varepsilon}-\sigma_{3,\varepsilon}){x_1}
   -\sigma_{2,\varepsilon}\sigma_{3,\varepsilon}
   +\gamma_\varepsilon {x_2},
  \label{eq:as-test-error}
\end{align}
and
\begin{align}
  \nabla_T\mathcal E_\varepsilon
  =\bigl(2{x_1}-(\sigma_{2,\varepsilon}-\sigma_{3,\varepsilon})\bigr)
    b_{1,\varepsilon}
   +\gamma_\varepsilon b_{2,\varepsilon}.
  \label{eq:as-test-error-gradient}
\end{align}
The coefficient $\gamma_\varepsilon$ of the thin direction is the maximum-angle mechanism in its exact form.

\begin{lemma}[Exact asymmetric error ratios]
  \label{lem:asymmetry-affine}
  We set $A_\varepsilon:=|T_\varepsilon|_2$ and we define
  \begin{align*}
    \mathcal A_{1,\varepsilon}
    &:=\mu_\varepsilon
       h_{1,\varepsilon}
       |\partial_{\ell_{1,\varepsilon}}
         w_\varepsilon|_{H^1(T_\varepsilon)}
      +2h_{2,\varepsilon}
       |\partial_{\ell_{2,\varepsilon}}
         w_\varepsilon|_{H^1(T_\varepsilon)},
    \\
    \widetilde{\mathcal A}_{1,\varepsilon}
    &:=2h_{1,\varepsilon}
       |\partial_{\ell_{1,\varepsilon}}
         w_\varepsilon|_{H^1(T_\varepsilon)}
      +\mu_\varepsilon
       h_{2,\varepsilon}
       |\partial_{\ell_{2,\varepsilon}}
         w_\varepsilon|_{H^1(T_\varepsilon)}.
  \end{align*}
  Then,
  \begin{align}
    |\partial_{\ell_{1,\varepsilon}}
      w_\varepsilon|_{H^1(T_\varepsilon)}
    &=2\cos\frac{\varphi_{2,\varepsilon}}2
      \sqrt{A_\varepsilon},
    \quad
    |\partial_{\ell_{2,\varepsilon}}
      w_\varepsilon|_{H^1(T_\varepsilon)}
    =2\cos\frac{\varphi_{1,\varepsilon}}2
      \sqrt{A_\varepsilon},
    \notag\\
    |\mathcal E_\varepsilon|
      _{H^1(T_\varepsilon)}^2
    &=A_\varepsilon
      \left(
        \gamma_\varepsilon^2
        +\frac{\sigma_{2,\varepsilon}^2+\sigma_{3,\varepsilon}^2}{3}
      \right),
    \label{eq:as-errors-and-scales}
  \end{align}
  and consequently
  \begin{align}
    \mathcal A_{1,\varepsilon}
    =2\sqrt{A_\varepsilon}
     \bigl(
       \mu_\varepsilon\sigma_{2,\varepsilon}
       +2\sigma_{3,\varepsilon}
     \bigr),
    \quad
    \widetilde{\mathcal A}_{1,\varepsilon}
    =2\sqrt{A_\varepsilon}
     \bigl(
       2\sigma_{2,\varepsilon}
       +\mu_\varepsilon\sigma_{3,\varepsilon}
     \bigr).
    \label{eq:as-scales}
  \end{align}
  Furthermore,
  \begin{align}
    \frac{|\mathcal E_\varepsilon|
      _{H^1(T_\varepsilon)}}
      {\mathcal A_{1,\varepsilon}}
    \longrightarrow\frac14,
    \quad
    \frac{|\mathcal E_\varepsilon|
      _{H^1(T_\varepsilon)}}
      {\widetilde{\mathcal A}_{1,\varepsilon}}
    =\frac{\mu_\varepsilon}8
     \bigl(1+o(1)\bigr)
    \longrightarrow\infty
    \quad\text{as }\varepsilon\downarrow0 .
    \label{eq:as-ratios}
  \end{align}
\end{lemma}

\begin{proof}
  Because $D_T^2w_\varepsilon =2b_{1,\varepsilon}\otimes b_{1,\varepsilon}$, we have $\nabla_T(\partial_{\ell}w_\varepsilon) =2(\ell\cdot b_{1,\varepsilon})b_{1,\varepsilon}$ for any constant direction $\ell$. From \eqref{eq:as-altitude-coordinates},
  \begin{align*}
    \ell_{1,\varepsilon}\cdot b_{1,\varepsilon}
    =\frac{\sigma_{2,\varepsilon}}{h_{1,\varepsilon}}
    =\cos\frac{\varphi_{2,\varepsilon}}2,
    \quad
    \ell_{2,\varepsilon}\cdot b_{1,\varepsilon}
    =-\frac{\sigma_{3,\varepsilon}}{h_{2,\varepsilon}}
    =-\cos\frac{\varphi_{1,\varepsilon}}2,
  \end{align*}
  which gives the two directional data in \eqref{eq:as-errors-and-scales}. Multiplying by the lengths in \eqref{eq:as-geometry} therefore gives \eqref{eq:as-scales}.

  The function on the right-hand side of \eqref{eq:as-test-error} vanishes at the three vertices, and its difference from $w_\varepsilon$ is affine. Unisolvence therefore identifies it as the nodal interpolation error. In the altitude coordinates, the element is
  \begin{align*}
    0\leq {x_2}\leq\zeta_\varepsilon,
    \quad
    -\sigma_{3,\varepsilon}
      \left(1-\frac {x_2}{\zeta_\varepsilon}\right)
    \leq{x_1}\leq
    \sigma_{2,\varepsilon}
      \left(1-\frac {x_2}{\zeta_\varepsilon}\right),
  \end{align*}
  so that
  \begin{align*}
    \int_{T_\varepsilon}{x_1}\,\diff s
    =\frac{(\sigma_{2,\varepsilon}^2-\sigma_{3,\varepsilon}^2)
      \zeta_\varepsilon}6,
    \quad
    \int_{T_\varepsilon}{x_1}^2\,\diff s
    =\frac{(\sigma_{2,\varepsilon}^3+\sigma_{3,\varepsilon}^3)
      \zeta_\varepsilon}{12},
    \quad
    A_\varepsilon
    =\frac{(\sigma_{2,\varepsilon}+\sigma_{3,\varepsilon})
      \zeta_\varepsilon}2 .
  \end{align*}
  Integrating the square of \eqref{eq:as-test-error-gradient} and using
  \begin{align*}
    \sigma_{2,\varepsilon}^3+\sigma_{3,\varepsilon}^3
    +\sigma_{2,\varepsilon}^2\sigma_{3,\varepsilon}
    +\sigma_{2,\varepsilon}\sigma_{3,\varepsilon}^2
    =(\sigma_{2,\varepsilon}+\sigma_{3,\varepsilon})
     (\sigma_{2,\varepsilon}^2+\sigma_{3,\varepsilon}^2)
  \end{align*}
  give the second identity in \eqref{eq:as-errors-and-scales}.

  For the limits, \eqref{eq:as-parameters},
  \eqref{eq:as-altitude-coordinates} and
  \eqref{eq:as-thin-direction-coefficient} give, as $\varepsilon\downarrow0$,
  \begin{align*}
    \mu_\varepsilon
    &=\frac4\varepsilon\bigl(1+O(\varepsilon^2)\bigr),
    \quad
    \gamma_\varepsilon=2R\varepsilon\bigl(1+O(\varepsilon^2)\bigr),
    \\
    \sigma_{2,\varepsilon}
    &=R\varepsilon^2\bigl(1+O(\varepsilon^2)\bigr),
    \quad
    \sigma_{3,\varepsilon}=R\varepsilon^4\bigl(1+O(\varepsilon^2)\bigr).
  \end{align*}
  Therefore, $\sigma_{2,\varepsilon}^2+\sigma_{3,\varepsilon}^2 =O(\varepsilon^2)\gamma_\varepsilon^2$, so that $|\mathcal E_\varepsilon|_{H^1(T_\varepsilon)} =\gamma_\varepsilon\sqrt{A_\varepsilon}(1+o(1))$, while
  \begin{align*}
    \mu_\varepsilon\sigma_{2,\varepsilon}
    +2\sigma_{3,\varepsilon}
    =4R\varepsilon\bigl(1+o(1)\bigr),
    \quad
    2\sigma_{2,\varepsilon}
    +\mu_\varepsilon\sigma_{3,\varepsilon}
    =2R\varepsilon^2\bigl(1+o(1)\bigr).
  \end{align*}
  Division proves \eqref{eq:as-ratios}.
\end{proof}

We set $K_\varepsilon :=\Psi_{T_\varepsilon}(T_\varepsilon)$ and $v_\varepsilon :=w_\varepsilon\circ G_{T_\varepsilon}$. By \eqref{eq:sharpness-common-directional-transfer}, the two affine directional data in Lemma~\ref{lem:asymmetry-affine} are exactly $\mathscr G_1(v_\varepsilon;K_\varepsilon)$ and $\mathscr G_2(v_\varepsilon;K_\varepsilon)$.

\begin{theorem}[Asymmetry of maximum-angle scaling]
  \label{thm:sharpness-asymmetry}
  As $\varepsilon\downarrow0$,
  \begin{align}
    \frac{|v_\varepsilon
      -I_{K_\varepsilon}v_\varepsilon|
      _{H^1(K_\varepsilon)}}
      {\mathcal A_{1,\varepsilon}}
    \longrightarrow\frac14,
    \quad
    \frac{|v_\varepsilon
      -I_{K_\varepsilon}v_\varepsilon|
      _{H^1(K_\varepsilon)}}
      {\widetilde{\mathcal A}_{1,\varepsilon}}
    \sim\frac{\mu_\varepsilon}8
    \longrightarrow\infty.
    \label{eq:as-spherical-limits}
  \end{align}
  Consequently, no estimate of the form
  \begin{align*}
    |v-I_Kv|_{H^1(K)}
    \leq C\chi_T^{1/2}
    \left(
      2h_{T,1}\mathscr G_1(v;K)
      +\mu_Th_{T,2}\mathscr G_2(v;K)
    \right)
  \end{align*}
  can hold with a constant $C$ independent of the element. More generally, for any $\theta<1$,
  \begin{align}
    \frac{|v_\varepsilon
      -I_{K_\varepsilon}v_\varepsilon|
      _{H^1(K_\varepsilon)}}
      {\mu_\varepsilon^{\theta}
       h_{1,\varepsilon}
       \mathscr G_1(v_\varepsilon;K_\varepsilon)
       +2h_{2,\varepsilon}
       \mathscr G_2(v_\varepsilon;K_\varepsilon)}
    \longrightarrow\infty.
    \label{eq:as-exponent-sharpness}
  \end{align}
  Thus, within the two-term directional scale of \eqref{eq:local-anisotropic-H1}, the maximum-angle factor cannot be transferred from the $h_{T,1}$ term to the $h_{T,2}$ term, and its exponent cannot be reduced below one.
\end{theorem}

\begin{proof}
  Because $\chi_\varepsilon\to1$, we may apply Lemma~\ref{lem:sharpness-transfer-principle} with the positive normalising scales of Lemma~\ref{lem:asymmetry-affine}, and the affine ratios \eqref{eq:as-ratios} are preserved after lifting. This proves \eqref{eq:as-spherical-limits}; the second limit there contradicts any uniform estimate with the exchanged scale.

  For the last assertion, let $0\leq\theta<1$. The asymptotics in the proof of Lemma~\ref{lem:asymmetry-affine} give
  \begin{align*}
    \mu_\varepsilon^{\theta}\sigma_{2,\varepsilon}
    +2\sigma_{3,\varepsilon}
    =4^{\theta}R\varepsilon^{2-\theta}\bigl(1+o(1)\bigr),
    \quad
    \gamma_\varepsilon=2R\varepsilon\bigl(1+o(1)\bigr).
  \end{align*}
  Therefore, the corresponding affine quotient is $4^{-\theta}\varepsilon^{\theta-1}(1+o(1))$ and diverges. Because $\chi_\varepsilon\to1$, Lemma~\ref{lem:sharpness-transfer-principle} gives the same divergence after lifting. If $\theta<0$, then $\mu_\varepsilon\geq2$ implies $\mu_\varepsilon^{\theta}\leq1$. The denominator in \eqref{eq:as-exponent-sharpness} is therefore no larger than its value at $\theta=0$. The divergence already proved for $\theta=0$ proves \eqref{eq:as-exponent-sharpness} also for $\theta<0$.
\end{proof}

\begin{remark}[Relation to the symmetric family]
  \label{rem:as-contains-ma}
  The two-angle construction in \eqref{eq:as-family-definition} recovers the family of Subsection~\ref{subsec:maximum-angle-sharpness} under the symmetric choice $\varphi_{1,\varepsilon}=\varphi_{2,\varepsilon} =\delta_\varepsilon$, up to interchanging $p_{2,\varepsilon}$ and $p_{3,\varepsilon}$. Then, $\sigma_{2,\varepsilon}=\sigma_{3,\varepsilon}=\lambda_\varepsilon$, $\zeta_\varepsilon$ is the apex offset, and \eqref{eq:as-errors-and-scales} reduces to the second identity of \eqref{eq:ma-errors-and-scales}. The symmetric case has $\mathscr G_1=\mathscr G_2$ and $h_{T,1}=h_{T,2}$, and therefore measures only the size of the penalty; the asymmetric case measures its position.
\end{remark}

\subsection{Optimality of the \texorpdfstring{$H^1$}{H1} radial-transfer exponent}
\label{subsec:radial-exponent-sharpness}
The preceding families refine while their radial factors tend to one. This family instead keeps an equilateral chordal shape and moves the supporting plane towards the sphere centre. Its purpose is to test the radial-transfer factor when spherical locality is lost. The parameter $\varepsilon$ now denotes the support-plane distance. For $0<\varepsilon<R$, we set
\begin{align*}
  d_\varepsilon:=\varepsilon,
  \quad
  s_\varepsilon:=\sqrt{R^2-\varepsilon^2},
  \quad
  c_\varepsilon:=(0,0,\varepsilon)^{\top}.
\end{align*}
Let $T_*\subset\{x_3=0\}$ be the equilateral triangle centred at the origin with circumradius one. Then,
\begin{align*}
  h_*:=\sqrt3,
  \quad
  A_*:=|T_*|_2=\frac{3\sqrt3}{4},
  \quad
  \mu_*:=\frac4{\sqrt3}.
\end{align*}
Let $\xi_{*,1},\xi_{*,2},\xi_{*,3}$ be the vertices of $T_*$. For
\begin{align*}
  T_\varepsilon
  :=c_\varepsilon+s_\varepsilon T_*,
  \quad
  K_\varepsilon
  :=\Psi_{T_\varepsilon}
    (T_\varepsilon),
\end{align*}
the vertices of $T_\varepsilon$ are $p_{i,\varepsilon} :=c_\varepsilon+s_\varepsilon\xi_{*,i}$. Because $|\xi_{*,i}|=1$ and $c_\varepsilon\cdot\xi_{*,i}=0$,
\begin{align*}
  |p_{i,\varepsilon}|^2
  =\varepsilon^2+s_\varepsilon^2
  =R^2,
  \quad i=1,2,3.
\end{align*}
Thus, the vertices lie on $\Sph_R^2$.  Since $s_\varepsilon>0$, they remain non-collinear, and their support plane is $x_3=\varepsilon$. Therefore,
\begin{align*}
  \nu_{T_\varepsilon}=(0,0,1)^{\top},
  \quad
  d_{T_\varepsilon}=\varepsilon>0.
\end{align*}
Furthermore,
\begin{align}
  h_{1,\varepsilon}
  &=h_{2,\varepsilon}
  =h_\varepsilon
  =s_\varepsilon h_*,
  \quad
  |T_\varepsilon|_2 =s_\varepsilon^2A_*,
  \notag\\
  \mu_\varepsilon&=\mu_*,
  \quad
  H_{T_\varepsilon} =4s_\varepsilon,
  \quad
  \chi_\varepsilon=\frac R\varepsilon
    \longrightarrow\infty,
  \notag\\
  \vartheta_{T_\varepsilon}
  &=\frac{s_\varepsilon}{R}\longrightarrow1.
  \label{eq:locality-sharp-geometry}
\end{align}
The vertices approach the equator, while the spherical triangle contains the north pole and expands towards the northern hemisphere. Since $h_\varepsilon\to\sqrt3R$, this is a coarse family rather than a refining one. It isolates radial distortion without chordal shape degeneration. Every member has $d_{T_\varepsilon}>0$, so the radial parametrisation remains well defined, although its distortion is unbounded along the family. We define
\begin{align*}
  w_\varepsilon(y)
  :=|y-c_\varepsilon|^2,
  \quad y\in T_\varepsilon,
  \quad
  v_\varepsilon
  :=w_\varepsilon
    \circ G_{T_\varepsilon}
  \quad\text{on }K_\varepsilon.
\end{align*}
Because $d_{T_\varepsilon}>0$, the inverse radial map $G_{T_\varepsilon}$ extends smoothly to an open neighbourhood of $K_\varepsilon$ in $\Sph_R^2$. Since $w_\varepsilon$ is a polynomial, $v_\varepsilon=w_\varepsilon\circ G_{T_\varepsilon}$ also extends smoothly to that neighbourhood and, in particular,
\begin{align*}
  v_\varepsilon
  \in H^2(K_\varepsilon).
\end{align*}
The three nodal values of $w_\varepsilon$ equal $s_\varepsilon^2$, so
\begin{align}
  I_{T_\varepsilon}w_\varepsilon
  =s_\varepsilon^2,
  \quad
  \nabla_T(w_\varepsilon
    -I_{T_\varepsilon}w_\varepsilon)
  =2(y-c_\varepsilon).
  \label{eq:locality-sharp-error}
\end{align}
The commuting identity gives
\begin{align}
  \left(
    v_\varepsilon
    -I_{K_\varepsilon}v_\varepsilon
  \right)
  \circ\Psi_{T_\varepsilon}
  =w_\varepsilon
   -I_{T_\varepsilon}w_\varepsilon.
  \label{eq:locality-sharp-error-pullback}
\end{align}
Furthermore,
\begin{align*}
  D_T^2w_\varepsilon
  =2I_{V_{T_\varepsilon}}.
\end{align*}
Thus, for any unit vector
$\ell\in V_{T_\varepsilon}$,
\begin{align*}
  \nabla_T(\partial_\ell w_\varepsilon)=2\ell,
  \quad
  |\partial_\ell w_\varepsilon|_
  {H^1(T_\varepsilon)}
  =2|T_\varepsilon|_2^{1/2}
  =2s_\varepsilon\sqrt{A_*}.
\end{align*}
Therefore, \eqref{eq:sharpness-common-directional-transfer} gives, for the two unit edge directions,
\begin{align*}
  \mathscr G_i(v_\varepsilon;
  K_\varepsilon)
  =2s_\varepsilon\sqrt{A_*},
  \quad i=1,2.
\end{align*}
Therefore, the asymmetric directional scale is
\begin{align}
  \mathcal A_{1,\varepsilon}
  &:=
  \mu_\varepsilon
  h_{1,\varepsilon}\mathscr G_1
  +2h_{2,\varepsilon}\mathscr G_2
  =2s_\varepsilon^2h_*\sqrt{A_*}(\mu_*+2).
  \label{eq:locality-sharp-scale}
\end{align}

\begin{theorem}[Sharpness of the radial-transfer exponent]
  \label{thm:sharpness-locality-transfer}
  As $\varepsilon\downarrow0$,
  \begin{align}
    \frac{|v_\varepsilon
      -I_{K_\varepsilon}v_\varepsilon|_{H^1(K_\varepsilon)}}
      {\chi_\varepsilon^{1/2}
       \mathcal A_{1,\varepsilon}}
    \longrightarrow
    C_{\mathrm{loc}}
    :=\frac{
      \left(\displaystyle\int_{T_*}|\xi|^3\,\diff\xi\right)^{1/2}}
      {h_*\sqrt{A_*}(\mu_*+2)}
    >0.
    \label{eq:locality-sharp-limit}
  \end{align}
  Consequently, for any $\theta<1/2$,
  \begin{align}
    \frac{|v_\varepsilon
      -I_{K_\varepsilon}v_\varepsilon|_{H^1(K_\varepsilon)}}
      {\chi_\varepsilon^{\theta}
       \mathcal A_{1,\varepsilon}}
    \longrightarrow\infty.
    \label{eq:locality-subcritical-divergence}
  \end{align}
  Thus, no exponent $\theta<1/2$ can replace $1/2$ in \eqref{eq:local-anisotropic-H1} while retaining the same directional scale and a constant independent of $\varepsilon$.  In this precise sense, the exponent $1/2$ is optimal even for equilateral chordal triangles.
\end{theorem}

\begin{proof}
  We set $z:=y-c_\varepsilon$ and $r:=|y|$. Because $c_\varepsilon\cdot z=0$, one has $r^2=\varepsilon^2+|z|^2$. From \eqref{eq:locality-sharp-error-pullback} and \eqref{eq:locality-sharp-error}, the intrinsic gradient of the pullback of the spherical interpolation error is $2z$.  For $z\neq0$, this vector is aligned with the in-plane radial singular direction $e_{\mathrm{rad}}=z/|z|$ and has no component in the orthogonal tangential direction. Therefore, the exact energy splitting \eqref{eq:pointwise-energy-splitting} gives
  \begin{align*}
    |v_\varepsilon
      -I_{K_\varepsilon}v_\varepsilon|_
      {H^1(K_\varepsilon)}^2
    =4\int_{T_\varepsilon}
      \frac{|y|}{\varepsilon}|z|^2\,\diff s(y).
  \end{align*}
  The change of variables $z=s_\varepsilon\xi$ then yields
  \begin{align}
    |v_\varepsilon
      -I_{K_\varepsilon}v_\varepsilon|_{H^1(K_\varepsilon)}^2
    =\frac{4s_\varepsilon^5}{\varepsilon}
      \int_{T_*}
      \sqrt{(\varepsilon/s_\varepsilon)^2+|\xi|^2}
      |\xi|^2\,\diff\xi.
    \label{eq:locality-sharp-energy}
  \end{align}
  Dividing its square root by $\chi_\varepsilon^{1/2} \mathcal A_{1,\varepsilon}$ yields
  \begin{align*}
    \left(\frac{s_\varepsilon}{R}\right)^{1/2}
    \frac{
      \left(
      \displaystyle\int_{T_*}
      \sqrt{(\varepsilon/s_\varepsilon)^2+|\xi|^2}
      |\xi|^2\,\diff\xi
      \right)^{1/2}}
      {h_*\sqrt{A_*}(\mu_*+2)}.
  \end{align*}
  Here, $s_\varepsilon/R\to1$ and $\varepsilon/s_\varepsilon\to0$.  For all sufficiently small $\varepsilon$, the integrand is bounded by $\sqrt{1+|\xi|^2}\,|\xi|^2$, which is integrable on the bounded triangle $T_*$. Dominated convergence gives \eqref{eq:locality-sharp-limit}, whose limit is positive because $T_*$ has positive area. Multiplication by $\chi_\varepsilon^{1/2-\theta}\to\infty$ proves \eqref{eq:locality-subcritical-divergence}.
\end{proof}

\begin{remark}[Role of spherical locality]
  The exponent $1/2$ is sharp even when the chordal triangles are equilateral. Under uniformly chordally semi-regular refinement, Corollary~\ref{cor:refining-locality} gives $\chi_T\to1$, so the radial factor tends to one. For the same test function, the transported $L^2$ directional scale is $6s_\varepsilon^3h_*^2\sqrt{A_*}$, whereas $|w_\varepsilon-I_{T_\varepsilon}w_\varepsilon|\leq s_\varepsilon^2$ and $|K_\varepsilon|_2\leq2\pi R^2$. The spherical $L^2$ error divided by that scale therefore remains bounded as $\varepsilon\downarrow0$, so this family does not determine the optimal $L^2$ radial exponent. Optimality of the factor $\chi_T$ in the $L^2$ interpolation estimate remains open.
\end{remark}

\section{Global interpolation on exact spherical triangulations}
\label{sec:global}
We pass from the local estimates to compatible triangulations of $\Sph_R^2$. We first isolate the edge-parametrisation condition required for a conforming mapped finite-element space and then verify that it is automatic for the radial spherical maps used here. In the exactly integrated formulation, both the finite-element functions and the variational integrals live on the sphere; the chordal polyhedral surface is only a parametrisation and causes no geometric consistency error.

\subsection{Mesh notation and compatible triangulations}
For $K\in\mathscr K_h$, $T_K$ denotes the chordal core. Directions and transported quantities carrying the index $K$ are defined from $T_K$, and $h$, $H_h$, $\mu_h$, and $\vartheta_h$ denote the corresponding meshwise maxima.

To assemble the local mapped spaces conformingly, coincidence of the physical spherical edges is not by itself sufficient; the adjacent element parametrisations must induce the same polynomial trace space.

\begin{definition}[Compatible exact spherical triangulation]
  A finite collection $\mathscr K_h$ of exact spherical triangles is called a compatible exact spherical triangulation of $\Sph_R^2$ if
  \begin{align*}
    \Sph_R^2=\bigcup_{K\in\mathscr K_h}K,
  \end{align*}
  the interiors of distinct elements relative to $\Sph_R^2$ are disjoint, and for any distinct $K,K'\in\mathscr K_h$, the intersection $K\cap K'$ is either empty, a common vertex, or a common shorter great-circle edge. Each element $K$ is the radial image of the chordal triangle $T_K$ formed by its vertices, with $d_{T_K}>0$.
\end{definition}

\begin{definition}[Edge-parametrisation compatibility]
  Let two mapped elements $K^+$ and $K^-$ share a physical edge $E$, and let $\widehat e^+$ and $\widehat e^-$ be the corresponding edges of the reference triangle. We say that the element maps are edge-parametrically compatible across $E$ if there exists an affine bijection
  \begin{align*}
    R_E:\widehat e^+\longrightarrow\widehat e^-
  \end{align*}
  mapping reference endpoints that represent the same physical vertex to each other and satisfying
  \begin{align}
    F_{K^+}|_{\widehat e^+}
    =
    F_{K^-}|_{\widehat e^-}\circ R_E.
    \label{eq:general-edge-parametrisation-compatibility}
  \end{align}
  The map $R_E$ includes a possible reversal of edge orientation.
\end{definition}

\begin{remark}[Geometric coincidence versus parametric compatibility]
  For general curved element maps, equality of the physical edge images does not ensure equality of the polynomial trace spaces. Condition~\eqref{eq:general-edge-parametrisation-compatibility} gives matching traces because $R_E$ is affine. For the radial construction, this condition follows from the shared chord, as shown in Lemma~\ref{lem:compatible-traces}.
\end{remark}

Throughout the remainder of this section, $\mathscr K_h$ denotes a compatible exact spherical triangulation of $\Sph_R^2$.

For any $K\in\mathscr K_h$, we write
\begin{align*}
  \ell_{K,i}:=\ell_{T_K,i},
  \quad
  q_{K,i}:=q_{T_K,i},
  \quad i=1,2.
\end{align*}
All local transported fields and weighted norms below are understood with $T=T_K$. For $v\in H^2(\Sph_R^2)$ and $i,j=1,2$, we define
\begin{subequations} \label{eq:global-weighted-directional-quantities}
\begin{align}
\displaystyle
\mathscr E_{K,ij}(v)
  &:=
  \|\mathscr Q_{\ell_{K,i},\ell_{K,j}}v\|_{0,K;T_K},
  \\
  \mathscr G_{K,i}(v)
  &:=
  \left(
    \sum_{k=1}^2
    \|\mathscr Q_{\ell_{K,i},q_{K,k}}v\|_{0,K;T_K}^2
  \right)^{1/2},
\end{align}
\end{subequations}
and, for unit vectors $\xi,\eta\in V_{T_K}$, we set
\begin{align}
  \mathscr H_{K,\xi,\eta}(v)
  &:=
  \|\nabla_{\Sph}^2v[e_\xi,e_\eta]\|_{L^2(K)},
  \notag\\
  \mathscr P_{K,\xi}(v)
  &:=
  \|\partial_{e_\xi}^{\Sph}v\|_{L^2(K)}.
  \label{eq:global-intrinsic-directional-quantities}
\end{align}
Thus, these are precisely the elementwise quantities defined in \eqref{eq:local-directional-second-data}, \eqref{eq:local-directional-H1-data}, and \eqref{eq:intrinsic-directional-data}, with $T=T_K$.

We set
\begin{align}
  h&:=\max_{K\in\mathscr K_h}h_{T_K},
  \quad
  H_h:=\max_{K\in\mathscr K_h}H_{T_K},
  \notag\\
  \mu_h&:=\max_{K\in\mathscr K_h}\mu_{T_K},
  \quad
  \vartheta_h:=\max_{K\in\mathscr K_h}\vartheta_{T_K}.
  \label{eq:global-mesh-parameters}
\end{align}
Because $\vartheta_{T_K}=H_{T_K}/(4R)$ elementwise, these definitions give
\begin{align*}
  \vartheta_h=\frac{H_h}{4R}.
\end{align*}

For the present radial construction, the edge-parametrisation condition requires no additional mesh assumption.

\begin{lemma}[Automatic edge compatibility of radial spherical maps]
  \label{lem:compatible-traces}
  Let $K^+$ and $K^-$ share a spherical edge $E$ with endpoints $p$ and $q$, and let $\widehat e^+$ and $\widehat e^-$ be the corresponding reference edges. Let
  \begin{align*}
    R_E:\widehat e^+\longrightarrow\widehat e^-
  \end{align*}
  be the unique affine bijection that maps reference endpoints representing the same physical endpoints $p$ and $q$ to each other. Then,
  \begin{align}
    F_{K^+}|_{\widehat e^+}
    =
    F_{K^-}|_{\widehat e^-}\circ R_E.
    \label{eq:radial-edge-parametrisation-compatibility}
  \end{align}
  Thus, the general edge-parametrisation condition \eqref{eq:general-edge-parametrisation-compatibility} holds automatically. Furthermore,
  \begin{align*}
    G_{T_{K^+}}|_E
    =
    G_{T_{K^-}}|_E,
  \end{align*}
  and their common image is the chord $[p,q]$. Consequently, the traces on $E$ of the mapped nodal basis functions associated with $p$ and $q$ coincide, while the mapped basis function associated with the third vertex of either adjacent element vanishes on $E$.
\end{lemma}

\begin{proof}
  The chordal triangles $T_{K^+}$ and $T_{K^-}$ share the straight edge $[p,q]$. Since the restrictions of $\Phi_{T_{K^+}}$ and $\Phi_{T_{K^-}}$ to the corresponding reference edges are affine and have the same physical endpoint values after composition with $R_E$, one has
  \begin{align*}
    \Phi_{T_{K^+}}|_{\widehat e^+}
    =
    \Phi_{T_{K^-}}|_{\widehat e^-}\circ R_E.
  \end{align*}
  On the common chord, both curved corrections are restrictions of the same radial map
  \begin{align*}
    \rho(y):=R\frac{y}{|y|},
  \end{align*}
  so that
  \begin{align*}
    \Psi_{T_{K^+}}(y)
    =
    \rho(y)
    =
    \Psi_{T_{K^-}}(y),
    \quad y\in[p,q].
  \end{align*}
  Using $F_K=\Psi_{T_K}\circ\Phi_{T_K}$ therefore gives \eqref{eq:radial-edge-parametrisation-compatibility}.

  By Lemma~\ref{lem:radial-basic}, the radial projection restricts to a bijection from $[p,q]$ onto $E$. Therefore, for any $x\in E$, there exists a unique $y\in[p,q]$ such that
  \begin{align*}
    x=R\frac{y}{|y|}.
  \end{align*}
  Because the chord belongs to both $T_{K^+}$ and $T_{K^-}$, the inverse radial maps satisfy
  \begin{align*}
    G_{T_{K^+}}(x)
    =
    y
    =
    G_{T_{K^-}}(x).
  \end{align*}
  Thus, the two inverse radial maps agree on $E$.

  The restrictions to $[p,q]$ of the affine barycentric coordinates associated with $p$ and $q$ are the same linear barycentric coordinates on that chord, whereas the barycentric coordinate associated with the third vertex vanishes there. Composing these identities with the common inverse radial map on $E$ gives the asserted trace identities.
\end{proof}

We define
\begin{align*}
  V_h
  :=
  \left\{
    v_h\in \mathcal{C}^0(\Sph_R^2): \
    v_h|_K\in\mathbb P_F^1(K)
    \text{ for any }K\in\mathscr K_h
  \right\}.
\end{align*}
By Lemma~\ref{lem:compatible-traces}, the mapped nodal basis functions have compatible traces across any common edge. Therefore, local mapped functions with common vertex values have matching traces, and the usual global vertex numbering produces continuous mapped $\mathbb P^1$ functions. Because $d_{T_K}>0$, any function in $\mathbb P_F^1(K)$ is smooth up to the boundary of $K$ and hence belongs to $H^1(K)$. The standard patching property therefore yields
\begin{align*}
  V_h\subset H^1(\Sph_R^2).
\end{align*}

Because $H^2(\Sph_R^2)\hookrightarrow \mathcal{C}^0(\Sph_R^2)$, the nodal values of any $v\in H^2(\Sph_R^2)$ are well defined. We define the global nodal interpolant $I_h:H^2(\Sph_R^2)\to V_h$ elementwise as
\begin{align*}
  (I_hv)|_K:=I_K(v|_K),
  \quad K\in\mathscr K_h.
\end{align*}
Lemma~\ref{lem:compatible-traces} ensures that the local interpolants have matching traces on every common edge.

\subsection{Global directional anisotropic interpolation estimates}
Once conformity is established, the global anisotropic estimates follow by summing the local bounds elementwise.

\begin{theorem}[Global directional anisotropic interpolation]
  \label{thm:global-anisotropic-interpolation}
  Let $\mathscr K_h$ be a compatible exact spherical triangulation of $\Sph_R^2$. For $K\in\mathscr K_h$, we set 
\begin{subequations} \label{eq:global-directional-data}
\begin{align}
\displaystyle
\mathcal A_{0,K}(v)
    &:=\sum_{i,j=1}^2h_{T_K,i}h_{T_K,j}\mathscr E_{K,ij}(v),\\
    \mathcal A_{1,K}(v)
    &:=\mu_{T_K}h_{T_K,1}\mathscr G_{K,1}(v)
       +2h_{T_K,2}\mathscr G_{K,2}(v).
\end{align}
\end{subequations}
Then, for any $v\in H^2(\Sph_R^2)$,
  \begin{align}
    \|v-I_hv\|_{L^2(\Sph_R^2)}
    \leq C_{\mathrm{aff},0}
    \left(\sum_{K\in\mathscr K_h}
      \chi_{T_K}^2\mathcal A_{0,K}(v)^2\right)^{1/2},
    \label{eq:global-anisotropic-L2}
  \end{align}
  and
  \begin{align}
    |v-I_hv|_{H^1(\Sph_R^2)}
    \leq C_{\mathrm{aff},1}
    \left(\sum_{K\in\mathscr K_h}
      \chi_{T_K}\mathcal A_{1,K}(v)^2\right)^{1/2}.
    \label{eq:global-anisotropic-H1}
  \end{align}
\end{theorem}

No quasi-uniformity is assumed in this theorem, and its constants are independent of the number of elements and of all element aspect ratios. Neither $\mu_h$ nor $\vartheta_h$ is required to be uniformly bounded at this exact anisotropic level; the elementwise factors remain visible in \eqref{eq:global-anisotropic-L2}--\eqref{eq:global-anisotropic-H1}.

\begin{proof}
For each $K\in\mathscr K_h$, the restriction $v|_K$ belongs to $H^2(K)$, and the definition of $I_h$ gives
\begin{align*}
  (v-I_hv)|_K
  =
  v|_K-I_K(v|_K).
\end{align*}
Therefore, Theorem~\ref{thm:final-local-anisotropic} applies on any element. Because the element interiors are pairwise disjoint,
\begin{align*}
  \|v-I_hv\|_{L^2(\Sph_R^2)}^2
  =
  \sum_{K\in\mathscr K_h}
  \|v-I_hv\|_{L^2(K)}^2.
\end{align*}
Furthermore, $v-I_hv\in H^1(\Sph_R^2)$, and hence
\begin{align*}
  |v-I_hv|_{H^1(\Sph_R^2)}^2
  =
  \sum_{K\in\mathscr K_h}
  |v-I_hv|_{H^1(K)}^2.
\end{align*}
Squaring the corresponding local estimates, summing over $K\in\mathscr K_h$, and taking square roots gives \eqref{eq:global-anisotropic-L2} and \eqref{eq:global-anisotropic-H1}.
\end{proof}

For $v\in H^2(\Sph_R^2)$, we define the expanded intrinsic quantities
\begin{align}
  \mathcal B_{0,K}(v)
  :=\sum_{i,j=1}^2h_{T_K,i}h_{T_K,j}
  \left[\chi_{T_K}\mathscr H_{K,\ell_{K,i},\ell_{K,j}}(v) + \kappa_{T_K}
          (\mathscr P_{K,\ell_{K,i}}(v)
           +\mathscr P_{K,\ell_{K,j}}(v)) \right],
  \label{eq:global-expanded-L2-data}
\end{align}
and, for $i=1,2$,
\begin{align*}
  \mathcal C_{K,i}(v)
  :=\left\{\sum_{k=1}^2
  \left[\chi_{T_K}\mathscr H_{K,\ell_{K,i},q_{K,k}}(v)
        +\kappa_{T_K}
         (\mathscr P_{K,\ell_{K,i}}(v)
          +\mathscr P_{K,q_{K,k}}(v))\right]^2\right\}^{1/2},
\end{align*}
and we set
\begin{align}
  \mathcal B_{1,K}(v)
  :=\mu_{T_K}h_{T_K,1}\mathcal C_{K,1}(v)
    +2h_{T_K,2}\mathcal C_{K,2}(v).
  \label{eq:global-expanded-H1-data}
\end{align}

\begin{corollary}[Expanded global directional form]
  \label{cor:global-expanded-directional}
  Let $\mathscr K_h$ be a compatible exact spherical triangulation of $\Sph_R^2$. Then, for any $v\in H^2(\Sph_R^2)$,
  \begin{align}
    \|v-I_hv\|_{L^2(\Sph_R^2)}
    \leq C_{\mathrm{aff},0}
    \left(\sum_{K\in\mathscr K_h}
      \chi_{T_K}^2\mathcal B_{0,K}(v)^2\right)^{1/2},
    \label{eq:global-expanded-L2}
  \end{align}
  and
  \begin{align}
    |v-I_hv|_{H^1(\Sph_R^2)}
    \leq C_{\mathrm{aff},1}
    \left(\sum_{K\in\mathscr K_h}
      \chi_{T_K}\mathcal B_{1,K}(v)^2\right)^{1/2}.
    \label{eq:global-expanded-H1}
  \end{align}
\end{corollary}

\begin{proof}
  Theorem~\ref{thm:standard-directional-form}, applied elementwise, gives
  \begin{align*}
    \mathcal A_{0,K}(v)\leq\mathcal B_{0,K}(v),
    \quad
    \mathcal A_{1,K}(v)\leq\mathcal B_{1,K}(v).
  \end{align*}
  Substituting these bounds into \eqref{eq:global-anisotropic-L2} and \eqref{eq:global-anisotropic-H1} proves the result.
\end{proof}

\subsection{Standard rates and the intrinsic spherical criterion}
We first derive standard global rates under explicit chordal semi-regularity and radial-locality bounds. We then use the geometric interface of Section~\ref{subsec:geometric-interface} to replace these assumptions by the intrinsic spherical maximum-angle condition and an explicit bound on $h/R$.

\begin{corollary}[Global isotropic consequences and standard rates]
  \label{cor:global-standard-rates}
  Let $\eta_0\in(0,1)$ be fixed and suppose that $\vartheta_h\leq\eta_0$. Then, for any $v\in H^2(\Sph_R^2)$,
  \begin{align}
    \|v-I_hv\|_{L^2(\Sph_R^2)}
    \leq
    C(\eta_0)h^2
    \left(
      \|\nabla_{\Sph}^2v\|_{L^2(\Sph_R^2)}
      +
      \frac{H_h}{R^2}
      \|\nabla_{\Sph}v\|_{L^2(\Sph_R^2)}
    \right),
    \label{eq:global-standard-L2}
  \end{align}
  and
  \begin{align}
    |v-I_hv|_{H^1(\Sph_R^2)}
    \leq
    C(\eta_0)\mu_hh
    \left(
      \|\nabla_{\Sph}^2v\|_{L^2(\Sph_R^2)}
      +
      \frac{H_h}{R^2}
      \|\nabla_{\Sph}v\|_{L^2(\Sph_R^2)}
    \right).
    \label{eq:global-standard-H1}
  \end{align}
  If, in addition, $\gamma_0\geq2$ is fixed and $\mu_h\leq\gamma_0$, then $H_h\leq\gamma_0h$, and hence
  \begin{align}
    \|v-I_hv\|_{L^2(\Sph_R^2)}
    \leq
    C(\gamma_0,\eta_0)
    \left(
      h^2
      \|\nabla_{\Sph}^2v\|_{L^2(\Sph_R^2)}
      +
      \frac{h^3}{R^2}
      \|\nabla_{\Sph}v\|_{L^2(\Sph_R^2)}
    \right),
    \label{eq:global-semiregular-L2}
  \end{align}
  and
  \begin{align}
    |v-I_hv|_{H^1(\Sph_R^2)}
    \leq
    C(\gamma_0,\eta_0)
    \left(
      h
      \|\nabla_{\Sph}^2v\|_{L^2(\Sph_R^2)}
      +
      \frac{h^2}{R^2}
      \|\nabla_{\Sph}v\|_{L^2(\Sph_R^2)}
    \right).
    \label{eq:global-semiregular-H1}
  \end{align}
  The constants are independent of $h$, $R$, the number of elements, and all element aspect ratios.
\end{corollary}

\begin{proof}
For $K\in\mathscr K_h$, we set
\begin{align*}
  a_K
  &:=
  \|\nabla_{\Sph}^2v\|_{L^2(K)},
  \quad
  b_K
  :=
  \|\nabla_{\Sph}v\|_{L^2(K)}.
\end{align*}
Because $\vartheta_{T_K}\leq\vartheta_h\leq\eta_0$, the local isotropic consequences apply on any element with the same constant $C(\eta_0)$. Using $h_{T_K}\leq h$, $H_{T_K}\leq H_h$, and $\mu_{T_K}\leq\mu_h$, we obtain
\begin{align*}
  \|v-I_hv\|_{L^2(K)}
  &\leq
  C(\eta_0)h^2
  \left(
    a_K+\frac{H_h}{R^2}b_K
  \right),
  \\
  |v-I_hv|_{H^1(K)}
  &\leq
  C(\eta_0)\mu_hh
  \left(
    a_K+\frac{H_h}{R^2}b_K
  \right).
\end{align*}
After squaring and summing these inequalities over the mesh, the triangle inequality in $\ell^2(\mathscr K_h)$ gives
\begin{align*}
  \left(
    \sum_{K\in\mathscr K_h}
    \left(
      a_K+\frac{H_h}{R^2}b_K
    \right)^2
  \right)^{1/2}
  &\leq
  \left(\sum_{K\in\mathscr K_h}a_K^2\right)^{1/2}
  +
  \frac{H_h}{R^2}
  \left(\sum_{K\in\mathscr K_h}b_K^2\right)^{1/2}
  \\
  &=
  \|\nabla_{\Sph}^2v\|_{L^2(\Sph_R^2)}
  +
  \frac{H_h}{R^2}
  \|\nabla_{\Sph}v\|_{L^2(\Sph_R^2)}.
\end{align*}
Together with the exact elementwise decompositions of the global norms, this proves \eqref{eq:global-standard-L2}--\eqref{eq:global-standard-H1}.

If $\mu_h\leq\gamma_0$, then
\begin{align*}
  H_h
  =
  \max_{K\in\mathscr K_h}
  \mu_{T_K}h_{T_K}
  \leq
  \gamma_0h.
\end{align*}
Substitution into the preceding two estimates, followed by absorption of the powers of $\gamma_0$ into $C(\gamma_0,\eta_0)$, gives \eqref{eq:global-semiregular-L2}--\eqref{eq:global-semiregular-H1}.
\end{proof}

\begin{corollary}[Global rates under the intrinsic spherical maximum-angle condition]
  \label{cor:global-intrinsic-criterion}
  Let $\{\mathscr K_h\}$ be a family of compatible exact spherical triangulations
  satisfying the intrinsic spherical maximum-angle condition of
  Definition~\ref{def:intrinsic-spherical-maximum-angle}. Thus, for some
  $c_{\mathrm{sr}}\in(0,1]$,
  \begin{align}
    \inf_h\inf_{K\in\mathscr K_h}t_K^{\mathrm{sr}}
    \geq c_{\mathrm{sr}}.
    \label{eq:global-intrinsic-criterion}
  \end{align}
  Let
  \begin{align*}
    h:=\max_{K\in\mathscr K_h}h_{T_K}.
  \end{align*}
  If $h/R\leq\sqrt3\,c_{\mathrm{sr}}/2$, then there exists a constant $C=C(c_{\mathrm{sr}})$, independent of $h$, $R$, the number of elements, and the element aspect ratios, such that
  \begin{align}
    \|v-I_hv\|_{L^2(\Sph_R^2)}
    &\leq
    C\left(
      h^2\|\nabla_{\Sph}^2v\|_{L^2(\Sph_R^2)}
      +\frac{h^3}{R^2}
       \|\nabla_{\Sph}v\|_{L^2(\Sph_R^2)}
    \right),
    \label{eq:global-intrinsic-L2}
    \\
    |v-I_hv|_{H^1(\Sph_R^2)}
    &\leq
    C\left(
      h\|\nabla_{\Sph}^2v\|_{L^2(\Sph_R^2)}
      +\frac{h^2}{R^2}
       \|\nabla_{\Sph}v\|_{L^2(\Sph_R^2)}
    \right)
    \label{eq:global-intrinsic-H1}
  \end{align}
  for any $v\in H^2(\Sph_R^2)$.
\end{corollary}

\begin{proof}
  Corollary~\ref{cor:intrinsic-partII-geometric-interface} gives
  \begin{align*}
    \mu_h
    \leq
    \frac{4}{\sqrt3\,c_{\mathrm{sr}}},
    \quad
    \vartheta_h
    \leq
    \frac{h}{\sqrt3\,c_{\mathrm{sr}}R}.
  \end{align*}
  Therefore, whenever $h/R\leq\sqrt3\,c_{\mathrm{sr}}/2$, one has $\vartheta_h\leq1/2$ and may take $\eta_0=1/2$ in Corollary~\ref{cor:global-standard-rates}. Furthermore,
  \begin{align*}
    \gamma_0
    :=
    \frac{4}{\sqrt3\,c_{\mathrm{sr}}}
    \geq
    \frac4{\sqrt3}
    >2.
  \end{align*}
  Applying that corollary with this value of $\gamma_0$ proves \eqref{eq:global-intrinsic-L2}--\eqref{eq:global-intrinsic-H1}.
\end{proof}

\subsection{Anisotropic Laplace--Beltrami Galerkin application}
The Galerkin estimate inherits the interpolation hierarchy above. We first retain the fully anisotropic elementwise factors, then impose the intrinsic maximum-angle condition, and then use an explicit smallness condition on $h/R$ to obtain the standard $h$-rate.

For $v\in H^2(\Sph_R^2)$ and $K\in\mathscr K_h$, introduce the two directional scales
\begin{subequations}
  \label{eq:galerkin-unweighted-directional-scales}
\begin{align}
\displaystyle
  \mathcal U_{1,K}^{\mathrm{tr}}(v)
  &:=
  h_{T_K,1}\mathscr G_{K,1}(v)
  +h_{T_K,2}\mathscr G_{K,2}(v),
  \\
  \mathcal U_{1,K}^{\mathrm{std}}(v)
  &:=
  h_{T_K,1}\mathcal C_{K,1}(v)
  +h_{T_K,2}\mathcal C_{K,2}(v).
\end{align}
\end{subequations}
The first uses the exact transported quantities; the second uses only the expanded covariant-Hessian and tangential-gradient norms on $K$. By \eqref{eq:directional-H1-data-bound},
\begin{align*}
  \mathcal U_{1,K}^{\mathrm{tr}}(v)
  \leq
  \mathcal U_{1,K}^{\mathrm{std}}(v).
\end{align*}

\begin{corollary}[Exactly integrated anisotropic Laplace--Beltrami Galerkin error]
  \label{cor:laplace-beltrami-galerkin}
  We define the mean-zero spaces
  \begin{align*}
    H^1_\star(\Sph_R^2)
    &:={\left\{v\in H^1(\Sph_R^2): \ \int_{\Sph_R^2}v\,\diff\sigma=0\right\}},\\
    V_{h,\star}&:=V_h\cap H^1_\star(\Sph_R^2).
  \end{align*}
  Let $f\in L^2(\Sph_R^2)$ satisfy
  $\int_{\Sph_R^2}f\,\diff\sigma=0$, and let
  $u\in H^1_\star(\Sph_R^2)$ solve
  \begin{align*}
    \int_{\Sph_R^2}\nabla_{\Sph}u\cdot\nabla_{\Sph}v\,\diff\sigma
    =\int_{\Sph_R^2}fv\,\diff\sigma
    \quad\forall v\in H^1_\star(\Sph_R^2).
  \end{align*}
  Let $u_h\in V_{h,\star}$ be the conforming Galerkin approximation with the surface integrals evaluated on the exact sphere. Equivalently, $u_h$ is defined as
  \begin{align}
    \int_{\Sph_R^2}
    \nabla_{\Sph}u_h\cdot\nabla_{\Sph}v_h\,\diff\sigma
    =
    \int_{\Sph_R^2}fv_h\,\diff\sigma
    \quad
    \forall v_h\in V_{h,\star}.
    \label{eq:laplace-beltrami-discrete-problem}
  \end{align}
  Elliptic regularity on the smooth closed sphere gives $u\in H^2(\Sph_R^2)$. Then,
  \begin{align}
    |u-u_h|_{H^1(\Sph_R^2)}
    &\leq
    C_{\mathrm{aff},1}
    \left(
      \sum_{K\in\mathscr K_h}
      \chi_{T_K}\mathcal A_{1,K}(u)^2
    \right)^{1/2}
    \label{eq:laplace-beltrami-anisotropic-energy-error}
    \\
    &\leq
    C_{\mathrm{aff},1}
    \left(
      \sum_{K\in\mathscr K_h}
      \chi_{T_K}\mathcal B_{1,K}(u)^2
    \right)^{1/2}.
    \label{eq:laplace-beltrami-expanded-energy-error}
  \end{align}
  Thus, at the exact anisotropic level, no quasi-uniformity and no uniform bound on $\mu_h$ or $\vartheta_h$ is required; the directional lengths, the asymmetric maximum-angle factor and the radial-distortion factor remain visible element by element. Suppose, in addition, that for some $c_{\mathrm{sr}}\in(0,1]$,
  \begin{align*}
    \min_{K\in\mathscr K_h}t_K^{\mathrm{sr}}
    \geq
    c_{\mathrm{sr}}.
  \end{align*}
  Then,
  \begin{align}
    |u-u_h|_{H^1(\Sph_R^2)}
    &\leq
    C(c_{\mathrm{sr}})
    \left(
      \sum_{K\in\mathscr K_h}
      \chi_{T_K}
      \mathcal U_{1,K}^{\mathrm{tr}}(u)^2
    \right)^{1/2}
    \leq
    C(c_{\mathrm{sr}})
    \left(
      \sum_{K\in\mathscr K_h}
      \chi_{T_K}
      \mathcal U_{1,K}^{\mathrm{std}}(u)^2
    \right)^{1/2}.
    \label{eq:laplace-beltrami-maxangle-directional-error}
  \end{align}
  Thus, the intrinsic maximum-angle condition removes the chordal factor $\mu_{T_K}$ from the displayed directional scale, while the radial factor $\chi_{T_K}$ remains explicit. If, more specifically, $h/R\leq\sqrt3\,c_{\mathrm{sr}}/2$, then
  \begin{align}
    |u-u_h|_{H^1(\Sph_R^2)}
    &\leq
    C(c_{\mathrm{sr}})
    \left(
      \sum_{K\in\mathscr K_h}
      \mathcal U_{1,K}^{\mathrm{tr}}(u)^2
    \right)^{1/2}
    \leq
    C(c_{\mathrm{sr}})
    \left(
      \sum_{K\in\mathscr K_h}
      \mathcal U_{1,K}^{\mathrm{std}}(u)^2
    \right)^{1/2}.
    \label{eq:laplace-beltrami-refined-directional-error}
  \end{align}
  In particular, collapsing the directional scales only at this final stage gives
  \begin{align}
    |u-u_h|_{H^1(\Sph_R^2)}
    \leq
    C(c_{\mathrm{sr}})\left(
      h\|\nabla_{\Sph}^2u\|_{L^2(\Sph_R^2)}
      +\frac{h^2}{R^2}\|\nabla_{\Sph}u\|_{L^2(\Sph_R^2)}
    \right),
    \label{eq:laplace-beltrami-energy-error}
  \end{align}
  where all constants denoted by $C(c_{\mathrm{sr}})$ are independent of $h$, $R$, the number of elements, and all element aspect ratios.
\end{corollary}

\begin{proof}
  On $H^1_\star(\Sph_R^2)$, the bilinear form
  \begin{align*}
    a(v,w)
    :=\int_{\Sph_R^2}
      \nabla_{\Sph}v\cdot\nabla_{\Sph}w\,\diff\sigma
  \end{align*}
  is an inner product. The Poincar\'e inequality on the connected sphere makes this form coercive on $H^1_\star(\Sph_R^2)$, so the continuous and discrete mean-zero problems are uniquely solvable. Standard elliptic regularity on the smooth closed sphere gives $u\in H^2(\Sph_R^2)$ because $f\in L^2(\Sph_R^2)$. Galerkin orthogonality therefore gives the energy best-approximation property
  \begin{align*}
    |u-u_h|_{H^1(\Sph_R^2)}
    \leq
    \inf_{v_h\in V_{h,\star}}
    |u-v_h|_{H^1(\Sph_R^2)}.
  \end{align*}
  Because constants belong to $V_h$, the function
  \begin{align*}
    v_h:=I_hu-
    \frac{1}{|\Sph_R^2|}
    \int_{\Sph_R^2}I_hu\,\diff\sigma
  \end{align*}
  belongs to $V_{h,\star}$. Because $u-v_h$ and $u-I_hu$ differ only by a constant,
  \begin{align*}
    |u-v_h|_{H^1(\Sph_R^2)}
    =|u-I_hu|_{H^1(\Sph_R^2)}.
  \end{align*}
  Theorem~\ref{thm:global-anisotropic-interpolation} yields \eqref{eq:laplace-beltrami-anisotropic-energy-error}, and Corollary~\ref{cor:global-expanded-directional} yields \eqref{eq:laplace-beltrami-expanded-energy-error}. Under the intrinsic maximum-angle condition, the elementwise comparison in Corollary~\ref{cor:intrinsic-partII-geometric-interface} gives
  \begin{align*}
    \mu_{T_K}
    \leq
    \gamma_0
    :=
    \frac{4}{\sqrt3\,c_{\mathrm{sr}}},
    \quad K\in\mathscr K_h.
  \end{align*}
  Because $c_{\mathrm{sr}}\leq1$, one has $\gamma_0>2$, and therefore
  \begin{align*}
    \mathcal A_{1,K}(u)
    \leq
    \gamma_0\mathcal U_{1,K}^{\mathrm{tr}}(u).
  \end{align*}
  Substitution into \eqref{eq:laplace-beltrami-anisotropic-energy-error}, followed by $\mathcal U_{1,K}^{\mathrm{tr}} \leq\mathcal U_{1,K}^{\mathrm{std}}$, proves \eqref{eq:laplace-beltrami-maxangle-directional-error}. Under the additional condition $h/R\leq\sqrt3\,c_{\mathrm{sr}}/2$, the geometric comparison in Corollary~\ref{cor:global-intrinsic-criterion} gives $\vartheta_h\leq1/2$. Therefore,
  \begin{align*}
    \max_{K\in\mathscr K_h}\chi_{T_K}
    \leq
    \left(1-\frac14\right)^{-1/2}
    =
    \frac2{\sqrt3}.
  \end{align*}
  Absorbing this uniform radial factor into the constant proves \eqref{eq:laplace-beltrami-refined-directional-error}. Finally, Corollary~\ref{cor:global-intrinsic-criterion} gives the isotropic consequence \eqref{eq:laplace-beltrami-energy-error}.
\end{proof}

\section{Concluding remarks}
\label{sec:conclusion}
We have derived anisotropic $\mathbb P^1$ interpolation estimates for the standard radially mapped element on exact geodesic spherical triangles. The $H^1$ bound retains the asymmetric directional scales $\mu_T h_{T,1}$ and $2h_{T,2}$, with radial distortion represented by $\chi_T^{1/2}$. Four explicit families quantify the loss caused by replacing directional lengths with the diameter, establish the sharp order of the maximum-angle factor, show that this factor cannot be transferred from the longer to the shorter directional scale within the two-term bound, and prove the optimality of the exponent $1/2$ in the upper $H^1$ radial-transfer bound with the directional scale fixed.

The radial construction provides compatible edge parametrisations, so the local estimates extend to conforming spherical triangulations. Together with the geometric results of \cite{Ishizaka2026SphericalMaxAngle}, they give global $L^2$- and $H^1$-rates under the intrinsic maximum-angle condition and relative refinement, without a minimum-angle or shape-regularity assumption. The same directional estimates yield the energy-error bound for the conforming mean-zero Laplace--Beltrami method with exact integration.

% ------------------------------------------------------------
\section*{Acknowledgements}
% ------------------------------------------------------------
\paragraph{Use of AI tools.}
The research questions and main ideas of this work were formulated by the author. The author used ChatGPT (OpenAI) during the preparation and revision of this manuscript. This tool was used for language editing and drafting parts of the exposition, assistance with checking mathematical derivations and proofs, computational checks of selected identities and examples, and checking references. The author reviewed and verified all mathematical content and conclusions and takes full responsibility for the final manuscript.

% ------------------------------------------------------------
\section*{Statements and Declarations}
% ------------------------------------------------------------
\paragraph{Funding.}
The author declares that no funds, grants, or other support were received during
the preparation of this manuscript.

\paragraph{Data availability.}
No datasets were generated or analysed during the current study.

\paragraph{Competing interests.}
The author declares no competing interests.

\end{document}